\documentclass[12pt]{amsart}

\usepackage{amsmath,amsthm,amsfonts,amssymb,amscd,mathrsfs}
\usepackage[margin=3cm]{geometry}
\usepackage{color}

\def\dist{{\rm dist}}
\def\eps{{\varepsilon}}

\def\Card{{\rm Card}}

\def\one{{\mathbf{1}}}
\def\Prob{{\mathbb{P}}}

\def\Var{{\rm Var}}

\def\EXP{{\mathbb{E}}}

\def\BAN{\mathbb{B}}

\def\naturals{\mathbb{N}}

\def\Tor{\mathbb{T}}

\def\reals{\mathbb{R}}

\def\integers{\mathbb{Z}}

\def\bA{\mathbf{A}}

\def\bS{\mathbf{S}}

\def\bc{\mathbf{c}}

\def\bd{\mathbf{d}}

\def\bv{\mathbf{v}}

\def\bsigma{\boldsymbol{\sigma}}

\def\brA{{\bar A}}

\def\brB{{\bar B}}

\def\brC{{\bar C}}

\def\brF{{\bar F}}

\def\brH{{\bar H}}

\def\brN{{\bar N}}

\def\brc{{\bar c}}

\def\brk{{\bar k}}

\def\brn{{\bar n}}

\def\breps{{\bar\varepsilon}}

\def\brx{{\bar x}}

\def\breps{{\bar\eps}}

\def\brtheta{{\bar \theta}}

\def\cA{\mathcal{A}}

\def\cB{\mathcal{B}}

\def\cC{\mathcal{C}}

\def\cD{\mathcal{D}}

\def\cG{\mathcal{G}}

\def\cH{\mathcal{H}}

\def\cE{\mathcal{E}}

\def\cK{\mathcal{K}}

\def\cL{\mathcal{L}}

\def\tcL{\tilde{\mathcal{L}}}

\def\cM{\mathcal{M}}

\def\cO{\mathcal{O}}

\def\cQ{\mathcal{Q}}

\def\cR{\mathcal{R}}

\def\cS{\mathcal{S}}

\def\cU{\mathcal{U}}

\def\cW{\mathcal{W}}

\def\cX{\mathcal{X}}

\def\cY{\mathcal{Y}}

\def\fA{\mathfrak{A}}

\def\fh{\mathfrak{h}}

\def\hB{{\hat B}}

\def\hcB{{\hat\cB}}

\def\hC{{\hat C}}

\def\hc{{\hat c}}

\def\hg{{\hat g}}

\def\hs{{\hat s}}

\def\heps{{\hat{\eps}}}

\def\hpsi{{\hat\psi}}

\def\tA{{\tilde A}}

\def\tcA{{\tilde \cA}}

\def\tB{{\tilde B}}

\def\tE{{\tilde E}}

\def\tcM{{\tilde\cM}}

\def\tS{{\tilde S}}

\def\tcS{{\tilde \cS}}

\def\tLambda{{\tilde\Lambda}}

\def\tmu{{\tilde\mu}}

\def\tphi{{\tilde\phi}}

\def\tPhi{{\tilde{\Phi}}}

\def\tcB{{\tilde{\mathcal{B}}}}

\makeatother

\def\beq{\begin{equation}}
\def\eeq{\end{equation}}

\usepackage{xcolor,cancel}

\usepackage{ulem}
\usepackage{bbm}

\usepackage{tikz}
\usetikzlibrary{calc}
\usetikzlibrary{arrows,shapes,chains}

\usepackage{comment}
\usepackage{graphicx, graphics}
\usepackage[pdftex]{pict2e}
\usepackage{cases}
\usepackage{todonotes}
\usepackage{psfrag,pst-node,subfigure,rotating,
amsmath, bbm, amsthm, amssymb, amsthm, setspace, picture, epsfig,
amsfonts}
\usepackage{multirow}
\usepackage{indentfirst}
\usepackage[colorlinks,linkcolor=blue,anchorc
olor=green,citecolor=blue]{hyperref}
\usepackage{hyperref}
\numberwithin{equation}{section}
\newcommand{\bea}{\begin{eqnarray}}
\newcommand{\eea}{\end{eqnarray}}
\newcommand{\Bea}{\begin{eqnarray*}}
\newcommand{\Eea}{\end{eqnarray*}}

\theoremstyle{plain}

\newtheorem{Thm}{Theorem}[section]
\newtheorem{Lem}[Thm]{Lemma}
\newtheorem{Sublem}[Thm]{Sublemma}

\newtheorem{Con}[Thm]{Conjecture}
\newtheorem{Prop}[Thm]{Proposition}
\newtheorem{Cor}[Thm]{Corollary}
\newtheorem{Def}[Thm]{Definition}
\newtheorem{Claim}[Thm] {Claim}
\newtheorem{Rem}[Thm] {Remark}
\theoremstyle{remark}

\newtheorem{Que}[Thm] {Question}

\long\def\begcom#1\endcom{}

\newcommand{\Sep}{\operatorname{Sep}}
\newcommand{\wSep}{\widehat{{\rm Sep}}}

\newcounter{RomanNumber}

\def \N{{\mathbb N}}
\def\E{          \mathbb E}
\def\R{          \mathbb R}
\def\T{          \mathbb T}
\def\Z{          \mathbb Z}
\def\cO{{\mathcal{O}}}
\def\a{          \alpha}

\def\DS{\displaystyle}
\def\brr{\bar{r}}

\def\ln{\operatorname{ln}}

\def\one{{\mathbf{1}}}

\def\RR{\mathbb{R}}

\def\bA{\bar A}

\def\brx{{\bar{x}}}

\def\breps{{\bar\varepsilon}}
\def\heps{{\hat\varepsilon}}
\def\brtheta{{\bar\theta}}

\def\bS{{\mathbf{S}}}
\def\bv{{\mathbf{v}}}

\def\cA{{\mathcal A}}
\def\cD{{\mathcal D}}
\def\cG{{\mathcal G}}

\def\cL{{\mathcal L}}
\def\cU{{\mathcal U}}

\def\cR{{\mathcal R}}
\def\cB{{\mathcal B}}
\def\cH{{\mathcal H}}
\def\cE{{\mathcal E}}

\def\cM{{\mathcal M}}

\def\cR{{\mathcal R}}
\def\cS{{\mathcal S}}
\def\cW{{\mathcal W}}
\def\cX{{\mathcal X}}
\def\bcG{{\bar{\mathcal G}}}
\def\bcH{{\bar{\mathcal H}}}

\def\RR{{\mathbb R}}

\def\dd{{\mathbf{d}}}

\def\fA{\Omega}
\def\bsigma{\bar{\sigma}}
\def\bfA{\bar{\Omega}}

\def\fr{{\rho}}
\def\fs{{\mathfrak{s}}}
\def\fv{\sigma}

\def\hB{{\hat{B}}}
\definecolor{bluegray}{rgb}{0.1, 0.1, 0.6}

\definecolor{dgreen}{rgb}{0.1,0.6,0.1}

\definecolor{bluegreen}{rgb}{0.1,0.5,0.2}
\definecolor{bpurple}{rgb}{0.74,0.2,0.64}

\def\bluegray{\color{bluegray}}

\def\black{\color{black}}
\def\bgreen{\color{bluegreen}}
\def\bpurple{\color{bpurple}}

\definecolor{orange}{rgb}{0.8, 0.33, 0.0}

\usepackage{mathabx}

\def\EXP{{\mathbb{E}}}
\def\Prob{{\mathbb{P}}}

\def\eps{{\varepsilon}}
\definecolor{deepcarrotorange}{rgb}{1, 0.31, 0.0}

\title[Multiple Borel Cantelli Lemma in Dynamics]{ Multiple Borel Cantelli Lemma in dynamics and MultiLog law for recurrence.}

\author{Dmitry Dolgopyat}
\address[D.~Dolgopyat]{Department of Mathematics\\
University of Maryland, 4417 Mathematics Bldg, College Park\\
MD 20742, USA}
\email{dmitry@math.umd.edu}
\author{Bassam Fayad}
\address[B.~Fayad]{Institut de Math\'{e}matiques de Jussieu-Paris Rive Gauche\\
Paris, France}
\email{bassam.fayad@imj-prj.fr}
\author{Sixu Liu}
\address[S.~Liu]{Yau Mathematical Sciences Center\\
  Tsinghua University\\
  Beijing 100084, China}
\email{liusixu@mail.tsinghua.edu.cn}

\begin{document}

\begin{abstract}
A classical Borel Cantelli Lemma gives conditions for deciding whether an infinite number of rare events will almost surely happen. In this article, we propose an extension of Borel Cantelli Lemma to characterize the multiple occurrence of events on the same time scale.
Our results imply multiple Logarithm Laws   for recurrence and hitting times, as well as Poisson Limit Laws
for systems which are exponentially  mixing of all orders.
The applications include geodesic flows on compact negatively curved manifolds, geodesic excursions, Diophantine approximations and extreme value theory
for dynamical systems.
\end{abstract}

 \maketitle
 \tableofcontents
 \section{ Introduction} The study of rare events constitutes an important subject in probability theory.
On one hand, in many applications there are significant costs associated to
certain rare events, so one needs to know how often those events occur.
On the other hand, there are many phenomena in science which are driven by
rare events including metastability, anomalous diffusion (Levy flights) and traps
for motion in random media, to mention just a few examples.

In the independent setting there are three classical regimes. For the first two,  consider
an array $\{\fA_n^k\}_{k=1}^n$ of independent events such that $p_n=\Prob(\fA_n^k)$
does not depend on $k.$ Let $N_n$ be the number of events from the $n$-th array which
have occurred.  The first two regimes are:

(i) {\it CLT regime:} $np_n\to\infty.$ In this case $N_n$ is asymptotically normal.

(ii) {\it Poisson regime:} $n p_n\to \lambda.$ In this case $N_n$ is asymptotically Poisson with parameter
$\lambda.$

For the third, {\it Borel Cantelli regime} we consider a sequence $\{\fA_n\}$ of independent events with
different probabilities. The classical Borel Cantelli Lemma says that infinitely many $\fA_n$s occur
if and only if $\DS \sum_n \Prob(\fA_n)=\infty.$

A vast literature is devoted to extending the above classical results to the case where independence is replaced
by weak dependence. In particular, there are convenient moment conditions which imply
similar results  for {\it weakly} dependent events. One important distinction between the Poisson regime and the other two
regimes, is that the Poisson regime requires  additional geometric conditions on close-by events to extend the statement to the dependent case. Without such conditions, one can have clusters of rare events where the number of clusters has Poisson distribution while several events may occur inside each cluster. We refer the reader to \cite{A89} for a comprehensive discussion of Poisson clustering.

\medskip
\noindent{\sc  The multiple Borel Cantelli Lemma.} In the present paper, we consider a regime which is intermediate between the Poisson and Borel Cantelli. Namely we consider
a family of events $\fA^n_{\rho}$ which are nested: $\fA^n_{\rho_1}\subset\fA^n_{\rho_2}$ for $\rho_1<\rho_2$ and for large $n,$
$\DS
\Prob(\fA_\fr^n)\approx \sigma(\rho)$ for some function $\sigma(\rho)$.
Let $N^n_{\rho}$ be the number of $\fA^k_{\rho},$ $k\leq n$
which has occurred. We fix a sequence $\rho_n$ such that $n \sigma(\rho_n)\to 0$ as $n\to\infty$ and $r\in \N,$ and ask if infinitely many
events
$$ N^n_{\rho_n}=r$$
occur. Even if the events $\fA^n_{\rho}$ are independent for different $\rho,$ the variables $N^{n_1}_{\rho_{n_1}}$ and $N^{n_2}_{\rho_{n_2}}$ are strongly dependent
if $n_1$ and $n_2$ are of the same order. On the other hand if $n_2\gg n_1$ then those variables are weakly dependent since conditioned on
$N^{n_2}_{\rho_{n_2}}\neq 0$ it is very likely that all the events $\fA^{k}_{\rho_{n_2}}$ occur for $k>n_1.$ Using this, one can show under appropriate
monotonicity assumptions (see \cite{Mo76})
that $N^n_{\rho_n}=r$ infinitely often if and only if
$$ \sum_{M} \Prob(N^{2^M}_{\rho_{2^M}}=r)=\infty. $$
Under the condition $n \sigma(\rho_n)\to 0$ it follows that in the independent case
$$ \Prob(N^n_{\rho_n}=r)\approx \frac{(n \sigma(\rho_n))^r }{r!}. $$
Therefore, under independence, infinitely many $N^n_{\rho_n}=r$ occur if and only if
$$ \sum_M 2^{Mr} \sigma^{r}\left(\rho_{2^M}\right) =\infty. $$
The multiple Borel Cantelli Lemma  was extended to the dependent setting in \cite{AN}. However,  the mixing assumptions made in \cite{AN}
are quite strong requiring good symbolic dynamics which limits greatly the applicability of that result.
In the present paper we present more flexible mixing conditions for the multiple Borel Cantelli Lemma. Our
conditions are similar to the assumptions typically used to prove Poisson
limit theorems for dynamical systems. The precise statements of our abstract results will be given in
Sections \ref{ScBCMult} and \ref{SSMultExpMix}.
Here we describe sample applications to dynamics, geometry, and number theory.

\medskip
\noindent{\sc  MultiLog Law for recurrence.} Let $f$ be a map preserving a measure $\mu.$ Given two points $x, y$ let $d_n^{(r)}(x,y)$ be the $r$ closest
distance among $d(x, f^k y)$ for $1\leq k\leq n.$ In particular, $d_n^{(1)}(x, y)$ is the closest distance the orbit of $y$ comes to $x$ up to time $n.$
It is shown in \cite{Galatolo07} that for systems with superpolynomial decay {for Lipschitz observables}, for all $x$ and $\mu$-almost all $y$
$\DS \lim_{n\to\infty} \frac{|\ln d_n^{(1)}(x,y)|}{\ln n}=\frac{1}{\dd},$ where $\dd$ is the local dimension of $\mu$ at $x$ provided that it exists.

Under some additional assumptions, one can prove a dynamical Borel Cantelli Lemma which implies in particular that, if $\mu$ is smooth then for all $x$
and almost all $y$ we have
$$ \limsup_{n\to \infty} \frac{|\ln d_n^{(1)}(x,y)|-\frac{1}{\bd} \ln n}{\ln \ln n}=\frac{1}{\bd}.$$
In Section \ref{ScHit} we extend this result to $r>1,$ for systems that have multiple exponential mixing properties. For example, if $f$ is an expanding map of the circle, we shall show that for
Lebesgue almost all $x$ and $y$ we have
$$ \limsup_{n\to \infty} \frac{|\ln d_n^{(r)}(x,y)|-\ln n}{\ln \ln n}=\frac{1}{r}.$$
The smoothness assumption on the invariant measure, the Lebesgue typicality assumption on $x$ and the
hyperbolicity assumption on $f$ are all essential. Namely, if $\mu$ is an invariant Gibbs measure which is
{\bf not} conformal, $\lambda$ is the Lyapunov exponent of $\mu,$ then we show in Section \ref{ScGibbs} that
for $\mu$ almost all $x$ and $y$ and for all $r\in\N,$
$$ \limsup_{n\to\infty} \frac{|\ln d_n^{(r)}(x,y)|- \ln n}{\sqrt{2(\ln n) (\ln \ln \ln n)}}=\frac{ \sigma}{\dd\sqrt{\dd\lambda}} $$
for some $\sigma>0$ which will be given in \eqref{GibVar}.
We shall also show that there is $G_\delta$--dense set $\cH$ such that for all $x\in \cH,$ Lebesgue almost all $y$ and all $r\geq1$, we have
$$ \limsup_{n\to \infty} \frac{|\ln d_n^{(r)}(x,y)|-\ln n}{\ln \ln n}=1.$$
Finally if the expanding map is replaced by a rotation $T_\alpha$ then we have (see Theorem \ref{ThReturnsTransl} below) that for almost all $(x, y, \alpha)$ it holds that
$$ \limsup_{n\to \infty} \frac{|\ln d_n^{(r)}(x,y)|-\ln n}{\ln \ln n}=\begin{cases} 1 & \text{if  } r=1, \\
\frac{1}{2} & \text{if } r>1. \end{cases} $$

\medskip
\noindent{\sc   Records of geodesic excursions.} Consider a hyperbolic manifold $\cQ$ of dimension $d+1$ which is
not compact but has finite volume. Such manifold admits a thick-thin decomposition. Namely
$\cQ$ is a union of compact part and several cusps. {\it A cusp excursion} is a maximal time segment
such that the geodesic stays in a cusp for the whole segment.
Let
$$H^{(1)}(T)\geq H^{(2)}(T)\geq\dots H^{(r)}(T)\geq \dots $$
be the maximal heights achieved during the excursions which occur before time $T$
placed in the decreasing order. {\it Sullivan's Logarithm Law} is equivalent to saying that
for almost every geodesic
\begin{equation}
\label{ExcSul}
\limsup_{T\to\infty} \frac{H^{(1)}(T)}{\ln T}=\frac{1}{d}.
\end{equation}
The proof of \eqref{ExcSul} relies on Sullivan's Borel-Cantelli Lemma and it actually also shows that
for almost every geodesic
$$ \limsup_{T\to\infty}\frac{H^{(1)}(T)-\frac{1}{d}\ln T}{\ln \ln T}=\frac{1}{d}. $$
We obtain a multiple version of this result by showing that
for almost every geodesic
$$ \limsup_{T\to\infty}  \frac{H^{(r)}(T)-\frac{1}{d}\ln T}{\ln \ln T}=\frac{1}{rd}. $$

\medskip
\noindent{\sc   Multiple Khinchine Groshev Theorem.} Let $\psi:\R\to\R$ be a positive
function (in dimension 1 we also assume that $\psi$ is monotone).
The classical Khinchine Groshev Theorem (\cite{Gr38, Kh24, Spr79})
says that for almost all $\alpha\in \R^d$ there are infinitely many solutions to
\begin{equation}
\label{EqKG1}
 |\langle k, \alpha \rangle+m|\leq \psi(\|k\|_\infty) \text{ with } k\in \Z^d, m\in \Z
\end{equation}
if and only if
\begin{equation}
\label{EqKG2}
\sum_{r=1}^\infty r^{d-1} \psi(r)  =\infty.
\end{equation}

In particular the inequality
$$ |k|^d |\langle k, \alpha \rangle+m|\leq \frac{1}{\ln |k| (\ln \ln |k|)^s} $$
where  $|k|=\sqrt{\sum k_i^2}$, has infinitely many solutions for almost every $\alpha$ if and only if $s\leq 1.$
We now replace the above inequality by
\begin{equation}
\label{KNApprox}
 |k|^d |\langle k, \alpha \rangle+m|\leq \frac{1}{\ln N (\ln \ln N)^s}
\end{equation}
and say that $\alpha$ is $(r,s)$ approximable if there are infinitely $N$ for which \eqref{KNApprox}
has $r$ positive solutions (that is, solutions with $k_1>0$).
(Our interest in smallness of
\begin{equation}
\label{KDApprox}
 |k|^d |\langle k, \alpha \rangle+m|
\end{equation}
is motivated by \cite{DF14} where the discrepancy of Kronecker sequences with respect
to convex sets is studied. Indeed the set of $k$ where \eqref{KDApprox} is small are small denominators
of the discrepancy and they determine its growth rate.)
We show in Section \ref{ScKhinchine} that
almost every $\alpha\in \R^d$ is $(r,s)$  approximable if and only if $s\leq \frac{1}{r}.$
\smallskip

The layout of the paper is the following.
In Section \ref{ScBCMult} we describe an abstract result on an array of rare events in a probability
space which ensures that for a given $r,$ $r$ events in the same row happen for infinitely many
(respectively, finitely many) rows. In Section \ref{SSMultExpMix}
this abstract criterion is applied in the case
of rare events that consist of visits to a sublevel set of a Lipschitz function by the orbits of a smooth
exponentially mixing dynamical systems. The results of Section \ref{SSMultExpMix} are then used
to obtain MultiLog Laws in various settings. Namely, Section \ref{ScHit} studies hitting and return times for multi-fold exponentially mixing smooth systems.
 Section \ref{ScConf} treats similar problems in the configuration space
 for the geodesic flows on compact negatively curved manifolds.
 Geodesic
excursions are discussed in Section \ref{ScExcursions}, and
Diophantine approximations are treated in Section \ref{ScKhinchine}.
The MultiLog Law for non-conformal measures is discussed in Section \ref{ScGibbs}.
As it was mentioned, the regime we consider is intermediate between the Poisson and Borel-Cantelli.
Section \ref{ScPoissonHR} contains an application of our results to the Poisson regime. Namely we derive
Poisson distribution for hits and mixed Poisson distribution for returns for exponentially mixing systems
on smooth manifolds. Section \ref{ScExtreme}
describes the application of our results to the extreme value theory
for dynamical systems.
Each section ends with some  notes where the related literature is discussed.

Some useful auxiliary results are collected in the appendices.

\section{ Multiple Borel Cantelli Lemma.}
\label{ScBCMult}
\subsection{ The result.} The classical Borel Cantelli Lemma is a standard tool for deciding when an infinite number of rare events
occur with probability one. However in case an infinite number of events do occur, the Borel Cantelli Lemma does not give an information
about how well separated in time those occurrences are. In this section we present a criterion which allows to decide when several rare events
occur on the same time scale. The criterion is based on various  independence conditions between the rare events.

\begin{Def} Consider a probability space $(\bf \Omega, \mathcal{F}, \mathbb P)$. Given $r \in \N^*$ and a family of events $\{ \fA_{\rho_{n}}^k \}_{(n,k) \in \N^2; 1\leq k \leq n}$,
we let $N^n_{\fr_n}$ be the number of times $k\leq n$ such that $\fA^k_{\fr_n}$ occurs.
\end{Def}

\begin{Rem}  In all our applications it will be the case
that
\begin{equation} \label{eq21}
\fA^n_{\fr_1}\subset \fA^n_{\fr_2} \text{ if } \fr_1\leq \fr_2
\end{equation}
however, part of our results will not require this condition.
\end{Rem}

Our goal is to give a criterion that allows to tell when almost surely $N^n_{\fr_n}\geq r$ will hold for infinitely many $n$. For this, we introduce several conditions quantifying asymptotic independence between the events $\fA^k_{\fr_n}.$
 The statement of the conditions requires the existence of:
 \begin{itemize}
 \item an  increasing  function $\sigma : \R_+ \to \R_+,$
 \item  a sequence $\eps_n \to 0,$
 \item  a function $\fs: \N
   \righttoleftarrow$ such that $ \fs(n) \leq (\ln n)^2 $,
   \item a function $\hat\fs: \N
   \righttoleftarrow$ such that  $ \varepsilon n\leq \hat\fs(n) <n(1-q)/{(2r)}$ for some $0<q<1,$ and some $0<\varepsilon<(1-q)/{(2r)},$
   \end{itemize}
    for which the following holds.

  For an arbitrary $r$-tuple $0\leq k_1<k_2\dots<k_r \leq n$ we consider the {\it separation indices} \begin{align*}\Sep_n(k_1,\dots, k_r)=\text{Card}\left\{j\in \{0, \dots r-1\}: k_{j+1}-k_j\geq \fs(n)\right\}, \quad k_0:=0,\\
  \wSep_n(k_1,\dots, k_r)=\text{Card}\left\{j\in \{0, \dots r-1\}: k_{j+1}-k_j\geq \hat\fs(n)\right\}, \quad k_0:=0.\end{align*}

\begin{itemize}

\item[$(M1)_r$]
If $0\leq k_1<k_2<\dots k_r\leq n$ are such that $\Sep_n(k_1, \dots, k_r)=r$ then
$$
\fv(\fr_n)^r (1-\eps_n)\leq
  \Prob\left(\bigcap_{j=1}^r \fA^{k_j}_{\fr_n} \right)\leq \fv(\fr_n)^r (1+\eps_n).$$

\item[$(M2)_r$] There exists $K>0$ such that if $0\leq k_1<k_2<\dots k_r\leq n$ are such that $\Sep_n(k_1, \dots, k_r)=m<r$, then
$$  \Prob\left(\bigcap_{j=1}^r \fA^{k_j}_{\fr_n} \right)
\leq {\frac{K  \fv(\fr_n)^m}{(\ln n)^{100r}}}.$$
\item[$(M3)_r$]  If $0\leq k_1<k_2<\dots<k_r<l_1<l_2<\dots<l_r,$ are such that  $2^i<k_\alpha\leq 2^{i+1}, 2^j<l_\beta\leq 2^{j+1}$,  for $1\leq\alpha,\beta\leq r$, $j-i\geq b$ for some constant $b\geq 1,$ and such that
$$\wSep_{2^{i+1}}(k_1,\ldots,k_r)=r, \quad \wSep_{2^{j+1}}(l_1,\ldots,l_r)=r,
 \quad l_1-k_r\geq \hat\fs(2^{j+1}),$$
then
$$
\Prob\left(\left[\bigcap _{\alpha=1}^r \fA^{k_\alpha}_{\fr_{2^i}}\right]\bigcap
\left[ \bigcap_{\beta=1}^r \fA^{l_\beta}_{\fr_{2^j}} \right]\right)\leq \fv(\rho_{2^i})^r\fv(\rho_{2^j})^r(1+\eps_i).$$
\end{itemize}

\begin{Def} For $r\in \N^*,$
we say that  the events of the family $\{ \fA_{\rho_{n}}^k \}_{(n,k) \in \N^2; 1\leq k \leq n}$
 are {\it  $2r$--almost independent at a fixed scale}
if $(M1)_{\brr}$ and $(M2)_{\brr}$ are satisfied for every $\brr\in [1,2 r].$
We say that
 $\fA_\fr^n$ are {\it  $2r$--almost independent at all scales}
if $(M1)_{\brr},$ $(M2)_{\brr}$ are satisfied for $\brr\in[1,2r]$, and
 $(M3)_{\brr}$ is satisfied for $\brr\in [1,r].$
\end{Def}

\begin{Thm}\label{ThMultiBC}  Given a family of events $\{ \fA_{\rho_{n}}^k \}_{(n,k) \in \N^2; 1\leq k \leq n}$, define
$$ \bS_r=\sum_{j=1}^\infty \left(2^j {\fv(\fr_{2^j})}\right)^r. $$

(a) If $\bS_r<\infty,$ \eqref{eq21} holds,  and  $\{ \fA_{\rho_{n}}^k \}_{(n,k) \in \N^2; 1\leq k \leq n}$ are $2r$--almost independent at a fixed scale, then with probability 1, we have that  for large $n,$ $N^n_{\fr_n}<r.$

(b)  If $\bS_r=\infty,$ and $\{ \fA_{\rho_{n}}^k \}_{(n,k) \in \N^2; 1\leq k \leq n}$ are $2r$--almost independent at all scales
then with probability 1, there are infinitely many $n$ such that $N^n_{\fr_n}\geq r.$
\end{Thm}

Observe that since $\fr_n$ is decreasing and $\fv$ is an increasing function we have that
$$ \sum_{n=2^{j}}^{2^{j+1}-1} \fv^r(\fr_n) n^{r-1}\leq
\left(2^{j+1} {\fv(\fr_{2^j})}\right)^r\leq 2^{2r} \sum_{n=2^{j-1}}^{2^j-1} \fv^r(\fr_n) n^{r-1}$$
when $\eqref{eq21}$ holds. Hence, the convergence of $\bS_r$ is equivalent to the convergence of
$\DS \sum_{n=1}^\infty \fv^r (\fr_n) n^{r-1}. $

\begin{Rem}
An analogous statement has been obtained in \cite{AN} under different mixing conditions.
\end{Rem}

\subsection{ Estimates on a fixed scale.}
\label{SSFixedScale}

For $m \in \N$ let
\begin{align*}
\mathcal U_m&=\{(k_1,\ldots,k_r) \textrm{ such that } 2^m<k_1<k_2<\dots <k_r\leq 2^{m+1} \textrm{ and } \wSep_{2^{m+1}}(k_1,\dots k_r)=r\},\\
 \cA_m&:= \{\exists\,0<k_1<\cdots<k_r\leq2^{m+1}\,\textrm{ s.t. }  \fA^{k_\alpha}_{\rho_{2^m}}
\textrm{ happens for any } \a \in [1,r]\}, \\
\cD_m&:= \{\exists\, (k_1,\ldots,k_r)\in \cU_m \textrm{ s.t. } \fA^{k_\alpha}_{\rho_{2^{m+1}}}
\textrm{ happens for any } \a \in [1,r]\}.
\end{align*}

The goal of this section is to prove the following estimates from which it will be easy to derive Theorem \ref{ThMultiBC}.

\begin{Prop}\label{prop6}
Suppose
\begin{equation}
\label{SmallTargets}
n\fv(\fr_n)\to 0\quad\mathrm{as}\quad n\to\infty.
\end{equation}
If $\{ \fA_{\rho_{n}}^k \}_{(n,k) \in \N^2; 1\leq k \leq n}$  are $2r$--almost independent at a fixed scale, then there exists constants $C,\,c>0$  such that

\begin{equation}
\label{Am}
\Prob(\cA_{m})\leq C\left( 2^{rm} \fv(\fr_{2^{m+1}})^r+{m^{-10}} \right)
\end{equation}
\begin{equation}
\label{Dm}
 \Prob(\cD_{m})\geq c (2^{rm} \fv(\fr_{2^{m+1}})^r -m^{-10})
\end{equation}

If $\{ \fA_{\rho_{n}}^k \}_{(n,k) \in \N^2;1\leq k \leq n}$  are $2r$--almost independent at all scales,
then for $m'>m+1$ we have a sequence $\theta_m \to 0$ such that if $m'-m\geq b$ (given in $(M3)_r$)

\begin{equation}
\label{AsymInd}
 \Prob(\cD_m\cap \cD_{m'})\leq (\Prob(\cD_m)+m^{-10}) (\Prob(\cD_{m'}) +m'^{-10})(1+\theta_m)
\end{equation}
\end{Prop}

We start with some notations and a lemma. For $n \in \N^*$, for $k_1,\ldots,k_r \leq n$, define
$$ A^{k_1, \dots, k_r}_{\fr_n}:=\bigcap_{j=1}^r  \fA^{k_j}_{\fr_n} . $$
With these notations
\begin{align}
\label{dAm} \cA_m&=\bigcup_{0<k_1<k_2<\dots <k_r\leq 2^{m+1}} A^{k_1, \dots, k_r}_{\fr_{2^m}},\\ \label{dDm}
\cD_m&=\bigcup_{(k_1,\ldots,k_r) \in \cU_m}   A^{k_1, \dots, k_r}_{\fr_{2^{m+1}}}.
\end{align}

\begin{Lem}
\label{LmSum}

Fix  $0<a_1<a_2\leq 2.$
If  $(M1)_r$ and $(M2)_r$ hold then there exists  two sequences
$\delta_n\rightarrow0,$ $\eta_n\rightarrow0$ such that
  \begin{equation}
\label{EqSum}
\sum_{a_1 n<k_1<k_2<\dots <k_r\leq a_2n}
\Prob(A^{k_1, \dots, k_r}_{\fr_n})=\frac{((a_2-a_1) n \fv(\fr_n))^r}{r!}(1+\delta_n)+\eta_n\left(\ln n\right)^{-10}.
\end{equation}
For $a_2-a_1\geq\frac12,$ there exists constant $c_r$ such that
\begin{equation}
\label{WellSepHat}
\sum_{\overset{a_1 n<k_1<k_2<\dots <k_r\leq a_2 n}{ { \wSep_n}(k_1,\dots k_r)=r}} \Prob(A^{k_1, \dots, k_r}_{\fr_n})\geq
c_r {(n \fv(\fr_n))^r}.
\end{equation}

\end{Lem}

\begin{proof}
For $m\leq r$, denote
\begin{equation*}
S_m:=\sum_{\overset{a_1 n<k_1<k_2<\dots <k_r\leq a_2 n}
{ { \Sep_n}(k_1,\dots k_r)=m}} \Prob(A^{k_1, \dots, k_r}_{\fr_n}).\end{equation*}
 Note that $S_r$ includes
$\frac{n^r}{r!} (1+\delta'_n)$ terms for some sequence $\delta'_n\rightarrow0$ as $n\to\infty,$  hence $(M1)_r$ yields
\begin{equation}
\label{WellSep}
S_r=
\frac{( n \fv(\fr_n))^r}{r!}(1+\delta''_n).\end{equation}
where $\delta''_n\to 0$ as $n\to\infty.$

For $m<r,$
$S_m$ includes $O\left(n^m \fs^{r-m}(n) \right)$ terms.
Hence $(M2)_r$ gives
\begin{equation}
\label{NonWellSep}
 S_m\leq C n^m \fs^{r-m}(m)  \frac{K\fv(\fr_n)^m}{(\ln n)^{100r}} =\eta_n (n\fv(\fr_n))^m \left(\ln n\right)^{-10}
\end{equation}
for some sequence $\eta_n\rightarrow0.$
Combining \eqref{WellSep} with \eqref{NonWellSep} we obtain \eqref{EqSum}. The proof of \eqref{WellSepHat} is similar to that of \eqref{WellSep},
except that the number of terms is not anymore equivalent to $\frac{1}{r!}n^r(1+\delta'_n)$ but just larger than $\frac{1}{r!}(\frac{n}{2}-r\hat \fs(n))^r$  which is larger than $\frac{q^r}{2^rr!}n^r$, due to the hypothesis $\hat\fs(n) <n(1-q)/{(2r)}$.
\end{proof}

\begin{proof}[Proof of Proposition \ref{prop6}] First, \eqref{Am} follows directly from \eqref{dAm} and  \eqref{EqSum}.  Next,
 define
\begin{align*} I_m&=\sum_{(k_1,\dots k_r)\in \cU_m}\Prob(A^{k_1, \dots, k_r}_{\fr_{2^{m+1}}})\\
J_m&=\sum_{\overset{(k_1,\ldots,k_r)\in \cU_m}{\overset{(k'_1,\ldots,k'_r)\in \cU_m}{\{k_1,\dots, k_r\}\neq \{k'_1,\dots, k'_r\}} }}
\Prob\left(A^{k_1,\dots, k_r}_{\fr_{2^{m+1}}}\bigcap A^{k'_1,\dots, k'_r}_{\fr_{2^{m+1}}}\right). \end{align*}
From \eqref{dDm} and Bonferroni inequalities we get that
\begin{equation} \label{eIJ}
I_m-J_m\leq \Prob(\cD_m)\leq I_m
\end{equation}

Now, \eqref{WellSepHat} implies that
\begin{equation} \label{eI}
I_m\geq c_r 2^{r(m+1)} \fv(\fr_{2^{m+1}})^r.
\end{equation}
On the other hand, since
$$ A^{k_1,\dots, k_r}_{\fr_{2^{m+1}}}\bigcap A^{k'_1,\dots, k'_r}_{\fr_{2^{m+1}}}=A^{\{k_1,\dots, k_r\}\cup\{k'_1,\dots, k'_r\}}_{\fr_{2^{m+1}}}, $$
we get that
$$J_m\leq C_r  \sum_{l=r+1}^{2r} \sum_{k_1<\dots< k_l}
\Prob(A^{k_1,\dots, k_l}_{\fr_{2^{m+1}}}),$$
and \eqref{EqSum} then implies that
\begin{equation}\label{eJ} J_m\leq C_r (2^{(r+1)m} \fv(\fr_{2^{m+1}})^{r+1} +m^{-10}).
\end{equation}
Combining \eqref{eIJ}, \eqref{eI} and \eqref{eJ}, and using the assumption \eqref{SmallTargets}
we obtain \eqref{Dm}.

Finally, observe that
$$
\Prob(\cD_m\cap \cD_{m'})
 \leq    \sum_{(k_1,\ldots,k_r) \in \cU_m, (l_1,\ldots,l_r) \in \cU_{m'}}\Prob(A^{k_1,\ldots,k_r}_{\fr_{2^{m+1}}} \cap A^{l_1,\ldots,l_r}_{\fr_{2^{m'+1}}}).
$$
But since $m'>m+1$ implies that $l_1-k_r \geq \hat \fs (2^{m'+1})$,  $(M3)_r$ then yields
\begin{equation*}\label{ApplyM3}
\Prob(A^{k_1,\ldots,k_r}_{\fr_{2^{m+1}}} \cap A^{l_1,\ldots,l_r}_{\fr_{2^{m'+1}}})\leq \Prob(A^{k_1,\ldots,k_r}_{\fr_{2^{m+1}}}) \Prob(A^{l_1,\ldots,l_r}_{\fr_{2^{m'+1}}}) (1+\eps_m),
\end{equation*}
so that using $(M1)_r$ and summing over all $(k_1,\ldots,k_r) \in \cU_m, (l_1,\ldots,l_r) \in \cU_{m'}$ we get that
$$\Prob(\cD_m\cap \cD_{m'})
 \leq  I_m I_{m'} (1+\eps_m)$$
 and \eqref{AsymInd} then follows  from \eqref{eIJ}, \eqref{eI} and \eqref{eJ}.
\end{proof}

\subsection{ Convergent case. Proof of Theorem \ref{ThMultiBC} (a).}

Suppose that $\bS_r<\infty.$  Then by monotinicity of $\sigma(\rho_n)$, we have that $n \sigma(\rho_n)\to 0$.
By \eqref{Am} of Proposition~\ref{prop6} we have that $ \sum_m \Prob(\cA_{m}) < \infty$.
By Borel-Cantelli Lemma, with probability one, $\cA_m$ happen only finitely many times.
Observe that for $n \in (2^m,2^{m+1}]$, $\{N^n_{\fr_n} \geq r \} \subset \cA_m $
because $\Omega_{\fr_n}\subset\Omega_{\fr_{2^{m}}}$ for
$n \geq 2^m$ due to \eqref{eq21}.
 Hence with probability one $\{N^n_{\fr_n}\geq r\}$
happen only finitely many times.    \hfill $ \  { \square}$

\subsection{ Divergent case. Proof of Theorem \ref{ThMultiBC} (b).}
Suppose that $\bS_r=\infty.$ We give a proof under the assumption \eqref{SmallTargets}.
 The case where \eqref{SmallTargets}
 does not hold requires minimal modifications which will be explained at the end of this section.

\begin{Claim}\label{ConAlmostSure}
Let $\DS Z_n = \sum_{m=1}^{n} 1_{\cD_m}.$ Then there exists a subsequence $\{Z_{n_k}\}$  such that a.s.
$\DS \frac{Z_{n_k}}{\EXP(Z_{n_k})} \to 1.$
\end{Claim}

Since $\EXP(Z_{n}) \to \infty$, due to \eqref{Dm},
the claim implies that, almost surely, $Z_n \to\infty$. That is, with probability one infinitely many of
$\cD_m$ happen. Note that $\cD_m\subset\{N^{2^{m+1}}_{\fr_{2^{m+1}}} \geq r \},$ which
completes the proof of Theorem ~\ref{ThMultiBC} (b) in case \eqref{SmallTargets} holds.
\begin{proof}[Proof of Claim \ref{ConAlmostSure}]
We first prove that \eqref{Dm} and \eqref{AsymInd} imply that
\begin{equation*}
\label{L2One}
\frac{Z_n}{\EXP(Z_n)}\to 1 \textrm{ in } L^2,
\end{equation*}
or equivalently  that
\begin{equation}
\label{VarVanish}
\frac{\Var(Z_n)}{\EXP^2(Z_n)}\to 0.
\end{equation}
Note that
\begin{equation}
\label{VarCount}
 \Var(Z_n)=\sum_{m=1}^n \Prob(\cD_m)-\sum_{m=1}^n \Prob(\cD_m)^2
 +2\sum_{i<j} \left[\Prob(\cD_i \cap \cD_j)-
 \Prob(\cD_i) \Prob(\cD_j)\right].
\end{equation}
By \eqref{AsymInd} for each $\delta$ there exists {$m(\delta)>b$} such that
if $i\geq m(\delta),$ $j-i\geq m(\delta) $ then
\begin{equation}
\label{DeltaBound}
 \Prob(\cD_i \cap \cD_j)-\Prob(\cD_i) \Prob(\cD_j)\leq \delta \Prob(\cD_i) \Prob(\cD_j)+2i^{-10} \Prob(\cD_j)+2j^{-10} \Prob(\cD_i)+2(ij)^{-10}.
\end{equation}
Split \eqref{VarCount} into two parts:

(a) Due to \eqref{DeltaBound}, the terms where $i\geq m(\delta),$ $j-i\geq m(\delta) $ contribute at most
$$  \sum_{i\geq m(\delta),j-i\geq m(\delta)} \left[\delta\Prob(\cD_i) \Prob(\cD_j)+2i^{-10} \Prob(\cD_j)+2j^{-10} \Prob(\cD_i)+2(ij)^{-10}\right]\leq
\delta (\EXP(Z_n))^2+8\EXP(Z_n)+8.$$
(b) The terms where $i\leq m(\delta)$ or $j-i\leq m(\delta)$ contribute at most
$$[2m(\delta)+1] \sum_{j=1}^n \Prob(\cD_j)=[2m(\delta)+1] \EXP(Z_n). $$
Since $\EXP(Z_n)\to \infty$, the case (a) dominates for large $n$ giving
$$ \limsup_{n\to\infty}\frac{\Var(Z_n)}{(\EXP(Z_n))^2}\leq \delta.$$
Since $\delta$ is arbitrary, \eqref{VarVanish} follows.

Let $n_k = \inf\{ n: (\EXP(Z_n))^2 \geq k^2 \text{Var}(Z_n)\}$. Then by Chebyshev inequality
$$\Prob\left(|Z_{n_k} - \EXP (Z_{n_k})| > \delta \EXP(Z_{n_k})\right) \leq \frac{1}{\delta^2 k^2}.$$
Thus $\displaystyle \sum_{k=1}^{\infty} \Prob(|Z_{n_k} - E(Z_{n_k})| > \delta \EXP(Z_{n_k}))  \leq \sum_{k=1}^{\infty} \frac{1}{\delta^2 k^2}<\infty.$
Therefore, by Borel-Cantelli Lemma, with probability 1, for large $k$,
$|Z_{n_k} - \EXP(Z_{n_k})| <\delta \EXP(Z_{n_k}).$ Hence $\frac{Z_{n_k}}{\EXP(Z_{n_k})} \rightarrow 1\, a.s.$, as claimed.
\end{proof}

It remains to consider the case where \eqref{SmallTargets} fails. After passing to a subsequence, we  choose a decreasing sequence $\nu_n$ such that $\tilde{\sigma}(\rho_n):= \nu_n \sigma(\rho_n)$ satisfies
$\DS \lim_{n\to \infty} n\tilde{\sigma}(\rho_n)=0$ and  $\DS \sum_{j=1}^\infty (2^j\tilde{\sigma}(\rho_{2^j}))^r=\infty.$

Next, we define for each $n \in \N$ and for each $k \leq n,$ a sequence of events $\{\tilde\fA^{k}_{\rho_{n}}\}_{k\leq n}$ as follows: If $\fA^{k}_{\rho_n}$ does not occur then $\tilde\fA^{k}_{\rho_{n}}$ does not occur and,
conditionally on $\fA^{k}_{\rho_{n}}$
occurring, $\tilde\fA^{k}_{\rho_{n}}$ occurs with
probability $\nu_n$ independently of all other events (all other $\fA^{k}_{\rho_{n}}$ with different $k$ or different $n$).

 The events $\{\tilde\fA_{\rho_{n}}^k\}_{(n,k) \in \N^2; 1\leq k \leq n}$
 thus satisfy  $(M1)_r$, $(M2)_r$, and $(M3)_r$ the same way as
 the events $\{\fA_{\rho_{n}}^k\}_{(n,k) \in \N^2; 1\leq k \leq n}$, with this difference that  ${\sigma}(\rho_n)$ is now replaced with $\tilde{\sigma}(\rho_n)$.\footnote{Note that the events $\{\tilde\fA_{\rho_{n}}^k\}$ will not satisfy \eqref{eq21} even if the events  $\{\fA_{\rho_{n}}^k\}$ satisfy it, but in this part of the proof of  Theorem \ref{ThMultiBC} (b) condition \eqref{eq21} is not needed.}
  Since condition \eqref{SmallTargets} is satisfied by $\tilde{\sigma}(\rho_n)$, and since
 $\DS \sum_{j=1}^\infty (2^j\tilde{\sigma}(\rho_{2^j}))^r=\infty,$ we get that, with probability one, more than $r$ events among the events $\{\tilde\fA_{\rho_{n}}^k\}_{k\leq n}$
occurs for infinitely many $n$. By definition, this implies that with probability one, more than $r$ events among the events $\{\fA_{\rho_{n}}^k\}_{k\leq n}$
occurs for infinitely many $n$.
The proof of Theorem \ref{ThMultiBC} (b) is thus completed. \hfill $\square$

\subsection{ Prescribing some details.}
\label{SSDetails}
In the remaining part of Section \ref{ScBCMult} we describe some extensions
of Theorem \ref{ThMultiBC}(b).

Namely, we assume that $\DS \fA^n_{\rho}=\bigcup_{i=1}^p \fA^{n,i}_\rho$ and there exists a constant
$\heps>0$ such that for each $i$,
$\Prob(\fA^{n,i}_\rho)\geq \heps \Prob(\fA^{n}_\rho). $
We also assume the following extension of $(M1)_r$:
for each $(k_1,\dots, k_r)$ with $\Sep_n(k_1,\dots, k_r)=r$ and each
$(i_1,\dots, i_r)\in \{1,\dots, p\}^r$
$$\widetilde{(M1)}_r \quad \left[\prod_{j=1}^r
\Prob(\fA^{k_j, i_j}_{\fr_n})\right](1-\eps_n)\leq
  \Prob\left(\bigcap_{j=1}^r \fA^{k_j, i_j}_{\fr_n} \right)\leq \left[\prod_{j=1}^r
  \Prob(\fA^{k_j, i_j}_{\fr_n})\right](1+\eps_n);$$
and the following extension of $(M3)_r$: for each $\delta$ there is $b=b(\delta)$
such that letting $\hs(n)=\delta n$ we have that
for each $(k_1,\dots, k_r),$ $(l_1,\dots, l_r)$ with
$$\wSep_{2^{i+1}}(k_1,\ldots,k_r)=r, \quad \wSep_{2^{j+1}}(l_1,\ldots,l_r)=r,
 \quad l_1-k_r\geq \hat\fs(2^{j+1}), \quad j-i\geq b$$
and for each $(i_1, i_2, \dots, i_r),\,(j_1, j_2\dots j_r)\in \{1,\dots, p\}^r$
  $$\widetilde{(M3)}_r\quad\quad\quad\quad\quad\quad
\Prob\left(\left[\bigcap _{\alpha=1}^r \fA^{k_\alpha, i_\alpha }_{\fr_{2^i}}\right]\bigcap
\left[ \bigcap_{\beta=1}^r \fA^{l_\beta, j_\beta}_{\fr_{2^j}} \right]\right)
$$
$$\leq
\left[ \prod_{\alpha=1}^r
\Prob(\fA^{k_\alpha, i_\alpha}_{\fr_{2^i}})\right]
 \left[\prod_{\beta=1}^r\Prob(\fA^{l_\beta, j_\beta}_{\fr_{2^j}})\right] (1+\eps_i). $$

\begin{Thm}
\footnote{This result is not used in the present paper, so it can be skipped during the
first reading. In a followup work, we shall use Theorem \ref{ThMultBCDet} to obtain some analogues of the Functional Law of Iterated Logarithm
for heavy tailed random variables.}
\label{ThMultBCDet}
 If $\bS_r=\infty,$ and $\widetilde{(M1)}_k$, $(M2)_k$ as well as
$\widetilde{(M3)}_k$  for $k=1,\dots, 2r$ are satisfied, then for any $i_1, i_2\dots i_r$ and
for any intervals $I_1, I_2 \dots I_r\subset [0,1],$
with probability 1 there are infinitely many $n$ such that
for some $k_1(n), k_2(n)\dots k_r(n)$ with $\frac{k_j(n)}{n}\in I_j,$
$\fA^{k_j, i_j}_{\rho_n}$ occur.
\end{Thm}

The proof of Theorem \ref{ThMultBCDet} is similar to the proof of Theorem \ref{ThMultiBC}(b).
Without the loss of generality we may assume that $I_j$ does not contain $0.$ Then we
consider the following modification of $\cD_m$
$$
\tilde\cD_m:= \Big\{\exists 2^m<k_1<\cdots<k_r\leq2^{m+1}\,\textrm{ such that } \frac{k_\alpha}{2^{m+1}}\in I_\alpha, $$
$$ \fA^{k_\alpha, i_\alpha }_{\rho_{2^{m+1}}}
 \textrm{ happens and }
k_{\alpha+1}-k_\alpha\geq \hat \fs(2^{m+1}),\,0\leq \alpha\leq r-1\Big\}.
$$
Arguing as in Proposition \ref{prop6} we conclude that $\tilde\cD_{m_1}$ and $\tilde\cD_{m_2}$ are
asymptotically independent
(in the sense of \eqref{AsymInd}) if
$m_2>m_1+p$ and $p$ is so large that $2^{-p}\not\in I_\alpha$ for $\alpha=1, 2\dots r.$
The rest of the proof is identical to the proof of Theorem \ref{ThMultiBC}(b).

\subsection{ Poisson regime}
\begin{Thm}
\label{ThPoisson}
Suppose that $(M1)_r$ and $(M2)_r$ hold for all $r$ and that $\DS \lim_{n\to\infty} {n \sigma(\rho_n)}=\lambda.$
Then $N^n_{\rho_n}$ converges in law as $n\to\infty$ to the Poisson distribution with parameter $\lambda.$
\end{Thm}
\begin{proof}We compute all (factorial) moments of the limiting distribution.
Let $\cX$ denote the Poisson random variable with parameter $\lambda.$
Below $\DS \left(\begin{array}{c} m \\ r\end{array}\right)$ denotes the binomial coefficient
$\DS \frac{m!}{r! (m-r)!}.$
Since (see e.g. \cite{Ross} formula (3.4) in section 7.3)
\begin{equation*}
\EXP\left(\left(\begin{array}{c} N^n_{\rho_n}\\ r\end{array}\right)\right)=
\sum_{k_1<k_2<\dots <k_r\leq n}
\Prob(A^{k_1, \dots, k_r}_{\rho_n}),
\end{equation*}
Lemma \ref{LmSum} implies for each $r$
\begin{equation}
\label{FactMom}
 \lim_{n\to \infty} \EXP\left(\left(\begin{array}{c} N^n_{\rho_n} \\ r\end{array}\right)\right)=
\frac{\lambda^r}{r!}=\EXP\left(\left(\begin{array}{c} \cX \\ r\end{array}\right)\right).
\end{equation}
Since this holds for all $r$ we also have that for all $r,$
$\DS \lim_{n\to\infty} \EXP((N^n_{\rho_n})^r)=\EXP(\cX^r).$
Since the Poisson distribution is uniquely determined by its moments the result follows.
\end{proof}

Similarly to Borel-Cantelli Lemma, we also have the following extension of Theorem~\ref{ThPoisson}
in the setting of \S \ref{SSDetails}. Denote $N^{n,i}_I$
the number of times event
$\fA^{k, i}_{\fr_n}$ occurs with $k/n\in I.$ Write $N^{n,i}:=N^{n,i}_{[0,1]}.$

\begin{Thm}
\label{ThPoissonProcess}
Suppose that $\widetilde{(M1)}_r$ and $(M2)_r$ hold for all $r$ and that $$ \lim_{n\to\infty}
{n\Prob(\Omega_{\fr_n}^{n,i})}=\lambda_i.$$
Then $\{N^{n,i}_{\rho_n}\}_{i=1}^p$
converge in law as $n\to\infty$ to the independent Poisson random variables
with parameter $\lambda_i.$

Moreover if $I_1, I_2, \dots I_s$ are disjoint intervals then
$\{N^{n, i}_{I_j}\},\; i=1\dots p, \;j=1\dots s$
converge in law as $n\to\infty$ to the independent Poisson random variables
with parameter $\lambda_i |I_j|.$
\end{Thm}

\begin{proof}
It suffices to prove the second statement. The proof is similar to the proof of Theorem
\ref{ThPoisson}. Namely, similarly to \eqref{FactMom} we show that for each set $r_{ij}\in \naturals$ we
have
\begin{equation*}
 \lim_{n\to \infty} \EXP\left(\prod_{i,j} \left(\begin{array}{c}  N^{n, i}_{I_j} \\ r_{ij}\end{array}\right)\right)=\prod_{i,j}
\frac{(\lambda_i |I_j|)^{r_{ij}}}{(r_{ij})!}=\prod_{i,j} \EXP\left(\left(\begin{array}{c} \cX_{ij} \\ r_{ij} \end{array}\right)\right)
\end{equation*}
where $\cX_{ij}$ are independent Poisson random variables with parameters $\lambda_i |I_j|.$
\end{proof}

\subsection{ Notes.}
The usual Borel Cantelli Lemma is a classical subject in probability.
There are many extensions to weakly dependent random variables, see
e.g. \cite[\S 12.15]{Wil91},  \cite[\S 1]{Sul}. The connection between Borel-Cantelli Lemma
and Poisson Limit Theorem is discussed in \cite{DF65, Fr73}.
The multiple Borel Cantelli Lemma for independent events is proven in
\cite{Mo76}. \cite{AN} obtains multiple Borel Cantelli Lemma for systems admitting good symbolic dynamics.
Extending multiple Borel Cantelli Lemma for more general sequences allows
to obtain many new applications, see Sections \ref{ScHit}--\ref{ScExtreme} of this paper.
We note that separation conditions similar to our have been used in \cite{DenkerKan07,Sevastjanov72}
to obtain the Poisson Law.

\section{ Multiple Borel Cantelli Lemma for exponentially mixing dynamical systems.}
\label{SSMultExpMix}
\subsection{ Good maps, good targets.}
 Let $f$ be a transformation of a metric space $X$ preserving a measure $\mu.$
Given a family of sets $\fA_\fr \subset X,$ $\rho \in \R_+^*$, we will, in a slight abuse of notations,
sometimes call $\fA_\fr$ the event $1_{\fA_\fr}$ and $\fA^k_\fr$ the event $1_{\fA_\fr}\circ f^k$. We will take $\fv(\fr)=\mu(\fA_\fr)$.

To deal with multiple recurrence and not just multiple hitting of targets, we need to consider slightly more complicated events.

Given a family of events ${\bfA}_{\fr}$  in $X\times X,$
let ${\bfA}^k_{\fr}\subset X$ be the event $$ {\bfA}^k_{\fr}=\{x: (x, f^kx)\in{\bfA}_{\fr}\}.$$
We will take $\bar\fv(\rho)=(\mu\times \mu)(\bfA_\fr)$.

{From now on we will always assume that if $\rho'\leq \rho$, then
$$\fA_{\fr'}\subset \fA_\fr, \quad \bfA_{\fr'}\subset \bfA_\fr.$$
}
For $\phi:X^k\rightarrow\mathbb{R},$ $k\in\mathbb{N}^+,$ we denote
$$\mu^k(\phi)=\int_{X^k} \phi(x_1,\cdots,x_k)d\mu(x_1)\cdots d\mu(x_k).$$

Given a sequence $\{\rho_n\}$, we recall that $N^n_{\fr_n}$  denotes the number of times $k\leq n$ such that $\fA^k_{\fr_n}$ (or $\bfA^k_{\fr_n}$) occurs.
We want to give conditions on the system $(f,X,\mu)$ and on the family $\{ \fA_{\rho_{n}}^k \}_{(n,k) \in \N^2; 1\leq k \leq n}$ or $\{ \bfA_{\rho_{n}}^k \}_{(n,k) \in \N^2; 1\leq k \leq n}$, that imply the validity of the dichotomy of Theorem \ref{ThMultiBC} for the number of hits $N^n_{\rho_n}$. For this, we take
$$ \bS_r=\sum_{j=1}^\infty \left(2^j \bv_j\right)^r$$
where $\bv_j=\fv(\fr_{2^j})$ if we are considering  targets of the type $\fA^k_\fr$ and
${\bv}_j=\bar\fv(\fr_{2^j})$  if we are considering targets of the type  $\bfA^k_\fr$.

The independence conditions $(M1)_r$, $(M2)_r$, $(M3)_r$ will be satisfied due to mixing conditions on the dynamical system $(f,X,\mu)$, and to some regularity and shrinking conditions on the targets that we now state.

\begin{Def}[$(r+1)$-fold exponentially mixing systems for $r\geq 1$]\label{def.mixing}
Let $\BAN$ be  a space of real
valued functions  defined over $X^{r+1}$, with a norm $\|\cdot\|_\BAN$. For $r \geq 1$, we say that $(f,X,\mu,\BAN)$ is $(r+1)$-fold exponentially mixing,  if there exist constants $C>0, L>0$ and $\theta<1$ such that
\begin{itemize}
\item[(Prod)]$\displaystyle{ \|A_1 A_2\|_\BAN \leq C \|A_1\|_\BAN \|A_2\|_\BAN,}$
\item[(Gr)]$\displaystyle{  \|A\circ (f^{k_0},\ldots,f^{k_{r}}) \|_\BAN \leq C L^{\sum_{i=0}^{r} k_i} \|A\|_\BAN,}$
\item[$({\rm EM})_r$] If $0=k_0\leq k_1\leq \ldots \leq k_r$ are such that  $\forall j\in [0,r-1], k_{j+1}-k_j \geq m$, then
 $\displaystyle{\left|\int_{X} A(x,f^{k_1}x,\cdots,f^{k_r}x) d\mu(x)-\int_{X^{r+1}}  A(x_0,\cdots,x_r)d\mu(x_0)\cdots d\mu(x_r)\right|\leq C \theta^m \left\|A\right\|_{\BAN}.}
$

\end{itemize}
\end{Def}

Given a  system $(f,X,\mu,\BAN)$, we now define the notion of simple admissible targets for $f$.

\begin{Def}[Simple admissible targets]\label{def.targets} Let $\fA_\fr$, $\rho\in \R_+^*$, be a decreasing collection of sets in $X$
 for which there are positive $\eta, \tau$ such that for all
  sufficiently small $\rho>0$
\begin{itemize}
\item[(Appr)] There are  functions $A_\fr^-, A_\fr^+:X\rightarrow \mathbb{R}$ such that $A_\fr^\pm\in \BAN$ and
\begin{itemize}
\item[(i)] $\|A_\fr^\pm\|_\infty\leq 2$
and $\|A_\fr^\pm\|_\BAN\leq \fr^{-\tau};$
\item[(ii)] $A_\fr^-\leq 1_{\fA_\fr}\leq A_\fr^+;$
\item[(iii)] $\mu(A_\fr^+)-\mu(A_\fr^-)\leq  \fv(\fr)^{1+\eta},$
\end{itemize}
\end{itemize}
where $\sigma(\rho)=\mu(\Omega_\rho)$.

 Let $\{\rho_n\}$ be a decreasing  sequence of positive numbers.
 We say that the sequence  $\{\fA_{\fr_n}\}$ is a simple  admissible sequence of targets for $(f,X,\mu,\BAN)$ if  there exists $u>0$  such that

\begin{equation}\label{rhobound}\tag{Poly} \rho_n \geq n^{-u}, \quad \sigma(\fr_n)\geq n^{-u},\end{equation}
and

\begin{equation} \tag{Mov} \label{eq.MOV}
 \forall R,  \exists \brC:\;
\forall k \in (0,R \ln n),  \;\;\; \mu(\fA_{\fr_n} \cap f^{-k} \fA_{\fr_n})\leq \brC \fv(\fr_n) (\ln n)^{-1000r}.\end{equation}
\end{Def}




\begin{Rem} {\rm
\label{RemL1}
Note that properties (Appr)(ii) and (iii) imply that
$$  \mu(A_\fr^+)-\mu(\fA_\fr)\leq  \mu(\fA_\fr)^{1+\eta},
\quad
\mu(\fA_\fr)- \mu(A_\fr^-)\leq \mu(\fA_\fr)^{1+\eta}.
$$}
\end{Rem}

{A useful situation where one can verify these properties is the following.
\begin{Lem}\label{RemLip} Suppose that
$f$ is Lipschitz and $\BAN$ is the space of Lipschitz functions. We have that
(Prod) and (Gr) hold with $L$ being the Lipschitz constant of $f.$
Moreover, if there exist constants $\xi,\xi'>0$ and  $\Phi:X \to \R$ a (uniformly) Lipschitz function\footnote{The typical situation for using Lemma \ref{RemLip} will be with $\Phi(x)$ defined by some distance $d(x_0,x)$.} such that for any interval $J \in \R$,
$$\mu(\{x: \Phi(x)\in J\}) \in [|J|^\xi,|J|^{\xi'}]$$  and two
(uniformly) Lipschitz functions   $a_1$ and $a_2: \R \to \R$  such that for some $\a,\a'>0$ we have
$$a_2(\rho)-a_1(\rho)\in [\rho^\alpha,\rho^{\alpha'}]$$
then (Appr) holds for the targets
$$ \fA_\fr=\left\{\Phi(x)\in [a_1(\fr) , a_2(\fr)]\right\}.$$
 {The same result holds if $\BAN$ is the space of $C^s$ functions or the space of compactly supported $C^s$ functions with $s>0$ arbitrary. }
\end{Lem}
The proof of Lemma \ref{RemLip} relies on simple approximation of characteristic functions by Lipschitz functions.

\begin{proof} We will construct $A^+_\rho$ that satisfies  $(i)$, $(ii)$ and $(iii)$ of (Appr), with $\mu(A_\fr^-)$ replaced by $\mu(\fA_\fr)$. The construction of $A^-_\rho$ is similar. Note that $\sigma(\rho)=\mu(\fA_\fr) \in [\rho^{\a\xi},\rho^{\a'\xi'}].$

Define a family of smooth function $\psi^+:\R^4 \to [0,2]$ such that for $v>u$ and $\eps>0$ and $x\in \R$ (we are not interested in the form of $\psi^+$ outside this domain) we have
$$\psi^+(u,v,\eps,x)=\left\{
\begin{array}{rl}
 1,  &\text{ for } x\in [u,v]\\
 0,   &\text{ for } x\notin [u-\eps(v-u),v+\eps(v-u)]
\end{array}
\right.$$
and for which there exist constants  $\eta>0$ and  $C>0$ such that that for any $\nu_0$ and for $\R^4 \supset \cR_{\nu_0}:=\{v-u\geq \nu_0, \eps\geq \nu_0\}$, we have that
$$\|\psi^+\|_{C^1(\cR_{\nu_0})}\leq C \nu_0^{-\eta},$$
where $C^1(\cR_{\nu_0})$ refers to the $C^1$ norm in the region $\cR_{\nu_0}$.

Define now $A_\fr^+ : X \to \R :  x \mapsto \psi^+(a_1(\rho),a_2(\rho),\rho^b,\Phi(x))$, where $b>1$ will be chosen later. It is clear that $ A_\fr^+$ is Lipschitz and that $1_{\fA_\fr}\leq A_\fr^+$. On the other hand $\|A_\fr^+\|_\infty\leq 2$ and $\|A_\fr^+\|_\BAN \leq C(\Phi)\rho^{-b\alpha\eta}$, and $(i)$ holds for $\tau=b\a\eta+1$. We turn now to $(iii)$. We  observe that with  $J_1=[a_1(\fr)-\rho^b(a_2(\rho)-a_1(\rho)), a_1(\fr)]$ and $J_2= [a_2(\rho),a_2(\fr)+\rho^b(a_2(\rho)-a_1(\rho))]$,
\begin{align*} \mu(A_\fr^+)-\mu(1_{\fA_\fr})&\leq 2 \mu\left(\{\Phi(x)\in J_1 \cup J_2 \}\right)\\
&\leq 4\rho^{\xi'(b+\a')}.\end{align*}
Hence, if $b$ is chosen sufficiently large we have $\rho>0$ sufficiently small that $\mu(A_\fr^+)-\mu(1_{\fA_\fr})\leq \sigma(\rho)^2$.

The fact that the same results hold if $\BAN$ is the space of $C^s$ functions or the space of compactly supported $C^s$ functions with $s>0$ arbitrary, is a simple consequence of the approximation of Lipschitz functions by smooth functions.
\end{proof}
}

To deal with recurrence, the following definition is useful.

\begin{Def}[Composite admissible targets] \label{def.recurrent.targets}
Let ${\bfA}_{\fr}$  be a decreasing collection of sets in $X\times X$ satisfying the following conditions for some positive constants $\bar C, \eta, \tau$ and for all
  sufficiently small $\rho>0,$

\begin{itemize}
\item[${\rm(\overline{Appr})}$] There are  functions $\bar{A}_\fr^-, \bar{A}_\fr^+:X\times X\rightarrow \mathbb{R}$ such that $\bar{A}_\fr^\pm\in \BAN$ and
\begin{itemize}
\item[(i)] $\|\bar{A}_\fr^\pm\|_\infty\leq 2$
and $\|\bar{A}_\fr^\pm\|_\BAN\leq  \fr^{-\tau};$
\item[(ii)] $\bar{A}_\fr^-\leq 1_{{\bfA}_{\fr}}\leq \bar{A}_\fr^+;$
\item[(iii)] For any fixed $x,$
$$\bsigma(\fr)-\bsigma(\fr)^{1+\eta}\leq
\int\bar{A}_\fr^-(x,y)d\mu(y)\leq\int\bar{A}_\fr^+(x,y)d\mu(y)\leq \bsigma(\fr)+ \bsigma(\fr)^{1+\eta},$$

\item[(iv)] For any fixed $y,$
$\DS 
\int\bar{A}_\fr^+(x,y)d\mu(x)\leq \brC \bsigma(\fr).$
\end{itemize}
\end{itemize}
The sequence $\bfA_{\fr_n}$ is said to be composite admissible  if
\begin{equation}\label{brhobound}\tag{$\overline{\mathrm{Poly}}$} \rho_n \geq n^{-u}, \quad\bsigma({\fr_n})\geq  n^{-u},\end{equation}
and there is a constant $a>0$ such that for any $k_1<k_2$
 \begin{equation}\label{eq.bSUB} \tag{${\rm \overline{Sub}}$} \bfA_{\fr}^{k_1}\cap\bfA_{\fr}^{k_2}\subset f^{-k_1}\bfA_{a\fr}^{k_2-k_1},\end{equation}
and
   \begin{equation}\label{eq.bMOV} \tag{${\rm \overline{Mov}}$} \forall k \neq 0,\quad \mu(
 \bfA^k_{a\fr_n})\leq \brC (\ln n)^{-1000r}.
 \end{equation}

\end{Def}
Observe that integrating condition ${\rm (\overline{Appr})(iii)}$ with respect to $x$ we obtain
for each $n \neq 0,$
\begin{equation} \label{eq.app} \brC^{-1}\mu\left({\bfA}^n_{\fr}\right)\leq
\mu\left(\bar{A}_\fr^-(x,f^nx)\right)\leq\mu\left(\bar{A}_\fr^+(x,f^nx)\right)\leq \brC\mu\left({\bfA}^n_{\fr}\right).\end{equation}

{The typical composite targets we will deal with are of the type $d(x,y)<\rho$ or $d(x,y)< \gamma(x)\rho,$ where $\gamma(x)$ is related to the local dimension of a smooth measure at the point $x$.  We state here a general Lemma that guarantees the admissibility of such targets. The statement is a bit technical but if we keep in mind that the function $\Phi(x,y)$ is usually defined by a distance, then the hypothesis of the Lemma become natural. The proof of the Lemma is very simple and follows a similar scheme of the proof of   Lemma \ref{RemLip} for simple targets.

\begin{Lem}\label{RemLipcomposite} Suppose that
$f$ is Lipschitz and $\BAN$ is the space of Lipschitz functions. Suppose there exists constants $C,\xi,\xi',\xi''>0$ and  $\Phi:X\times X \to \R$ a (uniformly) Lipschitz function such that
\begin{itemize}
\item[(h1)] $\DS \forall (x,y) \in X \times X, \quad \Phi(x,y) \leq C \Phi(y,x).$
\item[(h2)] For any interval $J \in \R,$
$\DS \bar \sigma(J):= (\mu\times \mu)\left(\{(x,y)\in X \times X: \Phi(x,y)\in J\}\right) \in [|J|^\xi,|J|^{\xi'}].$
\item[(h3)] For any $x\in X$, $\DS\mu\left( \left\{ y \in X : \Phi(x,y)\in J\right\}\right)  =\bar \sigma(J) (1+\cO(|J|^{\xi''})).$
\end{itemize}
If two
(uniformly) Lipschitz functions   $a_1$ and $a_2: \R \to \R$  are such that for some $\a,\a'>0$
$$a_2(\rho)-a_1(\rho)\in [\rho^\alpha,\rho^{\alpha'}]$$
then $\overline{\rm (Appr)}$ holds for the targets
$$ \bfA_\fr=\{\Phi(x,y)\in [a_1(\fr) , a_2(\fr)]\}$$
 {The same result holds if $\BAN$ is the space of $C^s$ functions or the space of compactly supported $C^s$ functions with $s>0$ arbitrary. }
\end{Lem}

\begin{proof} The proof is very similar to that of Lemma \ref{RemLip}.  We just explain the differences. Note that
 $\bsigma(\fr)=(\mu \times \mu) ({\bfA}_{\fr}) \in [\rho^{\a\xi},\rho^{\a'\xi'}]$.

We introduce
$\bar{A}_\fr^+: X\times X \to \R :  (x,y) \mapsto \psi^+(a_1(\rho),a_2(\rho),\rho^b,\Phi(x,y))$, where $\psi^+$ is as in the proof of Lemma \ref{RemLip}. Properties $(i)$ and $(ii)$ hold as in the proof of
Lemma~\ref{RemLip}.

We turn now to $(iii)$. We fix $x\in X$, and observe that with $I=[a_1(\rho),a_2(\rho)]$ and $J_1=[a_1(\fr)-\rho^b(a_2(\rho)-a_1(\rho)), a_1(\fr)]$, $J_2= [a_2(\rho),a_2(\fr)+\rho^b(a_2(\rho)-a_1(\rho))]$ we have that
\begin{align*}
\int\bar{A}_\fr^+(x,y)d\mu(y)-
 \mu\left( \left\{ y \in X : \Phi(x,y)\in I \right\}\right)&\leq   2\mu\left( \left\{ y \in X : \Phi(x,y)\in J_1 \cup J_2\right\}\right)\\
&\leq \bar \sigma(\rho)^2   \end{align*}
if $b$ is sufficiently large due to $(h2)$ and $(h3)$.  Applying $(h2)$ and $(h3)$, we also see that
$$|\mu\left( \left\{ y \in X : \Phi(x,y)\in I \right\}\right)-\bar \sigma(\rho)| =\cO(\bar \sigma(\rho)^{1+\eta})$$
for some $\eta>0$. This proves  $\overline{\rm (Appr)}(iii)$.

Finally, fix $y \in X$ and observe that $(h1)$ implies
$$ \int\bar{A}_\fr^+(x,y)d\mu(x)\leq C \int\bar{A}_\fr^+(y,x)d\mu(x)\leq 2C \bar \sigma(\rho),$$
which proves $\overline{\rm (Appr)}(iv)$.
\end{proof}
}

\subsection{ Multiple Borel-Cantelli Lemma for admissible targets.}

The goal of this section is to establish the following Theorem that gives  conditions on the system $(f,X,\mu)$ and on the family $\{ \fA_{\rho_{n}}^k \}_{(n,k) \in \N^2; 1\leq k \leq n}$ (or $\{ \bfA_{\rho_{n}}^k \}_{(n,k) \in \N^2; 1\leq k \leq n}$), that imply the validity of the dichotomy of Theorem \ref{ThMultiBC} for the number of hits $N^n_{\rho_n}$. Recall that
$$ \bS_r=\sum_{j=1}^\infty \left(2^j \bv_j\right)^r$$
where $\bv_j=\fv(\fr_{2^j})$ if we are considering  targets of the type $\fA^k_\fr$ and
${\bv}_j=\bar\fv(\fr_{2^j})$  if we are considering targets of the type  $\bfA^k_\fr$.

\begin{Thm} \label{theo.mixing} Assume a system $(f,X,\mu,\BAN)$ is {$(2r+1)$}-fold exponentially mixing.\footnote{Part $a)$ holds for  {$2r$}-fold exponentially mixing systems, as shown by the first part of Proposition \ref{ThEM-BC}.} Then

\noindent a) If $\{\fA_{\fr_n}\}$ is a sequence of simple admissible targets as in Definition \ref{def.targets},  { then  the events of the family $\{ \fA_{\rho_{n}}^k \}_{(n,k) \in \N^2; 1\leq k \leq n}$ are $2r$--almost independent at all scales.}

\noindent b) If $\{\bfA_{\fr_n}\}$  is a sequence of composite admissible targets as in Definition \ref{def.recurrent.targets}, {then  the events of the family $\{ \bfA_{\rho_{n}}^k \}_{(n,k) \in \N^2; 1\leq k \leq n}$ are $2r$--almost independent at all scales.} 
\end{Thm}

Hence, Theorem \ref{ThMultiBC} implies
\begin{Cor} \label{cor.mixing} If the system $(f,X,\mu,\BAN)$ is $(2r+1)$-fold exponentially mixing, and if
 $\{\fA_{\fr_n}\}$ (or $\{\bfA_{\fr_n}\}$) are as in Definition \ref{def.targets} (or Definition \ref{def.recurrent.targets}), then
 \begin{itemize}
\item[(a)] If $\bS_r<\infty,$  then with probability 1, we have that  for large $n$
$N^n_{\fr_n}<r.$
\item[(b)] If $\bS_r=\infty,$  then with probability 1, there are infinitely many $n$ such that
$N^n_{\fr_n}\geq r.$
\end{itemize}
\end{Cor}

In fact Theorem \ref{theo.mixing} is a direct consequence of the following Proposition. We accept a convention that $({\rm EM})_{k}$
for $k\leq 0$ is an always satisfied.

\begin{Prop}\label{ThEM-BC}
Given a dynamical system $(f,X,\mu,\BAN)$ and a sequence of  decreasing sets $\{\fA_{\rho_n}\}$ such that
 $({\rm Prod}),$ $({\rm Poly}),$  and $({\rm Appr})$ hold, then with the function  $\fv(\cdot):=\mu(\fA_\cdot)$, and

(i)  If  $({\rm EM})_{r-1}$ holds, then $(M1)_r$ is satisfied with  the function $\fs: \N \righttoleftarrow:  \fs(n)=R \ln n$, where $R$ is sufficiently large (depending on $r$, the system and the targets).

(ii) If $({\rm Gr})$, $({\rm Mov})$ and   $({\rm EM})_{r-2}$  hold, then $(M2)_r$ is satisfied.

(iii) If $({\rm Gr})$ and $({\rm EM})_r$ hold, then for arbitrary $\eps>0,$  ${(M3)}_r$ is satisfied
with $\hat\fs(n)=\eps n.$

Similarly,
given a dynamical system $(f,X,\mu,\BAN)$ and a sequence of  decreasing sets $\{\bfA_{\rho_n}\}$ such that
$(\rm Prod),$ $(\overline{\rm Poly})$ and $(\overline{\rm Appr})$ hold, then,  with the function  $\bar \fv(\cdot):=\mu\times \mu(\bfA_\cdot)$:

(i) If $({\rm EM})_r$ holds, then $(M1)_r$ is satisfied with the function $\fs: \N \righttoleftarrow:  \fs(n)=R \ln n$, with $R$ sufficiently large (depending on $r$, the system and the targets).

(ii) If $({\rm Gr})$, $(\overline{\rm Mov}),$ $(\overline{\rm Sub})$ and $({\rm EM})_{r-1}$  hold, then $(M2)_r$ is satisfied.

(iii) If $({\rm Gr})$ and $({\rm EM})_r$ hold, then for arbitrary $\eps>0$
${(M3)}_r$ is satisfied
with $\hat\fs(n)=\eps n.$


\end{Prop}

\begin{proof}[Proof of Proposition \ref{ThEM-BC}]  We use $C$ to denote a constant that may change from line to line but that will not depend on $\rho_n$, $\fA_{\fr_n}$, $\bfA_{\fr_n},$ the order of iteration of $f$, etc. 

\medskip

\noindent {\it Proof of (i)} For $\fA_{\fr_n},$ we prove $(M1)_r$ in case $k_{i+1}-k_i\geq \sqrt{R}\ln n,$ where $R$ is a sufficiently
large constant. Indeed, using ${\rm({Appr})}$ and  ${\rm {(EM)}_{r-1}}$ we get
$$ \mu\left(\prod_{i=1}^r 1_{\fA_{\fr_n}}(f^{k_i} x)\right)\leq
\mu\left(\prod_{i=1}^r A_{\fr_n}^+ (f^{k_i} x)\right)\leq
\prod_{i=1}^r \mu\left(A_{\fr_n}^+\right)+C \fr_n^{-r\tau} \theta^{\sqrt{R}\ln n} $$
$$ \leq
\left(\mu(\fA_{\fr_n})+C \mu(\fA_{\fr_n})^{1+\eta}\right)^r+C \fr_n^{-r\tau}   \theta^{\sqrt{R}\ln n},
$$
which yields the RHS of $(M1)_r$, due to ${\rm(Poly)}$ if  $R$ is sufficiently large.  The LHS is proved similarly.

For $\bfA_{\fr_n},$ we approximate $1_{\bar{\fA}_{\fr_n}}$ by $\bar A_{\fr_n}^\pm,$ apply ${\rm(\overline{Appr})},$ ${\rm {(EM)}_r}$ to the functions
$$B_{\fr_n}^+(x_0,\cdots,x_r)=\bar{A}_{\fr_n}^+(x_0,x_1)\cdots \bar{A}_{\fr_n}^+(x_0,x_r),$$
$$B_{\fr_n}^-(x_0,\cdots,x_r)=\bar{A}_{\fr_n}^-(x_0,x_1)\cdots \bar{A}_{\fr_n}^-(x_0,x_r),$$
and get
\begin{align*}
&\mu\left(\bigcap_{j=1}^r\bar{\fA}^{k_j}_{\fr_n}\right)\leq\left(\bar{\sigma}(\rho_n)+C\bar{\sigma}(\rho_n)^{1+\eta}\right)^r
+C{\rho_n}^{-r\tau}\theta^{\sqrt{R}\ln n},\end{align*}
which yields the RHS of $(M1)_r$ due to ${\rm (\overline{Poly})}$  if $R$ is taken sufficiently large. The LHS is proved similarly.

\medskip

\noindent {\it Proof of (ii)}. For $\fA_{\fr_n},$ it is enough to consider the case $\Sep(k_1,\dots, k_r)=r-1$
otherwise we can estimate all $1_{\fA_{\fr_n}} \circ f^{k_i} $ with
$k_i-k_{i-1}<\fs(n)$, except the first, by 1.

So we assume that $0<k_j-k_{j-1}<R \ln n$ and $k_i-k_{i-1}\geq R \ln n$ for $i\neq j.$
Since $(M1)_r$ was proven under the assumption that $\min_i (k_i-k_{i-1})>\sqrt{R} \ln n$ we may assume that $k_j-k_{j-1}<\sqrt{R} \ln n$. Note that by (Appr) and Remark \ref{RemL1}
$$ \mu\left(A_{\fr_n}^+  \left(A_{\fr_n}^+\circ f^k\right)\right)-
\mu\left(1_{\fA_{\fr_n}} \left(1_{\fA_{\fr_n}}\circ f^k\right)\right) \leq 4 \mu\left(A_{\fr_n}^+-1_{\fA_{\fr_n}}\right)
\leq 4 C \mu(\fA_{\fr_n})^{1+\eta} . $$
Therefore (Mov) implies :
$$ \mu\left(A_{\fr_n}^+  \left(A_{\fr_n}^+\circ f^{k_j-k_{j-1}}\right)\right)\leq C \mu(\fA_{\fr_n})(\ln n)^{-1000r}.$$
Take $B=A_{\fr_n}^+  \left(A_{\fr_n}^+\circ f^{k_j-k_{j-1}}\right),$  we get using ${\rm {(EM)}_{r-2}}$  and ${\rm (Poly)}$ that
$$ \mu\left(\prod_{i=1}^r 1_{\fA_{\fr_n}}\left(f^{k_i} x\right)\right)\leq
\mu\left(\prod_{i=1}^r A_{\fr_n}^+\left(f^{k_i} x\right)\right)=
\mu\left(\prod_{i\neq j-1, j} A_{\fr_n}^+\left(f^{k_i} x\right) B(f^{k_{j-1}} x)\right)
$$
$$\leq\mu\left(A_{\fr_n}^+\right)^{r-1}\mu(B)+{C} {\fr_n}^{-r\tau} L^{\sqrt{R} \ln n} \theta^{R \ln n} \leq C \mu(\fA_{\fr_n})^{r-1}(\ln n)^{-1000r}
$$
proving $(M2)_r$.

For $\bfA_{\fr_n},$ we approximate $1_{\bar{\fA}_{\fr_n}}$ by $\bar A_{\fr_n}^+.$ Consider
$$ \tB_r(x_0,\cdots,x_{j-1},x_{j+1},\cdots,x_r)$$
$$=
1_{\bar{\fA}_{\fr_n}}(x_0,x_1)\cdots 1_{\bar{\fA}_{\fr_n}}(x_0,x_{j-1})1_{\bar{\fA}_{a\fr_n}^{k_j-k_{j-1}}}(x_{j-1})
1_{\bar{\fA}_{\fr_n}} (x_0,x_{j+1})\cdots 1_{\bar{\fA}_{\fr_n}}(x_0,x_r), $$

$$ \hB_r(x_0,\cdots,x_{j-1},x_{j+1},\cdots,x_r)$$
$$=
\bA_{\fr_n}^+(x_0,x_1)\cdots \bA_{\fr_n}^+(x_0,x_{j-1})\bA_{a\fr_n}^+ (x_{j-1},f^{k_{j}-k_{j-1}}x_{j-1})
\bA_{\fr_n}^+(x_0,x_{j+1})\cdots \bA_{\fr_n}^+(x_0,x_r). $$
Since ${\rm(\overline{Appr})},$
${\rm(\overline{Mov})}$ and ${\rm(\overline{Sub})}$ hold, we obtain from ${\rm {(EM)}_{r-1}}$
 \Bea
\mu\left(\bigcap_{j=1}^r\bar{\fA}^{k_j}_{\fr_n}\right)&\leq&
\mu\left(\tB_r(x,\cdots,f^{k_{j-1}}x,f^{k_{j+1}}x,\cdots,f^{k_{r}}x)\right)\\
&\leq&
\mu\left(\hB_r(x,\cdots,f^{k_{j-1}}x,f^{k_{j+1}}x,\cdots,f^{k_{r}}x)\right)\\
&\leq& \mu^{r}(\hB_r) +\bar{C} {\fr_n}^{-r\tau} L^{\sqrt{R} \ln n} \theta^{R \ln n}.
 \Eea
{Integrating with respect to all variables except for $x_0$ and $x_{j-1}$, then using ${\rm(\overline{Appr}) (iv)}$ when integrating along $x_0$ for any fixed value of $x_{j-1}$, then finally integrating along $x_{j-1}$, we get
$$
  \mu^{r}(\hB_r)\leq
\left(\bar{\sigma}(\rho_n)+
 \bar{\sigma}(\rho_n)^{1+\eta}\right)^{r-1}\mu\left(\bar{A}_{a\fr_n}^+(x,f^{k_j-k_{j-1}}x)\right)$$
 which by \eqref{eq.app} gives

 \begin{equation*} \label{IntHB}
\mu^{r}(\hB_r)\leq
 \left(\bar{\sigma}(\rho_n)+
 \bar{\sigma}(\rho_n)^{1+\eta}\right)^{r-1}\bar{C}
 \mu(\bar{\fA}_{a\fr_{n}}^{k_j-k_{j-1}})\end{equation*}

Therefore, $(M2)_r$ follows from  ${\rm\overline{(Mov)}}$, provided $R$ is sufficiently large.}

\medskip

\noindent {\it Proof of (iii)} {Fix a large constant $b$ that will be given below.}
Consider first simple targets $\fA_{\fr_n}.$
Denoting $\DS B(x)=\prod_{\alpha=1}^r A_{\fr_{2^i}}^+(f^{k_\alpha} x)$
for $2^i<k_1<\cdots<k_r\leq 2^{i+1},$
we obtain from (Prod), (Gr), (Appr), (Poly) and ${\rm {(EM)}_{r}},$
that $\DS \|B\|_\BAN\leq C L^{r 2^{i+1}}.$ Thus
$$ \mu\left(\left(\prod_{\alpha=1}^r 1_{\fA_{\fr_{2^i}}}(f^{k_\alpha} x)\right)
\left(\prod_{\beta=1}^r 1_{\fA_{\fr_{2^j}}}(f^{l_\beta} x)\right)
\right)$$
$$\leq
 \mu\left(\left(\prod_{\alpha=1}^r A_{\fr_{2^i}}^+ (f^{k_\alpha} x)\right)
\left(\prod_{\beta=1}^r A_{\fr_{2^j}}^+(f^{l_\beta} x)\right)
\right)$$
$$= \mu\left(B(x)
\left(\prod_{\beta=1}^r A_{\fr_{2^j}}^+(f^{l_\beta} x)\right)
\right)\leq \mu(B) \mu\left(A_{\fr_{2^j}}^+\right)^r+C L^{r2^{i+1}}\fr_{2^i}^{-r\tau}
\fr_{2^j}^{-r \tau}
\theta^{2^j \eps}.
 $$Applying already established
$(M1)_r$ to estimate $\mu(B)$, and observing that  the second term is smaller than $\DS C (L^{r2^{-b+1}})^{2^j} 2^{2r\tau u j} \theta^{2^j \eps}$, which is thus much smaller than the first when $b$ is sufficiently large, we finally get $(M3)_r$.

Next, we analyze $\bfA_{\fr_n}.$ Consider
$$ B^*(x, x_1, x_2\dots x_r)=
\left(\prod_{\alpha=1}^r 1_{\bar{\fA}_{\fr_{2^i}}^{k_\alpha}}(x)\right)
\left(\prod_{\beta=1}^r 1_{\bar{\fA}_{\fr_{2^j}}}(x, x_\beta)\right). $$
By ${\rm(\overline{Appr})}$ and ${\rm{(EM)}_r}$ and the already established $(M1)_r$, we get
$$ \mu\Big(\bigcap_{1\leq \alpha,\,\beta \leq r}\big(\bar{\fA}_{\fr_{2^i}}^{k_\alpha}\bigcap\bar{\fA}_{\fr_{2^j}}^{l_\beta}\big)\Big)
\leq\mu\left(B^*(x, f^{l_1} x, \dots, f^{l_r} x)\right)$$
$$\leq\mu\left(\prod_{\alpha=1}^r \bar{A}^+_{\fr_{2^i}}(x, f^{k_\alpha } x)\right)
 \left(\bar{\sigma}(\rho_{2^j})+ \bar{\sigma}(\rho_{2^j})^{1+\eta}\right)^r+CL^{r2^{i+1}}\fr_{2^i}^{-r\tau}\fr_{2^j}^{-r \tau}\theta^{2^j \eps}. $$
Using $(M1)_r$ again we observe that
$$ \mu\left(\prod_{\alpha=1}^r \bar{A}^+_{\fr_{2^i}}(x, f^{k_\alpha } x)\right)\leq
 C\left(\bar{\sigma}(\rho_{2^i})+ \bar{\sigma}(\rho_{2^i})^{1+\eta}\right)^r,$$
which allows to conclude the proof of $(M3)_r$ in the case of $\bfA_{\fr_n}.$ \end{proof}

\begin{Rem}
\label{RkWeakAdm}
In fact, analyzing the proof of Theorem \ref{theo.mixing} we see that the composite targets $\rm{(\overline{Appr})(iii)}$
could be replaced by a weaker condition: there is a function $\sigma_r(\fr)$ such that
$C^{-1}\sigma^r(\fr)<\sigma_r(\fr)<C \sigma^r(\fr)$ and

\begin{subequations}
\begin{equation}
\label{IntMulta}
\int\dots\int \left( \prod_{j=1}^r \brA^+(x,y_j) d\mu(y_j)\right) d\mu(x)=
\sigma_r(\fr)(1+O(\sigma^\eta(\fr)),
\end{equation}
\begin{equation}
\label{IntMultb}
\int\dots\int \left( \prod_{j=1}^r \brA^-(x,y_j) d\mu(y_j)\right) d\mu(x)=
\sigma_r(\fr)(1+O(\sigma^\eta(\fr)) .
\end{equation}
\end{subequations}
We shall call the composite targets satisfying $\rm{(\overline{Mov})}$,
$\rm{(\overline{Sub})}$, $\rm{(\overline{Poly})}$ as well as $\rm{(\overline{Appr})}$ with condition $\rm{(iii)}$
replaced by \eqref{IntMulta}--\eqref{IntMultb} {\em weakly admissible}.
\end{Rem}

\subsection{ Notes.}
There is also a vast literature on Borel-Cantelli Lemmas for dynamical systems
starting with \cite{Phil67}. Some representative examples dealing with hyperbolic systems are
\cite{ANT17, CK01, FMP12, Gou07, GNO10, HNPV13,  HNVZ13, HV95, Kel17, KM99}
while \cite{Ch11, CC19, Kel17, Kim07a, Kim07b, Kim14, Kur55, Tseng08}
deal with systems of zero entropy. The later cases are  more complicated as counterexamples in
\cite{Fa06, GP10} show. Survey \cite{A09}
reviews the results obtained up to 2009 and contains many applications,
some of which parallel the results obtained in Sections \ref{ScHit}--\ref{ScKhinchine} of the present paper.
We refer the reader to Appendix \ref{AppExpMix} for more background on multiple exponential mixing
and for examples of dynamical systems which enjoy this property.
We note that limit theorems for smooth systems which are only assumed to
be multiply exponentially mixing (but without any additional assumptions) are considered in
\cite{BG19, Ch95, St10}. \cite{Galatolo07}  obtains a Logarithm Law for hitting times under an assumption
of superpolynomial mixing which is weaker than our exponentially mixing assumption.
We note that in our approach the {\it exponential} rate of mixing is crucial for verifying the condition
$(M3)_r$ pertaining to interscale independence. Therefore it is an open problem to ascertain if similar results
hold under weaker mixing assumptions.

\section{MultiLog Laws for recurrence and hitting times}
\label{ScHit}

In this section we apply the results of Section \ref{SSMultExpMix} to obtain MultiLog Laws for multiple exponentially mixing diffeomorphisms and flows. We will assume that $f$ is a smooth diffeomorphism of a compact $d-$dimensional Riemannian manifold $M$ preserving a
smooth measure $\mu.$ 
From now on, we take
$\BAN$ in Definition \ref{def.mixing}
to be the space of Lipschitz observables defined over $M^{d+1}$.

\subsection{  Results.}
Let $(f,M,\mu)$ be a smooth dynamical system. Let $d_n^{(r)}(x,y)$  be the $r$-th minimum of
$$d(x,fy),\cdots, d(x,f^n y).$$

 The following result was obtained for a large class of weakly hyperbolic systems as
a consequence of dynamical Borel-Cantelli Lemmas
\begin{subequations}
\begin{equation}
\label{ClosestRecurrence}
	\limsup_{n\rightarrow \infty} \frac{|\ln d_n^{(1)}(x,x)|}{\ln n}
        = \frac{1}{d},
\end{equation}
\begin{equation}
        \label{ClosestVisit}
	\limsup_{n\rightarrow \infty} \frac{|\ln d_n^{(1)}(x,y)|}{\ln n}
        = \frac{1}{d}.
\end{equation}
\end{subequations}

 In particular, the following results are known.

\begin{Thm}
\label{ThHitSuperPoly}
(a) If a smooth system $(f, M, \mu)$ has superpolynomial decay
of correlations for Lipschitz observables, that is,
$$|\mu(A(x) B(f^n x))-\mu(A)\mu(B)|\leq a(n) \|A\|_{Lip} \|B\|_{Lip}
\quad\text{where}\quad \forall s \lim_{n\to\infty} n^s a(n)=0,$$
then
for {\em \bf all} $x$ \eqref{ClosestVisit}
holds
for a.e. $y.$
If in addition, $f$ has positive entropy, then
\eqref{ClosestRecurrence}
holds for a.e. $x$.

(b) If, in addition, $f$ is partially hyperbolic then
for {\em \bf all} $x$ and a.e. $y$
\begin{equation}
\label{ClosestRecurrenceLnLn}
	\limsup_{n\rightarrow \infty} \frac{|\ln d_n^{(1)}(x,y)|-\frac{1}{d}\ln n}{\ln \ln n}
        = \frac{1}{d}.
\end{equation}
\end{Thm}

In part (a), \eqref{ClosestRecurrence} is proven in \cite[Theorem 1]{Saussol06} and
\eqref{ClosestVisit} is proven in
 \cite[Theorem 4]{Galatolo07}. Part (b) is proven in \cite[Theorem 7]{Dolgopyat}.

\begin{Que}
Suppose that $(f, \mu)$ is exponentially mixing then \eqref{ClosestRecurrenceLnLn} holds
for {\em \bf all} $x$ and a.e. $y.$
\end{Que}

\medskip

\noindent {\sc \bluegray MultiLog Law for recurrence and for hitting times. }The goal of this section is to obtain an analogue of of \eqref{ClosestRecurrenceLnLn} for multiple hits as well
as for returns
for multiple exponentially mixing systems as in Definition \ref{def.mixing}.

\begin{Def} \label{DefG}
Given a smooth system $(f,M,\mu)$, define
\begin{align*} \cG_r&= \left\{ x : \text{ for a.e. } y, \quad
	\limsup_{n\rightarrow \infty} \frac{|\ln d_n^{(r)}(x,y)|-\frac{1}{d}\ln n}{\ln \ln n}
        = \frac{1}{rd}\right\}, \\
       \bcG_r&= \left\{ x : 	\limsup_{n\rightarrow \infty} \frac{|\ln d_n^{(r)}(x,x)|-\frac{1}{d}\ln n}{\ln \ln n}
        = \frac{1}{rd}\right\}.
	\end{align*}
\end{Def}

\begin{Thm} \label{thm3}
Suppose that $(f,M,\mu,\BAN)$ is {$(2r+1)$-fold} exponentially mixing.\footnote{As seen from Proposition \ref{ThEM-BC}, part $(a)$ holds for $2r$-fold exponentially mixing systems.} Then

(a) $\mu(\cG_r)=1;$ (b) $\mu(\bcG_r)=1.$
\end{Thm}

 \noindent {\sc \bluegray Failure of the MultiLog laws for generic points.} Naturally, one can ask if in fact, $\cG_r$ equals to $M.$ If $r=1$ the answer is often positive (see
Theorem \ref{ThHitSuperPoly}(b)).
It turns out
that for larger $r$ the answer is often negative.

\begin{Def}
\label{DefH}
Given a function $\zeta:\N \to \N^*$, define
\begin{align*}  \cH&= \left\{ x : \text{ for a.e. } y, \quad
  \text{ for all }	r\geq 1 : \limsup_{n\rightarrow \infty} \frac{|\ln d_n^{(r)}(x,y)|-\frac{1}{d}\ln n}{\ln \ln n}
        = \frac{1}{d}\right\}, \\
        \bcH_\zeta&= \left\{ x :
  \text{ for all }	r\geq 1 :
	\limsup_{n\rightarrow \infty} \frac{|\ln d_n^{(r)}(x,x)|}{\zeta(n)}
        = \infty\right\}.
	\end{align*}
\end{Def}

\begin{Thm}
\label{ThTop}  Suppose that 
 the periodic points of $f$ are dense. Then

\noindent (a)  If $\cG_1=M$, then $\cH$   contains a  $G_\delta$  dense set.

\noindent (b) {For any   $\zeta:\N \to \N^*$, $\bcH_\zeta$  contains a  $G_\delta$  dense set.}

\end{Thm}
Thus for $r\geq 2$ topologically typical points do {\bf not} belong to $\cG_r$ or $\bcG_r$.

\medskip

\noindent {\sc \bluegray Failure of the MultiLog laws for non mixing systems. The case of toral translations.}

Theorem \ref{ThTop} emphasizes the necessity of a restriction on $x$ in Theorem \ref{thm3}.

In a similar spirit, we show that the mixing assumptions made in this paper are essential. To this end we consider the case when the dynamical system is $(T_\a,\Tor^d,\lambda)$ where $T_\a$ is the translation of vector $\a$ and $\lambda$ is the Haar measure on $\Tor^d$.

Define \begin{align*}
\cE_{r}&= \left\{ x : \text{ for a.e. } y, \quad
	\limsup_{n\rightarrow \infty} \frac{|\ln d_n^{(r)}(x,y)|-\frac{1}{d}\ln n}{\ln \ln n}= \frac{1}{2d}\right\},\\
\bar \cE_{r}&= \left\{ x : 	\limsup_{n\rightarrow \infty} \frac{|\ln d_n^{(r)}(x,x)|-\frac{1}{d}\ln n}{\ln \ln n}
        = \frac{1}{d}\right\}.
   \end{align*}
  \begin{Thm}
    \label{ThReturnsTransl}
 For $\lambda$-a.e. $\alpha \in \Tor^d$, the system $(T_\a,\Tor^d,\lambda)$, satisfies

  \noindent a) $\lambda(\cG_1)=1$ and $\lambda(\cE_{r})=1$ for $r\geq2;$

\noindent b) $\bar \cE_r=M$ for all $r\geq1.$
\end{Thm}

The proof requires different techniques from the rest of the results of this section, that are related to homogeneous dynamics on the space of lattices, so it will be given
in Section \ref{ScKhinchine} after we introduce the necessary tools.
\smallskip

\noindent {\sc \bluegray The case of flows.}	
Here we describe the analogue results of Theorems \ref{thm3} and \ref{ThTop}  for flows.
Let $\phi$ be a smooth flow on a  $(d+1)$ dimensional  Riemannian manifold $M$ preserving a smooth measure $\mu$.

 Observe that if $\phi^t(y)$ is close to $x$ for some $t,$ then the same is true for $\phi^{\tilde{t}}(y)$ with $\tilde{t}$ close to $t.$ Thus we would like to count only one return for the whole
	connected component lying in the neighborhood of $x$.
Namely, for some fixed $\rho>0,$
for $i\geq 0$, let $[t_i^-, t_i^+]$ denote the consecutive time intervals  such that $\phi^t y\in B(x, \rho)$
for $t\in [t_i^-, t_i^+].$ Let $t_i$ be the argmin of $d(x, \phi^t (y))$ for
$t\in [t_t^-, t_i^+]$. Let $d_n^{(r)} (x,y)$ be the $r-$th minimum of
\begin{equation}
d(x,\phi^{t_1}(y)),\ldots, d(x,\phi^{t_k}(y)), \quad t_k \leq n < t_{k+1}. \end{equation}

 Theorem~\ref{thm3} and Theorem~\ref{ThTop} have the following counterpart in the case of flows. Note that the dimension of the manifold in the case of flows is $d+1$.

\begin{Thm}
\label{ThFlowRec}
Suppose that  the smooth system $(\phi,M,\mu,\BAN)$ is $(2r+1)$-fold exponentially mixing.  Then

\noindent a) 
$\mu(\cG_r)=1;$

\noindent b) 
$\mu(\bcG_r)=1.$

If, in addition, periodic points of $\phi$ are dense then
	
\noindent c)If $\cG_1=M$ then $\cH$   contains a  $G_\delta$  dense set;

\noindent d)   For any   $\zeta:\N \to \N^*$, $\bcH_\zeta$  contains a  $G_\delta$  dense set.
  \end{Thm}

\subsection{Slowly recurrence and proof of Theorem \ref{thm3}.} \label{sec.wr}

Since $\mu$ is a smooth measure,  there is a smooth function $\gamma(x)$ such that
\begin{equation}
\label{SmallBallAsy}
\mu\left(B(x, \rho)\right)=\gamma(x) \rho^d+O\left(\rho^{d+1}\right),
\end{equation}
where the constant in $O\left(\rho^{d+1}\right)$ is uniform in $x$.

Given $x\in M,$  let
\begin{equation}
\label{BallTarget}
\fA_{x,\fr}=\{y:\,d(x,y)\leq\rho\}
\end{equation}
and\footnote{In the definition of the composite target $\bfA_{\fr}$, we include the factor ${(\gamma(x))^{-1/d}}$ because we want that for every $x$, $\int 1_{\bfA_\fr}(x,y)d\mu(y)$ be essentially the same number to be able to check
$ (\overline{\rm Appr})(iii)$ for these targets.}
\begin{equation}
\label{VaryingWidthTarget}
\bfA_{\fr}=\left\{(x,y):d(x,y)\leq \frac{\rho}{(\gamma(x))^{1/d}} \right\}
\end{equation}

{ We use the notation $\fA_{x,\fr}^k$ for the event $1_{\fA_{x,\fr}}\circ f^k$.
We also recall the notation  $ {\bfA}^k_{\fr}=\{x: (x, f^k x)\in{\bfA}_{\fr}\}.$
We also keep the notation $\sigma(\fr)=\mu(\fA_{x,\fr})$, and $\bsigma(\fr)=(\mu\times\mu)(\bfA_\fr)$.

For $s\geq 0$, we let $\fr_n=n^{-1/d} \ln^{-s}n$, and recall that $N^n_{\fr_n}$ denotes the number of times $k\leq n$ such that $\fA^k_{x,\fr_n}$ (or $\bfA^k_{\fr_n}$) occurs.

By compactness, there exists a constant $c>0$ such that
$$ \left\{(x,y):d(x,y)\leq c^{-1}\rho \right\} \subset \bfA_{\fr}\subset \left\{(x,y):d(x,y)\leq c\rho \right\}.$$
Thus the statement of Theorem \ref{thm3} becomes equivalent to the following :
\begin{itemize}
\item[(a)] If $s>\frac{1}{rd}$,  then for $\mu$-a.e. $x$, we have that  for large $n$,
$N^n_{\fr_n}<r.$
\item[(b)] If $s\leq\frac{1}{rd}$, then for $\mu$-a.e. $x$, there are infinitely many $n$ such that
$N^n_{\fr_n}\geq r.$
\end{itemize}

With the notation $\bS_r=\sum_{j=1}^\infty \left(2^j \bv_j\right)^r$
where $\bv_j=\fv(\fr_{2^j})$ (in the $\fA_{x,\rho_n}$ case) or
${\bv}_j=\bar\fv(\fr_{2^j})$ (in the $\bfA_{\rho_n}$ case), we see from \eqref{SmallBallAsy} that
$\bS_r=\infty$ if and only if $s\leq\frac{1}{rd}$.

Hence Theorem \ref{thm3} follows from the alternative of Corollary
 \ref{cor.mixing}, since $(f,M,\mu,\BAN)$ is 
 $(2r+1)$-fold exponentially mixing, provided we establish the following.
\begin{Prop} \label{prop.targets}
(a) For $\mu$-a.e. $x$ the targets $\{\fA_{x,\fr_n}\}$  are simple admissible targets.

(b)The targets $\{\bfA_{\fr_n}\}$ are composite admissible targets.
\end{Prop}
 }

 The rest of this section is devoted to the

\noindent { \it Proof of Proposition \ref{prop.targets}.}

{ Observe first that with the definition of $\rho_n$ and \eqref{SmallBallAsy}, we have that $({\rm Poly})$ and  $(\overline{\rm Poly})$ hold for every $x$ for the target sequences $\{\fA_{x,\fr_n}\}$ as well as for the sequence  $\{\bfA_{\fr_n}\}$.

We proceed with the proof of $({\rm Appr})$ and  $(\overline{\rm Appr})$ and  $(\overline{\rm Sub})$ properties.
\begin{Lem}
\label{lemma.appr}
For each $x,$ the targets $\fA_{x,\rho}$
satisfy ${\rm(Appr)}$.  The targets $\bfA_{\fr}$ satisfy ${\rm(\overline{Appr})}$ and $(\overline{\rm Sub})$.
\end{Lem}

\begin{proof}
{For the targets $\fA_{x,\fr},$ the statement follows from Lemma \ref{RemLip} by taking $\Phi(y)=d(x,y)$ (that is a Lipschitz function), $a_1(\rho)=0$ and $a_2(\rho)=\rho$.

For the  targets $\bfA_{\fr}$, we use  Lemma \ref{RemLipcomposite}. We take $\Phi(x,y)=d(x,y)\gamma(x)^{1/d},$  $a_1(\rho)=0$ and $a_2(\rho)=\rho$. We check $(h1)$ since $\gamma(x)/\gamma(y)$ is bounded for $(x,y)\in X \times X$. Property $(h2)$ is obvious.
As for $(h3)$ it follows from the definition of $\gamma(x)$ in \eqref{SmallBallAsy}.}

Finally, for any $k_1,k_2$, when $x\in\bfA_{\fr}^{k_1}\cap\bfA_{\fr}^{k_2},$ we have
$$d(f^{k_1}x,f^{k_2}x)\leq  d(x,f^{k_1}x)+d(x,f^{k_2}x)\leq \frac{2 \rho}{(\gamma(x))^{1/d}}\leq\frac{a \rho}{(\gamma(f^{k_1}x))^{1/d}},$$
for some $a>0.$ Hence $ \bfA_{\fr}^{k_1}\cap\bfA_{\fr}^{k_2}\subset f^{-k_1}\bfA_{a\fr}^{k_2-k_1}$, which is ${\rm(\overline{Sub})}.$ Lemma \ref{lemma.appr} is proved. \end{proof}

Next we prove of the ${\rm(Mov)}$  (for a.e. $x$)
and  $\overline{\rm(Mov)}$ properties. For this we state a Lemma on recurrence  for the multiple mixing system  $(f,M,\mu)$ that is of  an independent interest. We first introduce two definitions.

\begin{Def}[Slowly recurrent points]  Call $x$ {\it slowly recurrent} for  the system $(f,M,\mu)$
if for each
$A, K>0$, there $\exists \rho_0$ such that for all $\rho<\rho_0$ for all $n\leq K|\ln \rho|$ we have
\begin{equation*}
\label{EqWNR}
 \mu(B(x, \rho) \cap f^{-n} B(x, \rho))\leq  \mu (B(x, \rho)) |\ln \rho|^{-A}.
\end{equation*}
\end{Def}

\begin{Def}[Slowly recurrent system]  Call the system $(f,M,\mu)$
{\it slowly recurrent} if for each $A>0$ $\exists \rho_0$ such that for all $\rho<\rho_0$ for all $n\in \N^*$ we have
\begin{equation*}
\label{EqWNRS}
 \mu\left(\left\{x:d(x,f^nx)<\rho\right\}\right)\leq |\ln \rho|^{-A}.
\end{equation*}
\end{Def}

\begin{Lem} \label{lemma.mov} Suppose that $(f,M,\mu,\BAN)$ is {$2$-fold exponentially mixing}. Then
\begin{itemize}
\item[$i)$] $(f,M,\mu)$ is slowly recurrent.

\item[$ii)$] Almost every point is slowly recurrent.
\end{itemize}

As a consequence, we have that \begin{itemize}
\item[(a)] For $\mu$-a.e. $x$, the targets $\fA_{x,\rho_n}$ satisfy ${\rm(Mov)}$.
\item[(b)]  The targets $\bfA_{\fr_n}$ satisfy $\overline{\rm(Mov)}$.
\end{itemize}
\end{Lem}

\begin{proof} Take $B=A^2$. If $k\geq B\ln|\ln\rho|,$ take $\hat\fr=|\ln \rho|^{-A}.$ By $2$-fold exponential mixing, we get
\begin{multline}
\label{eq.high}
\mu (x: d(x, f^k x)\leq \fr)\leq
\mu (x: d(x, f^k x)\leq \hat\fr)\\ \leq \mu \left(\bar{A}^+_{\hat{\fr}}(x,f^kx)\right) \leq C\left(\hat\fr^{d}+\hat\fr^{d+d\eta}+\hat\fr^{-\tau} \theta^k\right)\leq |\ln \rho|^{-2A},
\end{multline}
provided $\fr$ is sufficiently small.

Now fix any $1\leq k\leq B\ln|\ln\rho|$. {Denote $\|f\|_{1}=\max_{x\in M}\|Df(x)\|.$} Assume that $x$ satisfies $d(x, f^k x)\leq \fr$, then for any $l$ we have that
$$d(f^{(l-1)k}(x), f^{lk} x)\leq \|f\|_{1}^{(l-1)k} \fr$$
If we take $L=[4B\ln|\ln\rho|/k]+1$ we find that $$d(x,f^{Lk} x)\leq \sum_{l\leq L-1} \|f\|_{1}^{lk} \fr \leq \sqrt{\fr},$$
 provided $\fr$ is sufficiently small. But $kL\geq B\ln|\ln{\sqrt{\rho}}|,$ hence \eqref{eq.high} applies and we get
 $$
\mu (x: d(x, f^k x)\leq \fr)\leq
\mu (x: d(x, f^{Lk} x)\leq \sqrt{\fr}) \leq |\ln \rho|^{-A},
$$
proving $i)$.

We proceed now to the proof of $ii)$. Define for $j,k\in \N^*$
$$ H_{j,k}(x):=\mu(B(x, 1/2^j)\cap f^{-k} B(x, 1/2^j)).$$

Note that
\begin{eqnarray*}
\int H_{j,k}(x)d\mu(x)&=&\iint 1_{[0,1/2^j]} d(x,y) 1_{[0,1/2^j]} d(x, f^k y) d\mu(x) d\mu(y)\\
&\leq& \iint 1_{[0,1/2^j]} d(x,y) 1_{[0,1/2^{j-1}]}d(y, f^k y) d\mu(x) d\mu(y)\\
&\leq& C\mu(B(x,1/2^j)) \int 1_{[0,1/2^{j-1}]} d(y, f^k y) d\mu(y)
\end{eqnarray*}
where we used that $\mu(B(y,1/2^j))\leq C\mu(B(x,1/2^j))$ for any $x,y \in M$. Part $i)$ then implies that for sufficiently large $j$  it holds that
$$\int H_{j,k}(x)d\mu(x)\leq \mu(B(x,1/2^j))j ^{-A-3}.$$
For such $j$ we get from Markov inequality
$$ \mu\left(x: \exists k \in(0,Kj]:\,\,
H_{j,k}(x)> \mu(B(x,1/2^j))j^{-A}\right)\leq Kj^{-2}.$$
Hence Borel Cantelli Lemma implies that for almost every $x$ there exists $\bar j$ such that $H_{j,k}(x)\leq \mu(B(x,1/2^j))j^{-A}$ for every $j\geq \bar j$ and every $k\in (0,Kj]$, which implies $ii)$.

Finally, $(a)$ and $(b)$ clearly follow from $ii)$ and $i)$ respectively.    Lemma \ref{lemma.mov} is thus proved. \end{proof}

With Lemmas \ref{lemma.appr} and \ref{lemma.mov}, the proof of Proposition \ref{prop.targets} is finished. \hfill $\tiny\Box$  

\begin{proof}[Proof of Theorem \ref{thm3}.] Theorem \ref{thm3} directly follows from Proposition \ref{prop.targets}  and Corollary \ref{cor.mixing}. \end{proof}

\subsection{  Generic failure of the MultiLog Law. Proof of Theorem \ref{ThTop}.}
\label{sec6} \begin{proof}
To prove part a),  we first prove that periodic points belong to $\cH_r$.
By assumption, for any $x \in M$ and almost  every $y,$
\begin{equation}
\label{DrD1Eq}
\limsup_{n\rightarrow \infty} \frac{|\ln d_n^{(1)}(x,y)|- \frac{1}{d} \ln n}{\ln \ln n} = \frac{1}{d}.
\end{equation}
Since $d_n^{(r)} (x,y) \geq d_n^{(1)} (x,y)$, it follows that for any $x \in M$, any $r\geq 1$, and almost  every $y$

\begin{equation}
\label{DrD1In}
\limsup_{n\rightarrow \infty} \frac{|\ln d_n^{(r)}(x,y)| -\frac{1}{d} \ln n}{\ln \ln n} \leq
\frac{1}{d}.
\end{equation}

To prove the opposite inequality let
	$$ \cH_{m,l,r}= \left\{ x : \exists \cY\text{-open, }  \mu(\cY)> {1}-\frac{1}{l}: \quad  \forall y \in \cY,  \frac{|\ln d_m^{(r)}(x,y)|-\frac{1}{d}\ln m}{\ln \ln m} >\frac{1}{d}-\frac{1}{l}  \right\}$$
We have that
$$ \left\{ x : \text{ for a.e. } y, \quad
  \text{ for all }	r\geq 1 : \limsup_{n\rightarrow \infty} \frac{|\ln d_n^{(r)}(x,y)|-\frac{1}{d}\ln n}{\ln \ln n}
        \geq  \frac{1}{d}\right\} = \bigcap_{l\geq 1, r \geq 1} \bigcup_{m\geq 1} \cH_{m,l,r}.$$
But 	$\cH_{m,l,r}$ is an open set. Hence we finish if we show that for any fixed
$r$ and $l$, $\bigcup_m \cH_{m,l,r}$ contains the dense set of periodic points.
	
		Let $\brx$ be a periodic point of period $p.$
		Take $U$ to be some small neighbourhood of $\brx$ and denote by
	$\Lambda$ the Lipschitz constant of $f^p$ in $U$.

		 By \eqref{DrD1Eq}, there exists
				$n \geq  \exp\circ \exp(\Lambda+pr)$ and $\cY$ such that $\mu(\cY)>1- \frac{1}{l}$,  such that for every $y \in \cY$, there exists $k\in [1,n]$ satisfying
	$$ d (\brx,f^k y) \leq \left( \frac{1}{n}\right)^{\frac{1}{d}} \left(\frac{1}{\ln n}\right)^{\frac{1}{d} - \frac{1}{2l}}.$$

 Then
	\begin{equation*}\label{eq1}
d (\brx,f^{k+pj} y)  = d (f^{pj} \brx,f^{k+pj} y)\leq  {\Lambda}^{r} \left( \frac{1}{n}\right)^{\frac{1}{d}} \left(\frac{1}{\ln n}\right)^{\frac{1}{d}- \frac{1}{2l}}, \quad 0\leq j \leq r-1.
\end{equation*}
Hence for $y \in \cY$ and $m=n+p(r-1)$, we have that
	$$ d_{m}^{(r)} (\brx,y) \leq {\Lambda}^{r} \left( \frac{1}{n}\right)^{\frac{1}{d}} \left(\frac{1}{\ln n}\right)^{\frac{1}{d} - \frac{1}{2l}} <\left( \frac{1}{m}\right)^{\frac{1}{d}}  \left(\frac{1}{\ln m}\right)^{\frac{1}{d} - \frac{1}{l}},$$
because we took $n \geq  \exp\circ \exp(\Lambda+pr)$. Hence $\bar x \in \cH_{m,l,r}$ and the proof of $(a)$ is finished.

{We now turn to the proof of $(b)$. Given any function $\zeta:\N \to \N^*$,
define
$$\cA_{m,l} = \left\{ x: |\ln d_m^{(l)} (x,x)| > m\zeta(m) \right\}.$$
Observe that $\bcH_\zeta \subset \bigcap_l \bigcup_m \cA_{m,l}$. But $\cA_{m,l}$ is open  and $\bigcup_m \cA_{m,l}$ clearly contains the periodic points. Part $(b)$ is thus proved.}
\end{proof}

\subsection{  The case of flows. Proof of Theorem \ref{ThFlowRec}. }
The proof  proceeds in the same way as for diffeomorphisms with minimal modifications that we now explain.
First, we need to modify the targets
$$ \fA_{x,\fr}=\{y: \exists\,s\in [0,1], d(x,\phi^sy)\leq \fr \},$$
and
$$ \bfA_{\fr}=\left\{(x,y): \exists\,s\in [0,1], d(x,\phi^sy)\leq \frac{\rho}{\gamma(x)^{1/d}} \right\}$$
where $\DS \gamma(x)=\lim_{\fr\to 0} \mu(\fA_{x,\fr})/\fr^d.$
Consider the targets
$$ \fA_{x,\fr}^n= \phi^{-n}\fA_{x,\fr}, \quad \bfA_{\fr}^n=\{x:(x,\phi^nx)\in \bfA_{\fr}\}$$
for $\mathbb{N}^*$ and let $\sigma(\rho)=\mu(\fA_{\fr,x}),$ $\bar\sigma(\rho)=(\mu\times\mu)(\bfA_{\fr}).$

To prove a) and b) of Theorem \ref{ThFlowRec} we can apply Corollary \ref{cor.mixing} to the smooth system $(\phi,M,\mu,\BAN)$ and to the targets  $\fA_{x,\fr}^n$ and $\bfA_{\fr}^n$. For this, we just need to see that the targets  are admissible targets. This can be checked as in the proof of Proposition \ref{prop.targets}, with very minor differences. Let us check for instance that ${\rm(\overline{Sub})}$ holds for $\bfA_{\fr}^n$.
Note that when $x\in\bfA_{\fr}^{n_1}\cap\bfA_{\fr}^{n_2}$ for $n_1<n_2,$ we have some $s_1,\,s_2\in[0,1]$ such that
$$d(x,\phi^{n_1+s_1}x)\leq\frac{\rho}{(\gamma(x))^{1/d}},\quad d(x,\phi^{n_2+s_2}x)\leq\frac{\rho}{(\gamma(x))^{1/d}}.$$
Hence
\begin{multline*} d(\phi^{n_1}x,\phi^{n_2+s_2-s_1}x)\leq \max_{s\in [-1,0]}\|\phi^s\|_{C^1} d(\phi^{n_1+s_1}x,\phi^{n_2+s_2}x)\\\leq \max_{s\in [-1,1]}\|\phi^s\|_{C^1}\frac{2\rho}{(\gamma(x))^{1/d}}\leq\frac{a \rho}{(\gamma(\phi^{n_1}x))^{1/d}}\end{multline*}
for some $a>0.$
It follows that $\bfA_{\fr}^{n_1}\cap\bfA_{\fr}^{n_2}\subset\phi^{-n_1}{\bfA}_{a\fr}^{n_2-n_1},$
which is ${\rm(\overline{Sub})}.$ \footnote{ When $s_2-s_1<0,$ we modify $\bfA_{\fr}$ by $\tilde{\Omega}_{\fr}=\left\{(x,y): \exists\,s\in [-1,1], d(x,\phi^sy)\leq \frac{\rho}{\gamma(x)^{1/d}} \right\}$ and get $\bfA_{\fr}^{n_1}\cap\bfA_{\fr}^{n_2}\subset\phi^{-n_1}{\tilde{\Omega}}_{a\fr}^{n_2-n_1},$ which gives $(M2)_r$ by a same argument of Proposition~\ref{ThEM-BC}(ii).} As for the proofs of
 ${\rm({Mov})}$ and  ${\rm(\overline{Mov})}$, they are obtained as in the case of maps {\it via} the notion of slow recurrence.
 We say that a point $x$ is {\it slowly recurrent} for  the flow
if for each
$A, K>0$, there $\exists \rho_0$ such that for all $\rho<\rho_0$ for all $n\leq K|\ln \rho|$ we have
\begin{equation*} \mu\left(\fA_{x,\fr} \cap \fA_{x,\fr}^n\right)\leq \mu(\fA_{x,\fr})|\ln \rho|^{-A}.
\end{equation*}

Similarly we say that  the  flow is
{\it slowly recurrent} if for each $A>0$ $\exists \rho_0$ such that for all $\rho<\rho_0$ for all $n\in \N^*$ we have
\begin{equation*}
 \mu\left(\bfA_{\fr}^n\right)\leq |\ln \rho|^{-A}.
\end{equation*} The same proof of Lemma \ref{lemma.mov} then shows that if the  system $(\phi,M,\mu,\BAN)$ is exponentially mixing, it holds that  $\mu$-a.e. point is slowly recurrent for the flow, and that the flow is
{\it slowly recurrent}.
Properties  ${\rm({Mov})}$ and  ${\rm(\overline{Mov})}$ are immediate consequences.

The proof of part c) and part d) also proceeds in the same way as for maps. Namely we first  see that periodic orbits of the flow belong to $\cH_r$ and $\bcH_r$ and then use the genericity argument. \hfill $\square$

}

\subsection{  Notes.}
Many authors obtain Logarithm Law \eqref{ClosestVisit}
for hitting times as a consequence of dynamical Borel-Cantelli
Lemmas. See \cite{CK01, Dolgopyat, Galatolo05,  HV95} and references wherein.
\cite{Galatolo07} also
studies
return times. We note that \cite{Galatolo07} works
under much weaker conditions than those imposed in the present paper,
however, his results are valid only for $r=1$ (the first visit).

\cite{Kel17, KrKR19, KlKR19} study the recurrence problem when the
$\lim\sup$ in \eqref{ClosestVisit} is replaced by $\lim\inf$.
In particular, \cite{KlKR19} proves that for several expanding maps
the
$$ \lim\inf_{n\to \infty}  \frac{n \; d_n^{(1)}(x,y)}{\ln \ln n} $$
exists for almost all $y.$

{Theorem \ref{ThReturnsTransl} shows that some systems may satisfy logarithmic laws for $r=1$ that are the same as in the exponentially mixing case, but fail to do so for $r\geq 2$. Logarithm Laws for unipotent flows were obtained in \cite{AM1, AM2, GK17, Kel17}. It is not known which kind of MultiLog Laws hold for such flows.
}

\section{  Poisson Law for near returns.}
\label{ScPoissonHR}
In this section we suppose that $\mu$ is a smooth measure and that $(f,M,\mu,\BAN)$
is an $r$-fold
exponentially mixing system for all $r.$
In the previous section we verified properties $(M1)_r$ and $(M2)_r$ for the targets
$\fA_{x,\fr}$ given by \eqref{BallTarget}, for almost every $x,$ and for the targets
$\bfA_{\fr}$ given by \eqref{VaryingWidthTarget}. Moreover, we have that $\DS \lim_{\rho\to0} {\rho^{-d} \sigma(\rho)}=\gamma(x)$ and $\DS \lim_{\rho\to0} {\rho^{-d} \bsigma(\rho)}=1$, where $\sigma(\rho)=\mu(\fA_{\fr,x}),$ $\bar\sigma(\rho)=(\mu\times\mu)(\bfA_{\fr}).$  Accordingly Theorem \ref{ThPoisson}
gives the following.

\begin{Thm}
\label{PrEasy}
(a) For almost all $x$ the following holds. Let $y$ be uniformly distributed with respect to
$\mu$. The number of visits of $\{f^k(y)\}_{k\in [1,\tau \fr^{-d}]}$ to $B(x,\fr)$ converges to a Poisson distribution with parameter $\tau \gamma(x)$ as $\fr\to 0.$
Moreover letting $n=\tau \fr^{-d}$ we have
the sequence
\begin{equation}
\label{ScaledHits}
\frac{d_n^{(1)}(x,y)}{\fr}, \frac{d_n^{(2)}(x,y)}{\fr}, \dots, \frac{d_n^{(r)}(x,y)}{\fr}, \dots
\end{equation}
converges to the Poisson process with measure {$\DS \gamma(x) \tau d t^{d-1}dt.$}

(b) Let $x$ be chosen uniformly with respect to $\mu.$ Then the number of visits of $\{f^k(x)\}_{k\in [1,\tau \fr^{-d}]}$ to
$\DS B\left(x, \frac{\fr}{\gamma^{1/d}(x)}\right)$ converges to a Poisson distribution with parameter $\tau $ as $\fr\to 0.$
\end{Thm}

\begin{proof} All the results except for Poisson limit for \eqref{ScaledHits}
follows from Theorem \ref{ThPoisson}.
To prove the Poisson limit for \eqref{ScaledHits} we need to check that for each choice of
$r_1^-<r_1^+<r_2^-<r_2^+<\dots <r_s^-<r_s^+$ the number of times $k\in [1, \tau \rho^{-d}]$ where
$d(x,f^k y)\in \left[r_j^-\fr, r_j^+\fr\right]$ are converging to independent Poisson random variables with parameters
$$ \gamma(x) \int_{r_j^-}^{r_j^+} {\tau dt^{d-1} dt}=
\gamma(x) \tau \left[(r_j^+)^d-(r_j^-)^d\right].
$$
But this follows from Theorem \ref{ThPoissonProcess}. The latter theorem can be applied since
$\widetilde{(M1)}_r$ follows from Property (Appr) of the  targets
$$ \Omega^{k, i}_\fr=\{y: d(x,f^ky)\in [r_i^-\fr , r_i^+\fr ]\} $$
that holds due to Lemma \ref{RemLip}.
\end{proof}

There are two natural questions dealing with improving this result.
In part (a) we would like to specify more precisely the set of $x$ where the Poisson limit law for hits holds.
In part (b) we would to remove an annoying factor $\gamma^{1/d}(x)$ from the denominator.
Regarding the first question we have

\begin{Con}
\label{ConjAP-Pois}
If $f$ is exponentially mixing then the conclusion of Proposition \ref{PrEasy}(a) holds for
{\bf all} non-periodic points.
\end{Con}

Regarding the second question we have the following.

\begin{Thm}
\label{ThPoisRet}
Let $x$ be chosen uniformly with respect to $\mu.$ Then the number of visits of $\{f^k(x)\}_{k\in [1,\tau \fr^{-d}]}$ to
$B(x, \fr)$ converges to a mixture of Poisson distributions. Namely, for each $l$
\begin{equation}
\label{MixPois}
\lim_{\fr\to 0} \mu(\Card(n\leq \tau \fr^{-d}: d(x, f^n x)\leq \fr)=l)=
\int_M e^{-\gamma(z) \tau} \frac{(\gamma(z) \tau)^l}{l!} d\mu(z).
\end{equation}
\end{Thm}
In other words to obtain the limiting distribution in Theorem \ref{ThPoisRet} we first sample
$z\in M$ according to the measure $\mu$ and then consider Poisson random variable with
parameter $\tau\gamma(z).$

\begin{Cor}
\label{CrExponential}
If $f$ preserves a smooth measure and is $r$-fold exponentially mixing for Lipschitz observables for all $r\geq 2$ then

(a) For almost all $x$ we have that if $\tau_\eps(y)$ is an the first time an orbit of $y$ enters
$B(x, \eps)$ then for each $t$
$$ \lim\mu(y: \tau_\eps(y) \eps^d>t)=e^{-\gamma(x) t} $$

(b) If $T_\eps(x)$ is the first time the orbit of $x$ returns to $B(x, \eps)$ then
$$ \lim\mu(x: T_\eps(x) \eps^d>t)=\int_M e^{-\gamma(z) t} d\mu(z) . $$
\end{Cor}

\begin{proof}
This is a direct consequence of Theorems \ref{PrEasy}(a) and \ref{ThPoisRet}. For example to get
part (b), take $l=0$ in \eqref{MixPois}.
\end{proof}

\begin{proof}[Proof of Theorem \ref{ThPoisRet}]
Consider the targets
$$\hat\fA_\fr(x, y)=\{(x,y)\in M\times M: d(x,y)\leq \fr\}$$
and let $\hat\fA_{\fr}^k=\{x: (x, f^k x)\in \hat\fA_\fr\}.$
Note that $(M2)_r$ for $\bfA_{\fr}^k$ implies $(M2)_r$ for $\hat\fA_{\fr}^k.$
However, $(M1)_r$ is false for targets $\hat\fA_{\fr}^k.$
We now argue similarly to the proof of
Theorem \ref{theo.mixing}
to obtain that for separated tuples $k_1, k_2, \dots, k_r,$
\begin{equation}\label{PoiLawM1}
\mu\left(\bigcap_{j=1}^r \hat\fA_{\fr}^{k_j} \right)=\rho^{rd}\int_M \gamma^r(z) d\mu(z)(1+o(1)).
\end{equation}

Namely, note that
$$ \int 1_{\hat\fA_\fr}(x_0, x_1)\dots 1_{\hat\fA_\fr}(x_0, x_r) d\mu(x_0) d\mu(x_1)\dots d\mu(x_r)$$
$$=
\int \mu^r(B(x_0, \fr)) d\mu(x_0)=\fr^{rd} (1+O(\fr)) \int_M \gamma^r (x_0) d\mu(x_0). $$
Thus approximating $1_{\hat{\fA}_{\fr}}$ by $\hat{A}_{\fr}^\pm$
satisfying ${\rm(\overline{Appr})}$, and applying ${\rm\overline{(EM)}_r}$ to the functions
$$\hat{B}_{\fr}^+(x_0,\cdots,x_r)=\hat{A}_{\fr}^+(x_0,x_1)\cdots \hat{A}_{\fr}^+(x_0,x_r),$$
$$\hat{B}_{\fr}^-(x_0,\cdots,x_r)=\hat{A}_{\fr}^-(x_0,x_1)\cdots \hat{A}_{\fr}^-(x_0,x_r),$$
we get that if $k_{j+1}-k_j>R|\ln\rho|$ for all $0\leq j\leq r-1,$ then
$$ \mu\left(\bigcap_{j=1}^r \hat\fA_{\fr}^{k_j} \right)\leq \mu\left(\hat{B}_{\fr}^+(x_0,f^{k_1}x_0,\cdots,f^{k_r}x_0)\right)\leq \mu\left(\hat{B}_{\fr}^+(x_0,\cdots,x_r)\right)+C{\rho}^{-r\sigma }\theta^{R|\ln\rho|}$$
$$\leq\left(\rho^d+C\rho^{d(1+\eta)}\right)^r\int_M \gamma^r(z) d\mu(z)
+C{\rho}^{-r\sigma }\theta^{R|\ln\rho|},$$
and, likewise,
$$\mu\left(\bigcap_{j=1}^r \hat\fA_{\fr}^{k_j} \right)\geq\left(\rho^d-C\rho^{d(1+\eta)}\right)^r\int_M \gamma^r(z) d\mu(z)
-C{\rho}^{-r\sigma }\theta^{R|\ln\rho|}.$$
Taking $R$ large we obtain \eqref{PoiLawM1}.

Summing \eqref{PoiLawM1}
 over all well separated couples with $k_j\leq \tau \fr^{-d}$ and using that the contribution of non-separated
couples is negligible due to $(M2)_r$ we obtain
\begin{equation*}
\label{RetMom}
 \lim_{\fr\to 0} \int_M \left(\begin{array}{c} N_{\fr, \tau, x} \\ r \end{array} \right) d\mu(x)=
\int_M \frac{\lambda^r(z)}{r!} d\mu(z)
\end{equation*}
where
$$ N_{\fr, \tau, x}=\Card\left\{k\leq \tau \fr^{-d}: d(x, f^k x)\leq \fr\right\}. $$
Since the RHS coincides with factorial moments of the Poisson mixture from \eqref{MixPois},
the result follows.
\end{proof}

\subsection{  Notes.}
Early works on Poisson Limit Theorems for dynamical systems include
\cite{Co00, Doe40, Hir93, Hir95, HSV99, Pit}. \cite{ChCol,  HP14, HW16, PS16} prove
Poisson law for visits to balls centered at a {\it good} point for nonuniformly hyperbolic dynamical systems and show that
the set of good points has a full measure. \cite{Dolgopyat} obtains Poisson Limit Theorem
for partially hyperbolic systems.
Some of those papers, including \cite{CFFHN, Dolgopyat, HP14, HV19} show that in various settings is
the hitting time distributions are Poisson for {\bf all} non-periodic points
(cf. our Conjecture \ref{ConjAP-Pois}).
The rates of convergence under appropriate mixing conditions are discussed in
\cite{Ab08, AV08, HV04}. The Poisson limit theorems for flows are obtained in
\cite{Mar17, PY17}. Convergence on the level of random measures where one records some
extra information about the close encounters, such as for example, the distance of approach
is discussed in \cite{Dolgopyat, FFM18, FFM20}. A mixed exponential distribution for a return time
for dynamical systems
similar to Corollary \ref{CrExponential} has been obtained in \cite{CGS99} in a symbolic setting.
For more discussion of the distribution of the entry times to small measure sets we refer the readers to
\cite{Col01, Keller, Saussol09,  Zh16} and references wherein.
We also refer to Section \ref{ScExtreme}
for the related results in the context of extreme value theory.

\section{  Gibbs measures on the circle: Law of iterated logarithm for {recurrence and} hitting times}
\label{ScGibbs}

\subsection{  Gibbs measures.}\label{SscGibbs}
The goal of this section is to show how absence of the hypothesis of  smoothness on the invariant measure $\mu$ may also alter the law of multiple recurrence and hitting times.

For simplicity we consider the case where $f$ is an expanding map of the circle $\mathbb{T}$ and $\mu$ is a Gibbs measure with Lipschitz potential $g.$ Adding a constant to $g$
if necessary we may and will assume in all the sequel that the topological pressure of $g$ is $0$, that is
\begin{equation}
\label{0Pressure}
P(g)=\int g d\mu+h_\mu(f)=0.
\end{equation}

This means (see \cite{Sinai72} for background on Gibbs measures)
that for each $\eps>0$ there is a constant $K_\eps$ such that
if $B_n(x, \eps)$ is the Bowen ball
$$ B_n(x,\eps)=\{y: d(f^k y, f^k x)\leq \eps \text{ for }k=0,\dots, n-1\},$$
then
$$K_\eps^{-1}\leq \frac{\mu(B_n(x, \eps))}{\exp\left[\left(\sum_{k=0}^{n-1} g(f^k x)\right)\right]}
\leq K_\eps.$$

We denote
\begin{equation}
\label{SRBPot}
f_u=\ln  |f'|,
\end{equation}
$\lambda=\lambda(\mu)$ the Lyapunov exponent of $\mu$
$$ \lambda = \lim_{n\rightarrow \infty} \frac{\ln |(f^n)' (x)|}{n} = \int f_u d\mu >0, $$
and by $\dd$ the dimension of the measure $\mu$
$$ \dd= \lim_{\delta\rightarrow 0} \frac{\ln \mu (B(x,\delta))}{\ln \delta}.$$

We know from \cite{LedrappierYoung85} that the limit exists for $\mu$-a.e. $x$ and
$$\dd=h_\mu (f)/\lambda=-\frac{\int g d\mu}{\int f_u d\mu}$$
where the last step relies on \eqref{0Pressure}.

We say that $\mu$ is {\it conformal} if there is a constant
$K$ such that for each $x$ and each $0<r\leq 1,$
\begin{equation*} \label{eq.conformal} K^{-1}\leq \frac{\mu(B(x, r))}{r^\dd}\leq K. \end{equation*}
It is known (see e.g. \cite{PP90})
that $\mu$
is conformal if and only if  $g$ can be represented in the form
$$ g=t f_u-P(t f_u)+\tilde{g}-\tilde{g}\circ f $$
for some H\"older function $\tilde g$ and $t\in\RR.$

Denote
\begin{equation}
\label{PsiPot}
\psi(x)=g(x)+\dd f_u(x),
\end{equation}
then we have $ \int \psi d\mu=0$ under the assumption $P(g)=0.$ Define $\sigma=\sigma(\mu)$ by the relation
\begin{equation}\label{GibVar}
\sigma^2=\int \psi^2 d\mu+2 \sum_{n=1}^\infty \int \psi\left(\psi\circ f^n\right) d\mu.
\end{equation}

The goal of this section is to prove the following

\begin{Thm}
\label{ThRecGibbs}

(a) If $\mu$ is conformal then Theorems \ref{thm3} and \ref{ThTop} remain valid with $d$ replaced by $\dd.$

  (b)   If $\mu$ is {\bf not} conformal then
   for $\mu$ almost every $x$ and
 $\mu\times \mu$ almost every $(x,y)$, it holds that
  \begin{align} \label{GibbsRec}
  \lim\sup_{n\to\infty} \frac{|\ln d_n^{(r)}(x,x)|-\frac{1}{\dd} \ln n}{\sqrt{2(\ln n) (\ln \ln \ln n)}}&=\frac{ \sigma}{\dd\sqrt{\dd\lambda}}, \\
  \label{GibbsHit}
  \lim\sup_{n\to\infty} \frac{|\ln d_n^{(r)}(x,y)|-\frac{1}{\dd} \ln n}{\sqrt{2(\ln n) (\ln \ln \ln n)}}&=\frac{ \sigma}{\dd\sqrt{\dd\lambda}}.
  \end{align}
\end{Thm}

{\subsection{  Preliminaries on expanding circle maps and their Gibbs measures.}
\label{SSGibbsPrep}
Here we prepare for the proof of Theorem~\ref{ThRecGibbs} by
collecting some facts on expanding maps of the circle and their Gibbs measures.

We first check multiple mixing for such maps.

Recall we take $\BAN={\rm Lip}.$ Let us denote by $\|\cdot\|_{\rm Lip}$ the Lipschitz norm
$$\|\phi\|_{\rm Lip}=\int|\phi|d\mu+\sup_{x,y\in\mathbb{T}}\frac{|\phi(x)-\phi(y)|}{d(x,y)}$$
for $\phi\in \BAN.$

\begin{Prop}
\label{PrGibbsMix}
For each Gibbs measure $\mu$, the system {$(f,\mathbb{T},\mu,\BAN)$} is $r$-fold exponentially mixing for any $r\geq 2$.
\end{Prop}

This fact is well known but for the reader's convenience we provide the argument in
 in \S \ref{AppMixGibbs}.

In the rest of the argument it will be important that if $\mu$ is a Gibbs measure then there are
positive constants
$a, b$ such that for all sufficiently small $\rho$ and for all $x,$
\begin{equation}
\label{GibbsSqueeze}
\rho^a\leq \mu(B(x,\rho))\leq \rho^b.
\end{equation}

We also need the fact that Gibbs measures are Alhfors regular, that is there is a constant $R$
such that for each $x, \fr$ we have
\begin{equation}
\label{Alhfors}
 \mu\left(B(x, 4 \fr)\right)\leq R \mu(B(x, \fr)).
\end{equation}
We recall the proofs of \eqref{GibbsSqueeze} and \eqref{Alhfors} in
\S \ref{SSGibbsReg}.

We also need a lemma on the fluctuations of
the  local dimension of Gibbs measures for expanding circle maps.

\begin{Lem}
\label{LmDimFl} $\ $

\noindent (a) $\sigma(\mu)=0$ if and only if $\mu$ is conformal.

\noindent   (b) If $\sigma>0$ then for $\mu$ almost every $x$
$$
\limsup_{\delta\to 0} \frac{|\ln \mu\left( B(x, \delta)\right)|-\dd |\ln \delta|}{\sqrt{2  |\ln \delta| (\ln \ln |\ln \delta|) }}  =   \frac{\sigma}{\sqrt{\lambda}}, \quad
\liminf_{\delta\to 0} \frac{|\ln \mu \left( B(x, \delta)\right)|-\dd |\ln \delta|}{\sqrt{  2|\ln \delta| (\ln \ln |\ln \delta|) }} =    -\frac{\sigma}{\sqrt\lambda}.
$$
\end{Lem}
The proof of this lemma is also given in
Appendix \ref{AppLILocDim}.

{\subsection{  The targets.} Given $x\in M,$  let
\begin{equation*}
\label{BallTarget2}
\fA_{x,\fr}=\{y:\,d(x,y)\leq\rho\},\quad
\bfA_{\fr}=\left\{(x,y):d(x,y)\leq \rho
\right\}.
\end{equation*}

  We use the notation $\fA_{x,\fr}^k$ for the event $1_{\fA_{x,\fr}}\circ f^k$.
We also recall the notation  $ {\bfA}^k_{\fr}=\{x: (x, f^k x)\in{\bfA}_{\fr}\}.$
In the sequel we will always assume that $\{\fr_n\}$ is a sequence such that $\fr_n>n^{-u}$ for some $u$.

We caution the reader that the targets $\bfA_\fr$  are {not} admissible targets in the non-conformal
case, so we need to use a roundabout approach, different from Section \ref{ScHit},
for proving Theorem \ref{ThRecGibbs}(b).

On the other hand, we will need a modification of the argument of Lemma \ref{lemma.mov} to show that for any Gibbs measure $\mu$ and for $\mu$-a.e.
$x\in M$,   the targets
$\fA_{x,\fr_n}$ are admissible for $(f,M,\mu,\BAN)$. The difference with the case of smooth measures, is that it does not hold anymore that $\mu(B(y,1/2^j))\leq C\mu(B(x,1/2^j))$ for any $x,y \in M$, while this was used in the proof of Lemma \ref{lemma.mov}.


\begin{Lem} \label{lem.gibbs} For any Gibbs measure $\mu$, for $\mu$-a.e. $x\in M$,   the targets
$\fA_{x,\fr_n}$ are admissible for $(f,M,\mu,\BAN)$.
\end{Lem}
 \begin{proof}Due to \eqref{GibbsSqueeze} and \eqref{Alhfors},
all the properties of admissible targets
 except for (Mov) are obtained exactly as in the smooth measure case. To prove
 (Mov), we modify the argument of Lemma \ref{lemma.mov} to overcome the fact that it does not hold anymore that $\mu(B(y,1/2^j))\leq C\mu(B(x,1/2^j))$ for any $x,y \in M$.

In fact we can prove  more than (Mov) in this context of expanding circle maps. Namely we can show that for a.e.
$x$ and all $k$
\begin{equation}
\label{GibbsMov}
 \mu(B(x,\fr)\cap f^{-k} B(x, \fr))\leq \mu(B(x, \fr))^{1+\eta}.
 \end{equation}
We consider two cases.

(I) $k>\eps|\ln \fr|$
where $\eps$ is sufficiently small (see case (II) for precise bound on $\eps$).
Take $A^+_\fr$ such that
$A^+_\fr=1$ on $B(x, \fr),$  $\int A^+_\fr d\mu\leq 2 \mu(B(x,\fr))$ and
$\|A_\fr^+\|_{Lip}\leq C \fr^{-\tau}$ for some $\tau=\tau(\mu).$ Let $\hat\fr=\fr^\sigma$ where
$\sigma$ is a small constant.
Then \eqref{PWM2Gibbs} gives

$$ \mu(B(x,\fr)\cap f^{-k} B(x, \fr))\leq \int A^+_{\hat\fr} (A^+_\fr\circ f^k) d\mu $$
$$\leq
4 \mu(B(x, \fr)) \mu(B(x, \hat\fr))
+2C\brtheta^k \hat\fr^{-\tau} \mu(B(x, \fr))\leq C \mu(B(x, \fr))
\left(\fr^{\sigma b} +\fr^{\eps|\ln\brtheta|}  \fr^{-\tau \sigma}\right)$$
for some $0<\brtheta<1.$
Taking $\sigma$ small we can make the second term smaller than $\fr^{\eps|\ln\brtheta|/2} $
which is enough for $\rm{(Mov)}$ in view of already established $\rm{(Poly)}.$
Note that no restrictions on $x$ are imposed in case (I).

(II) $k\leq\eps|\ln \fr|.$ In this case for a.e. $x$ the intersection
$B(x, \fr)\cap f^{-k} B(x, \fr)$ is empty for small $\fr$ due to the Proposition \ref{PrBarSaus} below.
\end{proof}

\begin{Prop}
\label{PrBarSaus}
(\cite[Lemma 5]{BarreiraSaussol01})
Let $T:\,X\to X$ be a Lipschitz map with Lipschitz constant $L>1$ on a compact metric space $X.$ If $\mu$ is an ergodic measure with $h_\mu(T)>0.$ Then for almost every $x,$ there exists $\fr_0(x)>0$ such that for all
$\fr\leq \fr_0(x),$ and all $0<k\leq \frac{1}{2L}|\ln \rho|,$ we have
$\DS T^{-k} B\left(x, \rho\right) \cap B\left(x, \rho\right)=\emptyset.$
\end{Prop}
}

The case of composite targets  $\bfA_\fr$ is more complicated, except for the conformal case.

In the conformal case,
the following Lemma is obtained exactly as in Proposition \ref{prop.targets} that dealt with the smooth measure case, so we omit its proof.

\begin{Lem} \label{lem.conformal} If $\mu$ is conformal, then the targets
$\bfA_{\fr_n}$ defined by \eqref{VaryingWidthTarget} are weakly admissible in the sense of Remark
\ref{RkWeakAdm}.
\end{Lem}


\subsection{  The conformal case.}\begin{proof}[Proof of Theorem \ref{ThRecGibbs} (a).]
 We  take  $\fr_n=n^{-1/{\bd}}   \ln^{-s}n$. Due to Lemmas \ref{lem.gibbs} and \ref{lem.conformal}, the targets  targets
$\fA_{x,\fr_n}$ are admissible for $\mu$-a.e. $x\in M$ and the targets
$\bfA_{\fr_n}$ are composite weakly admissible. Consequently, the proof of  Theorem \ref{ThRecGibbs} (a) follows exactly as that of  Theorems \ref{thm3} and \ref{ThTop} corresponding to the smooth measure case. \end{proof}

 \subsection{  The non conformal case. Proof of Theorem \ref{ThRecGibbs} (b).}

The proof of Theorem \ref{ThRecGibbs} (b) relies on the $\liminf$ in
Lemma \ref{LmDimFl}(b).

\subsubsection{  The iterated logarithm law for hitting times : Proof of \eqref{GibbsHit} of Theorem \ref{ThRecGibbs} (b).}

For $\eps>0$ and $c>0$ arbitrary  let
\begin{align}
\label{vth1} \rho_n&=\rho_n(c)=\frac{1}{n^{1/\dd}}\exp\left(-c\sqrt{2(\ln n) (\ln \ln \ln n)}\right).\\
\nonumber 
 \vartheta_\eps^\pm(\delta)&=\delta^\dd
\exp\left(\left(1\pm \varepsilon\right)\frac{\sigma }{\sqrt{\lambda}}\sqrt{2|\ln \delta|\left(\ln\ln|\ln \delta|\right)}\right),\\
\nonumber \label{vth3} \tilde{\vartheta}_{\eps,c}^\pm(n)&=\vartheta_\eps^\pm(\rho_n(c)).
\end{align}
then
$$\tilde{\vartheta}_{\eps,c}^\pm(n)=\frac{1}{n}\exp\left(\left(-c\dd+(1\pm \varepsilon)\frac{\sigma }{\sqrt{\dd\lambda}}+\eta_n\right)\sqrt{2\ln n (\ln \ln \ln n)}\right)$$
for some $\eta_n\rightarrow 0$ as $n\rightarrow\infty.$

  The $\liminf$ in Lemma \ref{LmDimFl}, has the following straightforward consequences, for any $\varepsilon>0$ and  for $\mu$ almost every $x$:
 \begin{enumerate}
   \item[] There exists $n(x)$ such that for $n\geq n(x)$, we have
 \begin{equation}  \label{eup} \mu \left(\Omega_{x,\rho_n}\right) \leq \tilde{\vartheta}_{\eps,c}^+(n).  \end{equation}
 \item[]  For a subsequence $n_l \to \infty$ we have
   \begin{equation} \label{elow}\mu \left(\Omega_{x,\rho_{n_l}}\right) \geq \tilde{\vartheta}_{\eps,c}^-(n_l).  \end{equation}
 \end{enumerate}
Now it follows that for any $r\geq 1$,
 $\DS S_r=\sum_{k=1}^{\infty}\left(2^k\mu
 (\Omega_{x,\rho_{2^k}})\right)^r$ is finite if $c>(1+\varepsilon)\frac{ \sigma}{\dd\sqrt{\dd\lambda}}$ and is infinite if $c<(1-\varepsilon)\frac{ \sigma}{\dd\sqrt{\dd\lambda}}.$
Hence \eqref{GibbsHit}   follows from Proposition \ref{PrGibbsMix}, Lemma \ref{lem.gibbs} and Corollary \ref{cor.mixing} \hfill $\Box$


\subsubsection{  The iterated logarithm law for return times: Proof of the upper bound in \eqref{GibbsRec}}

Now we turn to the proof of
\begin{equation}\label{GibbsRecleq}
\lim\sup_{n\to\infty} \frac{|\ln d_n^{(r)}(x,x)|-\frac{1}{\dd} \ln n}{\sqrt{2(\ln n) (\ln \ln \ln n)}}\leq\frac{ \sigma}{\dd\sqrt{\dd\lambda}}.
\end{equation}
Since $d_n^{(r)}(x,x)\geq d_n^{(1)}(x,x),$ we only need to show \eqref{GibbsRecleq} for $r=1.$

Denote
$$r_n=\frac{1}{n^{1/\dd}}\exp\left\{-(1+2\varepsilon)\frac{ \sigma}{\dd\sqrt{\dd\lambda}}\sqrt{2(\ln n)
(\ln \ln \ln n)}\right\}.$$
Let $N_k=2^k.$ Similarly to Section \ref{ScBCMult}
it is enough to show that for almost all $x,$ for all sufficiently large $k$ we have that
$$ d(x, f^m x)\geq r_{N_k}\text{ for } m=1, \dots, N_k. $$
Proposition \ref{PrBarSaus} allows us to further restrict the range of $m$ by assuming
$m\geq \breps \ln N_k,$ where $\breps$ is sufficiently small.

We say $x\in\mathbb{T}$ is $n-$good if $\mu\left(B(x,r_n)\right)\leq\vartheta^+(r_n).$
Fix $k_0$ and let
$$ \cA_k=\{x: x\text{ is $n-$good for } n\geq N_{k} \text{ but }
d(x, f^m x)\leq r_{N_k}\text{ for some } m=\breps\ln N_k, \dots, N_k\}.$$
Let $\cX_k= \{x_{j,k}\}_{j=1}^{l_k}$ to be a maximal $r_{N_k}$ separated set  of $N_k-$good points.
Thus if $x$ is $N_k$ good then there is $j$ such that $x\in B(x_{j,k}, r_{N_k}).$
Therefore if $f^m x\in B(x, r_{N_k})$ then $f^m x\in B(x_{j,k}, 2 r_{N_k}).$
Fix a large $K,$ for $m\leq K \ln N_k,$ \eqref{GibbsMov} is
telling us that
$$ \mu\left(B(x_{j,k}, 2 r_{N_k}) \cap f^{-m} B(x_{j,k}, 2 r_{N_k})\right)\leq K \mu(B(x_{j,k}, 2 r_{N_k}))^{1+\eta}. $$
while for
$m>K\ln N_k$ we get by exponential mixing that
$$ \mu\left(B(x_{j,k}, 2 r_{N_k}) \cap f^{-m} B(x_{j,k}, 2 r_{N_k})\right)\leq K \mu(B(x_{j,k}, 2 r_{N_k}))^2. $$
Summing those estimate for  we obtain
$$ \sum_{m=\breps\ln N_k}^{N_k}  \mu\left(B(x_{j,k}, 2 r_{N_k}) \cap f^{-m} B(x_{j,k}, 2 r_{N_k})\right)
\leq K \mu(B(x_{j,k} , 2 r_{N_k})) e^{-\kappa\sqrt{k}} $$
for some $\kappa=\kappa(\breps)>0.$ Since $B(x_{j,k}, r_{N_k}/2)$ are disjoint for
different $j,$ by \eqref{Alhfors} we conclude that
$$ \sum_j \mu\left(B(x_{j,k}, 2r_{N_k})\right)\leq
R \sum_j \mu\left(B(x_{j,k}, r_{N_k}/2)\right)\leq R. $$
It follows that
$$ \mu(\cA_k)\leq K R e^{-\kappa \sqrt{k}}. $$
Now the result follows from the classical Borel Cantelli Lemma.

\subsubsection{The law of iterated logarithm for return times: Proof of the lower bound in \eqref{GibbsRec}.}

Here we prove that
\begin{equation}\label{GibbsRecleq2}
\lim\sup_{n\to\infty} \frac{|\ln d_n^{(r)}(x,x)|-\frac{1}{\dd} \ln n}{\sqrt{2(\ln n) (\ln \ln \ln n)}}\geq\frac{ \sigma}{\dd\sqrt{\dd\lambda}}.
\end{equation}

Suppose $p$ to be a fixed point of $f.$ Take the Markov partition $\mathcal{P}_n$ of $\mathbb{T}$ such that if $P_n\in\mathcal{P}_n,$ then $f^n(\partial P_n)=p.$ Denote $P_n(x)=\{P_n\in\mathcal{P}_n: x\in P_n\}$, two sequences $k_j(x)$ and $n_j(x),$ $j\in\mathbb{N}$ such that $k_0(x)=n_0(x)=0,$
$$n_j(x)=\min\left\{n>k_{j-1}(x)^2:\mu\left(P_n(x)\right)\geq \vartheta_\varepsilon^-\left(\left|P_n(x)\right|\right)\right\}$$
and $$k_j(x)=\frac{2}{\mu(P_{n_j}(x))}.$$
Let
$$\mathcal{A}_j=\left\{x:\text{Card}\left\{k_{j-1}(x)\leq k\leq k_j(x):
f^kx\in P_{n_j}(x)\right\}\geq r\right\}.$$
Then $$
\lim\sup_{n\to\infty} \frac{|\ln d_n^{(r)}(x,x)|-\frac{1}{\dd} \ln n}{\sqrt{2(\ln n) (\ln \ln \ln n)}}\geq\frac{ \sigma}{\dd\sqrt{\dd\lambda}}(1-2\varepsilon)
$$
if $x$ belongs to infinitely many $\cA_j$s.

Denote by  $\Prob\left(\cdot|\cdot\right)$ the conditional probability and $\mathcal{F}_j=\mathcal{B}\left(\mathcal{P}_{k_1},\cdots,\mathcal{P}_{k_j}\right)$
be the $\sigma$ algebra generated by the itineraries up to the time $k_j.$
We will use the following L\'{e}vy's extension of the Borel-Cantelli Lemma.

\begin{Thm}
(\cite[\S 12.15]{Wil91}) If
$\DS\sum_j \Prob\left(\mathcal{A}_{j+1}|\mathcal{F}_j\right)=\infty \quad a.s. $ then
$\mathcal{A}_j$ happen infinitely many times almost surely.
\end{Thm}
Hence \eqref{GibbsRecleq2} follows from the lemma below.
\begin{Lem}
There exists $c^*>0,$ such that for almost all $x$ there is $j_0=j_0(x)$ such that
$\Prob\left(\mathcal{A}_{j+1}|\mathcal{F}_j\right)\geq c^*$
for all $j\geq j_0.$
\end{Lem}
\begin{proof}
For any $\Omega\subset\mathbb{T},$ $P_k\in\mathcal{P}_k,$
$$\mu \left(f^k(\Omega\cap P_k)\right)
=\frac{\mu \left(f^k(\Omega\cap P_k)\right)}{\mu\left( f^k(P_k)\right)}
\leq C\frac{\mu(\Omega\cap P_k)}{\mu(P_k)}
$$
by bounded distortion property. Note that
$$\Prob\left(\mathcal{A}_{j+1}|\mathcal{F}_j\right)(x)=
\frac{\mu\left(\mathcal{A}_{j+1}\cap P_{k_j(x)}(x)\right)}{\mu\left(P_{k_j(x)}(x)\right)}
\geq C^{-1}\mu\left(f^{k_j(x)} \left(\mathcal{A}_{j+1}\cap P_{k_j(x)}(x)\right)\right).
$$
By construction
$$ f^{k_j(x)} \left(\mathcal{A}_{j+1}\cap P_{k_j(x)}(x)\right) $$
is the set of points $y\in \T$ which visit $P_{n_{j+1}(x)}$ at least $r$ times before time
$$\brk_{j+1}(x)=k_{j+1}(x)-k_j(x).$$
By Lemma \ref{lem.gibbs}
for almost all $x$ the targets $P_{n_{j+1}}(x)$ satisfy $(M1)_r$ and $(M2)_r$
for all $r.$ Since by construction
$\DS \lim_{j\to\infty} \mu(P_{n_j}(x)) \brk_j(x)=2$ we can apply
Theorem~\ref{ThPoisson} to get
$$\Prob\left(\mathcal{A}_{j+1}|\mathcal{F}_j\right)(x)\geq
C^{-1}\mu\left(f^{n_j}(\mathcal{A}_{j+1}\cap P_{n_j})\right)\geq C_0\sum_{k=r}^\infty e^{-2}\frac{2^k}{k!}:=c^*.$$
proving the lemma.
\end{proof}

\subsection{  Notes.}
The fact that return times for the non-conformal Gibbs measures are dominated by fluctuations
of measures of the balls has been explored in various settings
\cite{BT09, BV03, ChU05, CGS99,  HV-Fluct, Ib62, Pac00, Saussol01, Var08}.
In particular, \cite{Haydn12} obtains a result similar to our Lemma \ref{LmDimFl} in the context of
symbolic systems. The papers mentioned above deal with either one dimensional or symbolic systems.
In higher dimensions even the leading term of $\ln \mu(B(x,r))$ is rather non-trivial and is analyzed in
\cite{BPS99}, while fluctuations are determined only for a limited class of systems
\cite{LS12}. Thus extending the results of this section to higher dimension is an interesting open
problem.

\section{  Geodesic excursions.}
\label{ScExcursions}
\subsection{  Excursions in finite volume hyperbolic manifolds.}
Let $\cQ$ be a finite volume non-compact $(d+1)$-dimensional manifold of curvature $-1$.
Let $S\cQ$ denote the unit tangent bundle to $\cQ$.
For $(q,v)\in S\cQ,$ let $\gamma(t)=\gamma(q,v,t)$ be the geodesic such that $\gamma(0)=q,$ $\dot\gamma(0)=v.$ We call $g^t$ the corresponding geodesic flow, that is $g^t(\gamma(0),\dot\gamma(0))=(\gamma(t),\dot\gamma(t))$.
$g^t$ preserves the Liouville measure $\mu.$ Fix a reference point $O\in \cQ$ and let
$D(q,v, t)=\dist(O, \gamma(t)).$
According to Sullivan's Logarithm Law for excursions \cite{Sul} for  $\mu$-a.e. $(q,v)\in S \cQ$, it holds that
\begin{equation}
\label{SulLawExc}
 \lim\sup_{T\to\infty} \frac{D(q, v, T)}{\ln T}=\frac{1}{d}.
\end{equation}
In fact, the Borel Cantelli Lemma of \cite{Sul}
 also shows
that
\begin{equation}
\label{SulMultiLawExc}
 \lim\sup_{T\to\infty} \frac{D(q, v, T)-\frac{1}{d}\; \ln T}{\ln \ln T}=\frac{1}{d}.
\end{equation}
Here we present a multiple excursions version of \eqref{SulMultiLawExc}.
Recall (\cite[Proposition D.3.12]{BP92})
that $\cQ$ admits a decomposition $\DS \cQ=\cK\bigcup \left(\bigcup_{j=1}^p \cC_j \right)$
where $\cK$ is a compact set and $\cC_j$ are cusps. Moreover each cusp is isometric to
$V_i\times [L_j, \infty) $ endowed with
the metric
$$ ds^2=\frac{dx^2+dy^2}{y^2}, \quad x\in V_j, \;\; y\in [L_j, \infty)$$
where $V_j$ is a compact flat manifold and $dx$ is the Euclidean metric on $V_j.$
Cusps are disjoint, so that a geodesic cannot pass between different cusps without visiting the
thick part $\cK$ in between. We note that\footnote{We identify hereafter each cusp $\cC_j$  with $V_i\times [L_j, \infty)$.} for each $q_0=(x_0, y_0)\in \cC_i$ there is a unique geodesic
($\{x=x_0\}$) which remains in the cusp for all positive time. We will call this geodesic
{\it the escaping geodesic} passing through $(x_0, y_0).$
Let $h(q, v, t)=0$ if $\gamma(q, v, t)\in \cK$ and $h(q, v, t)=\ln y(t)$ if $\gamma(q, v, t)=(x(t), y(t)) \in \cC_i$. It is easy to see using the triangle inequality that there exists a constant $C$ such that
$$ \left|D(q, v, t)-h(q, v,t)\right|\leq C. $$
{\it A geodesic excursion} is a maximal interval $I$ such that $\gamma(t)$ belongs to some cusp
$\cC_i$ for all $t\in I.$ Then, $\DS h(I)=\max_{t\in I} h(q, v, t)$ is called {\it the height of the excursion $I$}.
For every triple $(q,v,T)$ we can order the heights  of the excursions that correspond to maximal excursion intervals included inside $(0,T)$ starting from the highest one
$$ H^{(1)}(q, v, T)\geq
H^{(2)}(q, v, T)\geq \dots \geq H^{(r)}(q, v, T)\dots $$

Note that \eqref{SulMultiLawExc} is implied by
\begin{equation}
\label{MultiLawExcH1}
 \lim\sup_{T\to\infty} \frac{H^{(1)} (q, v, T)-\frac{1}{d}\; \ln T}{\ln \ln T}=\frac{1}{d}.
\end{equation}
Here we prove the following multiple excursions version of  \eqref{MultiLawExcH1}.

\begin{Thm}
\label{ThMultiExc}
For a.e. $(q,v)$ and all $r$ we have
\begin{equation*}
\label{MultiLawExcHr}
 \lim\sup_{T\to\infty} \frac{H^{(r)} (q, v, T)-\frac{1}{d}\; \ln T}{\ln \ln T}=\frac{1}{rd}.
\end{equation*}
\end{Thm}

We also have the following byproduct of our analysis.
\begin{Cor}
\label{CrPoisExc}
There is a constant $a_i$ such that for each $\fh$ the following holds.
Suppose that $(q,v)$ is uniformly distributed on $S\cQ.$
Then the number of excursions in the cusp $\cC_i$ which finished before time $T$ and reached the height
$\frac{\ln T}{d}+\fh$ is asymptotically Poisson with parameter $a_i e^{-d\fh}.$
\end{Cor}

In other words, for every $r\geq 1$, we have
\begin{equation}
\label{PoisProcExc}
 \lim_{T\to\infty} \mu\left(H^{(r)}_i(q,v, T)<\frac{\ln T}{d}+\fh\right)=\sum_{l=0}^{r-1}
\frac{(a_i e^{-d\fh})^l}{l!}\exp\left(-a_i e^{-d\fh}\right).
\end{equation}
In particular, taking $r=1$ in \eqref{PoisProcExc} we obtain

\begin{Cor}
\label{CrGeoGumbel}
{\it (Gumbel distribution for the maximal excursion)}
If $(q,v)$ is uniformly distributed on $S\cQ.$ Let $ H^{(1)}_i(q,v, T)$ denote the maximal height
reached by $\gamma(q, v, t)$ up to time $T$ inside cusp $\cC_i.$
Then
$$ \lim_{T\to\infty} \mu\left(H^{(1)}_i(q,v, T)-\frac{\ln T}{d}<\fh\right)=\exp\left(-a_i e^{-d\fh}\right).
$$
\end{Cor}

\subsection{  MultiLog law for geodesic excursions}

In this section we prove  Theorem \ref{ThMultiExc}. We first need to discuss the probability of having an excursion reaching a given level.
To this end let $\Pi$ be the plane passing through $\gamma$ and the escaping geodesic. In this plane the geodesics are half circles
centered at the absolute $\{y=0\}.$  The half circle (geodesic) given by $(x-x_0)^2+y^2=R^2$ reaches the maximum height of $\ln R+O(1).$ Let
$n^*$ be the first integer moment of time after the beginning of the excursion. Then the $y$ coordinate of $\gamma(n^*)$ is uniformly bounded from
above and below so the radius of the circle defining the geodesic is given by $\DS R=\frac{y(n^*)}{\sin \theta}$ where $\theta$ is the angle with
the escaping geodesic. It follows that the condition $R\geq R_0$ is equivalent to the condition $\sin\theta\leq \frac{y(n^*)}{R_0}.$

\begin{Def} \label{defAH} Given $H$ we consider the set $\cA_{i,H}$ which consists of points $(q,v)\in \cC_i$ such that

(i) The first positive time $\bar {t}(q,v)$ such that the backward geodesic $\gamma(q, v, -\bar t)$ exits the cusp satisfies $\bar{t}(q,v)\in [0,1]$;

(ii) The angle $v$ makes with the  escaping  geodesic at $q$ is less than $e^{-H}.$
\end{Def}

{The above discussion implies that for $(q,v) \in \cC_i$ satisfying (i) and (ii), the geodesic starting at $(q,v)$  will exit the cusp in backward time less than $1$ and will do an excursion in future time up to height $h\geq H$, consuming for this a time comparable to $h$.
Moreover
condition (ii) on the angle is a necessary and sufficient condition for the excursion to reach height $H$.}

We also introduce
\begin{equation}
\label{AH}
\cA_H=\bigcup_i \cA_{i,H}.\end{equation}

It is a basic fact (e.g. see the proof of Theorem 6 in \cite{Sul})
\begin{equation}
\label{MesHeight}
\mu(\cA_{i,H})=a_ie^{-dH} (1+o(1)).\end{equation}

{To prove Theorem \ref{ThMultiExc} we define for every $k\geq 0$
\begin{equation}
\label{exc.targets}
\fA_\fr^k=g^{-k}\cA_{-\ln \fr}.
\end{equation}

By a slight abuse of notation, we still denote the event  $1_{\fA^k_{\fr}}$ by $\fA_\fr^k$.  We also keep the notation $\sigma(\fr)=\mu(\fA_{\fr})$.

For $s\geq 0$, we let $\fr_n=n^{-1/d} \ln^{-s}n$, and recall that $N^n_{\fr_n}(q,v)$ denotes the number of times $k\in [1,n]$ such that $\fA^k_{\fr_n}$ occurs (i.e. $(q,v) \in \fA^k_{\fr_n}$).

Theorem \ref{ThMultiExc}  becomes equivalent to the following:
\begin{itemize}
\item[(a)] If $s>\frac{1}{rd}$,  then for $\mu$-a.e. $(q,v)$, we have that  for large $n$,
$N^n_{\fr_n}<r.$
\item[(b)] If $s\leq\frac{1}{rd}$, then for $\mu$-a.e. $(q,v)$, there are infinitely many $n$ such that
$N_{\fr_n}^{\frac{n}{2}}\geq r.$\footnote{We introduce a factor $1/2$ to make sure the last excursion that starts before $n/2$ finishes before $n$. Here, we are using the control on the excursion time that is comparable to the excursions height $\ln n \ll n/2$.}
\end{itemize}}

Observe that by \eqref{MesHeight}, we have that
\begin{equation} \label{mesAH}  \mu(\fA_{\fr_n}) \in \left[C^{-1}n^{-1}\ln^{-sd} n,C n^{-1}\ln^{-sd} n\right]\end{equation}

With the notation $\bS_r=\sum_{j=1}^\infty \left(2^j \bv_j\right)^r$
where $\bv_j=\fv(\fr_{2^j})$ we see from \eqref{mesAH} that
$\bS_r=\infty$ if and only if $s\leq\frac{1}{rd}$. We want thus to apply
Corollary \ref{cor.mixing}, but first we need to verify its conditions.

{The system $(g^1,S\cQ,\mu,\BAN)$ is $r$-fold exponentially mixing for every $r \geq 2$ in the sense of Definition \ref{def.mixing}. Indeed  (Prod) an (Gr) are clear, while
 $(EM)_r$ follows from Theorem 1.1 of \cite{BEG} (see also Theorem 1.2 of \cite{KM12})
 and Remark \ref{lowreg} and Theorem \ref{MultiMixRet} of our appendix. }



To apply Corollary \ref{cor.mixing}, we also need the admissibility of the targets.

\begin{Prop} \label{propadm} The family of targets $\{\Omega_{\rho_n}\}$ is  admissible as in Definition \ref{def.targets}.
\end{Prop}

Before we prove Proposition \ref{propadm}, we first complete the

\begin{proof}[Proof of Theorem \ref{ThMultiExc}.] From the equivalence stated in $(a)$ and $(b)$ above, and since by Proposition \ref{propadm} the targets $\{\Omega_{\rho_n}\}$ are admissible, the $\limsup$ of Theorem \ref{ThMultiExc} follows from Corollary \ref{cor.mixing} and the fact that  $\bS_r=\infty$ if and only if $s\leq\frac{1}{rd}$.
\end{proof}

\begin{proof}[Proof of Proposition \ref{propadm}.] First, the definition of $\rho_n$ and \eqref{mesAH} imply (Poly). Next, the first time  $\bar{t}(q,v)\geq 0$ such that $\gamma(q, v, -t)$ exits the cusp  is Lipschitz in $(q,v)$. Also, the angle $\Psi(q,v)$ that $v$ makes with  the escaping  geodesic  at $q$  is also a Lipschitz function of $(q,v)$. {We conclude that  (Appr) for the targets $\{\Omega_{\rho}\}$ follows from  Lemma \ref{RemLip} with {$\Phi(q,v)=\Psi(q,v),$} $a_1(\rho)=0$ and $a_2(\rho)=\rho$, {\it modulo} a very simple modification in the proof of Lemma \ref{RemLip} to account for the benign extra condition that $\bar{t}(q,v)\in [0,1]$.}

It remains to prove (Mov). We denote $\cA_H^{n}=g^{-n}\cA_H.$ (Mov)  is an immediate consequence from the following quasi-independence result on the excursions.
{Similar quasi-independence results are obtained in \cite{Sul,PauPol16}, and we will give a proof adapted to our setting for completeness.}

\begin{Lem}
\label{LmQIExc}
There is a constant $K$ such that for each $H >0$ and each $n_1< n_2,$
$$  \mu(\cA_H^{n_1} \cap  \cA_H^{n_2})\leq  K \mu\left(\cA_H\right)^2.$$
\end{Lem}

Up to proving Lemma \ref{LmQIExc}, we finished the proof of Proposition \ref{propadm}. \end{proof}

The rest of this section is devoted to the
\begin{proof}[Proof of Lemma \ref{LmQIExc}.]
Let $\tcA_H=I \cA_H$ where $I$ denotes the involution $I(q,v)=(q, -v).$ Given $n_1, \bar{n}$ define
$$ \cB_{H, n_1, \bar{n}}=\{x: g^{n_1}x\in \cA_H, \; g^{\bar{n}}x\in \tcA_H, \; g^n x\not\in \cK
\text{ for } n_1<n<\bar{n}\}. $$
Thus $\cB_{H, n_1, \bar{n}}$ consists of points which enter a cusp at time $n_1,$ reach the height $H,$
and then exit the cusp at time $\bar{n}.$ We have that
$\DS \cA_H^{n_1} = \bigcup_{\bar n > n_1}\cB_{H, n_1, \bar{n}}$.
 Note that $\bar{n}-n_1\geq H$.

Fix a small $\delta$ and let $\tcB_{H, n_1, \bar{n}}=\bigcup_{x\in\cB_{H, n_1, \bar{n}}} \cW^u\left(x, \delta e^{-\bar{n}}\right),$ where $\cW^u\left(x, \rho\right)$ denotes the local unstable cube containing $x$ of length $\rho.$ Note that if $y\in \tcB_{H, n_1, \bar{n}}$ then $g^{\brn_1} y\in \cA_{H-1}$ for some $\brn_1\in [n_1-1, n_1+1]$ and $g^{\brn_2} y\in \tcA_{H-1}$ for some $\brn_2\in [\bar{n}-1, \bar{n}+1].$
In particular for each $n_1$ the sets
$\{\tcB_{H, n_1, \bar{n}}\}_{\bar{n}\geq n_1} $ have at most 3 intersection multiplicity and hence
\begin{equation}
\label{SumCollars}
\sum_{\bar{n}} \mu(\tcB_{H, n_1, \bar{n}}) \leq 3 \mu(\cA_{H-1}).
\end{equation}

{Since $\tcB_{H, n_1, \bar{n}} \cap \cA_H^{n_2} =\emptyset$ if $\bar n \geq n_2$, we have for $n_2>n_1$
\begin{equation}
\label{SumCollars2}
 \mu(\cA_H^{n_1} \cap  \cA_H^{n_2})\leq \sum_{n_1<\bar{n}<n_2}
\mu\left(\tcB_{H, n_1, \bar{n}} \cap \cA_H^{n_2}\right). \end{equation}
We claim that for some constant $C\geq 1$ for each $n_1<\bar{n}< n_2$
\begin{equation}
\label{CondExcursion}
\mu\left(\cB_{H, n_1, \bar{n}} \cap \cA_H^{n_2} \right)\leq C \mu(\tcB_{H, n_1, \bar{n}})\mu(\cA_H).
\end{equation}
Now, \eqref{MesHeight}, \eqref{SumCollars}, \eqref{SumCollars2} and
\eqref{CondExcursion}, imply the estimate of Lemma \ref{LmQIExc}.

It remains to establish \eqref{CondExcursion}. To this end, fix a large $\brH$ and partition a
small neighborhood $\cU$ of $\tcA_\brH$ into unstable cubes of size $\delta$ (these are nice unstable cubes around points that are close to the compact region $\cK$).  For $H\geq \bar H$, let
$$\hcB_{H, n_1, \bar{n}}=\bigcup_{x\in B_{H, n_1, \bar{n}}} \cW^u(g^{\bar{n}} x,\delta) $$
where $\cW^u(y,\delta)$ is the element of the above partition containing $y.$
Note that
\begin{equation} \label{tildebar}
\cB_{H, n_1, \bar{n}}\subset g^{-\bar{n}} \hcB_{H, n_1, \bar{n}} \subset \tcB_{H, n_1, \bar{n}}. \end{equation}
Thus
$$ \mu\left(\cB_{H, n_1, \bar{n}} \cap \cA_H^{n_2}\right)\leq
\mu\left(g^{-\bar{n}} \hcB_{H, n_1, \bar{n}} \cap \cA_H^{n_2}\right)
=\mu\left(\hcB_{H, n_1, \bar{n}} \cap  \cA_H^{n^*}\right)$$
where $n^*=n_2-\bar{n}>0$. We thus finish if we show that
\begin{equation}
\label{HB-A-QI}\mu\left(\hcB_{H, n_1, \bar{n}} \cap  \cA_H^{n^*}\right) \leq C
 \mu(\hcB_{H, n_1, \bar{n}})\mu(\cA_H).
\end{equation}
Indeed \eqref{tildebar} and \eqref{HB-A-QI}
imply \eqref{CondExcursion}. By construction, $\hcB_{H, n_1, \bar{n}}$ is partitioned into
nice unstable cubes of size $\delta$. It suffices to show that for any such cube $\cW$ we have
\begin{equation}
\label{W-A-QI}
\mu(\cA_H^{n^*}|\cW)\leq C e^{-dH}
\end{equation}
where $\mu(\cdot|\cdot)$ denotes the conditional expectation. Let
$\DS Q=\bigcup_{x\in \cW} \bigcup_{|t|<\delta}\cW^s(g^t x,\delta),$ where $\cW^s\left(y,\delta\right)$ denotes the local stable leaf containing $y$ of length $\delta.$
Note that if $\delta$ is sufficiently small then due to the local product structure, for
each point $y\in Q$  there is unique $x\in \cW$ and $t\in [-\delta, \delta]$ such that
$y\in \cW^s(g^t x,\delta)  .$ In addition if $g^{n^*} x\in \cA_H$ then
$g^{n^*} y\in \cA_{H-1}.$ Since the measure of $Q$ is bounded  from below uniformly in
$\cW\subset \cU$,
it follows that
$$ \mu(\cA_H^{n^*}|\cW)\leq \mu(\cA_{H-1}^{n^*}|Q)=\frac{\mu(\cA_{H-1}^{n^*}\cap Q)}{\mu(Q)}
\leq \brc \mu(\cA_{H-1})\leq \hc e^{-dH}.$$
This establishes \eqref{W-A-QI} and, hence \eqref{HB-A-QI} completing the proof of
Lemma \ref{LmQIExc}.}
\end{proof}

\subsection{  Poisson Law for excursions. Proof of Corollary \ref{CrPoisExc}}
Here we take \begin{equation*} \label{eq.rho_n2} \fr_n := n^{-1/d}.\end{equation*}
{We fix $\fh \in \R$ and fix a cusp index $i$.
With the sets $\cA_{i,H}$ defined as in Definition \ref{defAH}, consider the targets
\begin{equation*}
\label{exc.target2}\fA^k_{i,\fr_n}=g^{-k} \cA_{i,-\ln \fr_n-\fh}.\end{equation*}

As in the proof of Theorem \ref{ThMultiExc}, we have  that
$\{\fA^{k}_{i,\fr_n}\}$ satisfies the assumptions $(M1)_r$ and $(M2)_r$ for all $r.$ Moreover,
by \eqref{MesHeight}
$$ \lim_{n\to \infty} n \mu(\fA_{i,\fr_n})=a_ie^{-d\fh}.$$}
Therefore Corollary  \ref{CrPoisExc} follows from Theorem \ref{ThPoisson}. \hfill $\Box$

\subsection{  Notes.}
The logarithm law for the highest excursion was proven in \cite{Sul}.
The extensions for infinite volume hyperbolic manifolds is studied in \cite{SV95}.
Corollary \ref{CrGeoGumbel} for surfaces is obtained in \cite{JKS13} where the authors
also consider infinite volume surfaces. Papers
\cite{BP06, ELJ97} obtain stable laws for geodesic windings on hyperbolic manifolds.
Those papers are relevant since
 the main contribution to windings comes from long excursions, so the
proofs of stable laws and of the Poisson laws for excursions are closely related,
see e.g. \cite{db_poisson, DS20}. In case the hyperbolic manifold under consideration
is the modular surface, the length of the $n$-th geodesic excursion is approximately
equal to the size of the $n$-th convergent of the continued fraction expansion of the
geodesic endpoint \cite{GLJ93}, therefore the multiple Borel-Cantelli Lemma in that case
follows from the results of \cite{AN}.

Several authors discussed extended Logarithm Law for excursion to other homogeneous spaces.
Namely, \cite{KM99} studies partially hyperbolic flows on homogenous spaces and presents
applications to metric number theory, cf. Section \ref{ScKhinchine} of the present paper.
Logarithm Law for unipotent flows is considered in \cite{AM1, AM2, GK17, Kel17}.
In the next section we obtain MultiLog Law for certain diagonal flows on the space of
lattices.

\section{  Recurrence in configuration space.}
\label{ScConf}
\subsection{ The results.} \label{sec91}
In this section we return to the study of compact manifolds, but we treat targets which have more
complicated geometry than the targets from Section \ref{ScHit}. We will see that a richer geometry
of targets leads to stronger results.

Let $\cQ$ be a compact manifold of a variable negative curvature and dimension $d+1$.
Denote $S\cQ$ for the unitary tangent bundle over $\cQ,$ $\pi : S\cQ \rightarrow \cQ$ the canonical
projection, $\phi$ the geodesic flow on $S\cQ$ preserving the Liouville measure $\mu.$

Fix a small number $\bar{\rho}>0.$ Given a point $a\in \cQ$ and $(q,v)\in S\cQ,$
let $t_j$ be consecutive times where the function $t\to d(a, \pi(\phi^t(q,v)))$ has  a local minima
such that $d_j:=d(a, \pi(\phi^{t_j}(q,v))\leq \bar{\rho}.$ Let $d^{(r)}_n(a,(q,v))$ be the $r$-th minima
  among the numbers $\{d_j\}_{t_j\leq n}.$

\begin{Thm}\label{ThmGeoBall} $ \ $
 (a) For each $a\in \cQ$  and almost every $(q,v)\in S\cQ,$
$$      \limsup_{n\rightarrow \infty} \frac{|\ln d^{(r)}_n(a,(q,v))|- \frac{1}{d}\ln n}{\ln \ln n}
        = \frac{1}{rd}.$$
\noindent (b) For  almost every $(q,v)\in S\cQ,$
$$ \limsup_{n\rightarrow \infty} \frac{|\ln d^{(r)}_n(q,(q,v))|- \frac{1}{d}\ln n}{\ln \ln n}
        = \frac{1}{rd}. $$
\end{Thm}

Note that in contrast with Section \ref{ScHit} there are no exceptional points for hitting. We also obtain a Poisson limit theorem. Denote
\begin{align*} B_{\rho}(a)&= \{ q\in \cQ : d(a,q) <\rho \},\\
\hat B_{\rho}(a)&=\{ (q, v)\in S\cQ: d(a,q) <\rho, v\in S_{q} \cQ  \},\\
\fA_{a,\fr}&= \bigcup_{t\in [0,\eps]} \phi^t \hat B_{\rho}(a), \\
\bfA_\rho&=\left\{\left((a,u),(q,v)\right)\in S\cQ\times S\cQ :\exists\,s\in[0,\varepsilon],
d\left(a, \pi(\phi^s(q,v))\right)<\rho\right\}.
\end{align*}
The following fact
proven in Appendix \ref{app.spheres} will be helpful in our argument.

\begin{Lem} \label{lem.mes} The following limits exist and does not dependent on $a\in\cQ$:
\begin{align} \label{GeoGamma}
\gamma&=\lim_{\rho\to 0}\mu \left(\fA_{a,\fr}\right)/\left(\eps\rho^d\right).
\end{align}
\end{Lem}

{The following will be a byproduct of our analysis and the proof will be given in \S \ref{SSPoisConf}.
\begin{Cor}
\label{CrPoisHit}
For each $a \in \cQ$, for every $\tau>0$, for  every $r\geq 1$, we have
\begin{align*}
(a)\quad \lim_{\rho \to 0} \mu\left((q,v)\in S\cQ: d_{\tau \rho^{-d}}^{(r)}(a, (q,v))<\rho \right)&=\sum_{l=0}^{r-1} e^{-\tau \gamma}\;
\frac{(\gamma \tau)^l}{l!}, \\
(b)\quad \lim_{\rho \to 0} \mu\left((q,v)\in S\cQ: d_{\tau \rho^{-d}}^{(r)}(q, (q,v))<\rho \right)&=\sum_{l=0}^{r-1} e^{-\tau  \gamma}\;
\frac{( \gamma \tau)^l}{l!}.
\end{align*}

\end{Cor} }

\subsection{   MultiLog Law. Proof of Theorem \ref{ThmGeoBall}.}
\label{SSGeoMultiLog}

We fix $r \in \N$ and consider the system $(f,S\cQ,\mu,\rm{Lip}),$ where $f=\phi^\eps$ for a small $\eps>0.$
We note  that it follows from \cite{L}[Theorem 2.4], \cite{Dolgopyat}[Theorem 2], Remark \ref{lowreg} and Theorem \ref{MultiMixRet} that  $(f,S\cQ,\mu,\rm{Lip})$
is $r$-fold exponentially mixing for every $r\geq 2$ as in Definition \ref{def.mixing}.

For $k\neq 0$, we keep the notations  $\fA^k_{a,\fr}$ for the event $1_{\fA_{a,\fr}}\circ f^k$, and $ {\bfA}^k_{\fr}$ for the event $\{(q,v): ((q,v), f^k(q,v))\in{\bfA}_{\fr}\}.$ We also keep the notation $\sigma(\fr)=\mu(\fA_{a,\fr})$ and $\bsigma(\fr)=(\mu\times\mu)(\bfA_\rho)$.

For $s\geq 0$, we let $\fr_n=n^{-1/d} \ln^{-s}n$, and recall that $N^n_{\fr_n}$ denotes the number of times $k\leq n$ such that $\fA^k_{a,\fr_n}$ (or $\bfA^k_{\fr_n}$) occurs.

The statement of Theorem \ref{thm3} becomes equivalent to the following :
\begin{itemize}
\item[(a)] If $s>\frac{1}{rd}$,  then for $\mu$-a.e. $(q,v)$, we have that  for large $n$,
$N^n_{\fr_n}<r.$
\item[(b)] If $s\leq\frac{1}{rd}$, then for $\mu$-a.e. $(q,v)$, there are infinitely many $n$ such that
$N^n_{\fr_n}\geq r.$
\end{itemize}

With the notation $\bS_r=\sum_{j=1}^\infty \left(2^j \bv_j\right)^r$
where $\bv_j=\fv(\fr_{2^j})$ (in the $\fA_{x,\rho_n}$ case) or
${\bv}_j=\bar\fv(\fr_{2^j})$ (in the $\bfA_{\rho_n}$ case), we see from \eqref{GeoGamma} 
that  $\bS_r=\infty$ if and only if $s\leq\frac{1}{rd}$.

Hence Theorem  \ref{ThmGeoBall} follows from Corollary \ref{cor.mixing}, 
 provided we establish the following.
\begin{Prop} \label{prop.targets2}
(a) For any $a \in \cQ$, the targets $\{\fA_{a,\fr_n}\}$  are simple admissible targets as in Definition \ref{def.targets}.

(b) The targets $\{\bfA_{\fr_n}\}$ are composite admissible targets as in Definition \ref{def.recurrent.targets}.
\end{Prop}

 The rest of this section is devoted to the
  \begin{proof}[Proof of Proposition \ref{prop.targets2}.]
Properties (Prod) and (Gr) are clear. Note that
$\Omega_{a,\rho}$ is a sublevel set of a Lipschitz function
$$ h(q,v)=\min_{s\in[0, \eps]} d(a,\pi \phi^s(q,v)) $$
so (Appr) follows as in Lemma \ref{RemLip}. To prove the first part of Proposition \ref{prop.targets2}, it only remains to check (Mov).
That is, we need to prove the following Lemma.
\begin{Lem} \label{lMovCS}
There exists $\eta>0$
\footnote{In fact, it can be seen from the proof that $\eta$ can be taken to be $d$, that is, we have
quasiindependence   in Lemma \ref{lMovCS} $\mu (\fA_{a,\fr} \cap \phi^{-t} \fA_{a,\fr})
\leq C\mu (\fA_{a,\fr})^{2}$}
 and $t_0>0$ such that for any $a\in Q$ and $\rho$ sufficiently small,
\begin{equation}
\label{MovCS}
 \mu (\fA_{a,\fr} \cap \phi^{-t} \fA_{a,\fr})
\leq \mu (\fA_{a,\fr})^{1+\eta},
\end{equation}
for all $t>t_0.$
\end{Lem}

Recall that $S_q\cQ$ is the unit tangent bundle at the point $q.$ Denote $\DS A_\eps(q)=\bigcup_{s\in [0, \eps]} \phi^sS_q\cQ,$ which is an embedded submanifold
with boundary in $S\cQ$ of dimension  $d+1.$

\begin{Lem} \label{lemmaMovLoc} We let $\nu$ be the restriction of $\mu$ on $A_\eps(q).$
For each $a\in \cQ$

\begin{equation}
\label{MovCSLoc}
 \nu \left(A_\eps(q) \cap \phi^{-t} \fA_{a,\fr} \right)
\leq C \rho^{\eta}\nu(A_\eps(q)).
\end{equation}
\end{Lem}

Lemma \ref{lMovCS} follows from Lemma \ref{lemmaMovLoc}  by integration on $q \in B_\rho(a).$

Introduce $\Sigma(t,q,\eps):=\phi^t \DS A_\eps(q)$. Note that  $\Sigma(t,q,\eps)$ is an embedded submanifold on $S\cQ$ of dimension $d+1.$

The proof of the following result is given in the Appendix \ref{app.spheres}.

\begin{Lem}[Geometry of expanded spheres in the configuration space] \label{lem.sigma}
We have that $\pi:\Sigma(t, q, \eps)\to \cQ$ is a local diffeomorphism. Moreover
for the inverse map \\$d\pi^{-1}: S\cQ\to S\Sigma(t, q, \eps)$ the norm
 $||d\pi^{-1}||$ is uniformly bounded.
\end{Lem}

\begin{proof}[Proof of Lemma \ref{lemmaMovLoc}.]
By elementary geometry and the bounded distortion property
\begin{equation}\label{eAB}
\nu(A_\eps (q)\cap \phi^{-t} \fA_{a,\fr} ) \leq C\rho^{-1}
\nu(\Sigma(t,q,\eps)\cap\hat B_{2\rho}(a)).
\end{equation}

\begin{figure}[b]
\begin{tikzpicture}[scale=0.5]

        \draw (-5,0) circle (3.5cm);
        \draw (-5,0) circle (2.6cm);

        \fill (-5+3.05,0) circle (0.3);
        \fill (-5-3.05,0) circle (0.3);
        \fill (-5,0+3.05) circle (0.3);
        \fill (-5,0-3.05) circle (0.3);
        \fill (-5+3.05/1.414,+3.05/1.414) circle (0.3);
        \fill (-5-3.05/1.414,+3.05/1.414) circle (0.3);
        \fill (-5-3.05/1.414,-3.05/1.414) circle (0.3);
        \fill (-5+3.05/1.414,-3.05/1.414) circle (0.3);

        \draw ($(-5,0)+(22.5:2.6)$) -- ($(-5,0)+(22.5:3.5)$);
        \draw ($(-5,0)+(22.5+45:2.6)$) -- ($(-5,0)+(22.5+45:3.5)$);
        \draw ($(-5,0)+(22.5+45*2:2.6)$) -- ($(-5,0)+(22.5+45*2:3.5)$);
        \draw ($(-5,0)+(22.5+45*3:2.6)$) -- ($(-5,0)+(22.5+45*3:3.5)$);
        \draw ($(-5,0)+(22.5+45*4:2.6)$) -- ($(-5,0)+(22.5+45*4:3.5)$);
        \draw ($(-5,0)+(22.5+45*5:2.6)$) -- ($(-5,0)+(22.5+45*5:3.5)$);
        \draw ($(-5,0)+(22.5+45*6:2.6)$) -- ($(-5,0)+(22.5+45*6:3.5)$);
        \draw ($(-5,0)+(22.5+45*7:2.6)$) -- ($(-5,0)+(22.5+45*7:3.5)$);

        \draw[<->] ($(-5,0)+(22.5:3.8)$) arc (22.5:67.5:3.8);
        \node[right] at ($(-5,0)+(50:4)$) {$\pi \Sigma_j(t)$};


        \draw (8.5,0) arc (0:360:3.5 and 2.6);

        \fill (5,0) -- (5+2.8,0+1.56) -- (5+2.3,0+1.96);

        \fill (5,0) -- (5-2.8,0-1.56) -- (5-2.3,0-1.96);

        \fill (5,0) -- (5+2.8,0-1.56) -- (5+2.3,0-1.96);

        \fill (5,0) -- (5-2.8,0+1.56) -- (5-2.3,0+1.96);

        \fill (5,0) -- (5+0.25,0+2.593) -- (5-0.25,0+2.593);

        \fill (5,0) -- (5+0.25,0-2.593) -- (5-0.25,0-2.593);

        \fill (5,0) -- (5+3.48,0-0.277) -- (5+3.48,0+0.277);

        \fill (5,0) -- (5-3.48,0-0.277) -- (5-3.48,0+0.277);

        \draw (5,0) -- (5+1.3, 0+2.414);
        \draw (5,0) -- (5-1.3, 0+2.414);
        \draw (5,0) -- (5+1.3, 0-2.414);
        \draw (5,0) -- (5-1.3, 0-2.414);
        \draw (5,0) -- (5-3.3, 0-0.866);
        \draw (5,0) -- (5+3.3, 0-0.866);

        \fill[white] (5+1.2,-0.15) arc (0:180:1.2 and 0.8);
        \fill[white] (5+1.5,0.4) arc (360:180:1.5 and 1);

        \draw (5+1.2,-0.15) arc (0:180:1.2 and 0.8);
        \draw (5+1.5,0.4) arc (360:180:1.5 and 1);

        \draw (5+1.1,0.2) -- (5+3.25, 0+0.965);
        \draw (5-1.1,0.2) -- (5-3.25, 0+0.965);

        \draw[<->] (5+3.53, 0+1.04) arc (21.8:69:3.8 and 2.8);
        \node[right] at ($(5,0)+(50:3.5)$) {$\Sigma_j(t)$};

		\end{tikzpicture}

\caption{Proof of Lemma \ref{lemmaMovLoc}}\label{fig1}
	\end{figure}

By Lemma \ref{lem.sigma},  $||d\pi^{-1}||$ is uniformly bounded.
 Note that $\pi \Sigma(t,a, \eps)$ is an annulus whose boundaries are spheres
of radii $t$ and $t+\eps$ respectively.
Note those spheres are perpendicular to the geodesics emanating from $q.$
Since the width of annulus is equal to $\eps$ and does not
depend on $t$, taking a maximal $1$-separated set in the sphere of radius of $t+(\eps/2)$
 and considering
associated Voronoi cells
we see that
$\Sigma(t, q, \eps)$ can be cut into several disjoint piece $\Sigma_j(t)$ satisfying that for each $j,$
$\pi\Sigma_j(t)$ is contained in a ball of radius $\eps_2$
(independent of $t$ and $q$) and
contains a ball of radius $\eps/2.$
Decreasing  $\eps_2$ if necessary we obtain that the intersection  $\pi\Sigma_j(t)\cap B_{2\rho}(a)$ has only one component and
since $d\pi^{-1}$ is bounded we get that
$$ \nu(\Sigma_j(t)\cap \hat B_{2\rho}(a))\leq C(\eps_1)  \rho^{d+1}\nu\left(\Sigma_j(t)\right). $$
Summing over $j$ in \eqref{eAB} we obtain \eqref{MovCSLoc} which finishes the proof of Lemma \ref{lemmaMovLoc}.
\end{proof}

The proof of Proposition \ref{prop.targets2} (a) is thus completed. 

Now we turn to the proof of Proposition \ref{prop.targets2} (b).  The task is to
 verify the conditions $\rm{\overline{(Appr)}},$ $\rm{\overline{(Mov)}}$ and
 ${\rm(\overline{Sub})}$ for the targets $\bfA_\rho$ defined in Section \ref{sec91}. {The proof of $\rm{\overline{(Appr)}}$ and $\rm{\overline{(Sub)}}$ is obtained from Lemma \ref{RemLipcomposite}  exactly  as in the proof of Lemma \ref{lemma.appr} that treats the case of the composite targets of Section \ref{sec.wr}.
 }
It is left to verify $\rm{\overline{(Mov)}}.$
Take $x_i\in \cQ,$  $B_i=B(x_i,\rho),$ $1\leq i\leq k$ such that $\DS \cQ=\bigcup_{i=1}^k B_i$ and $k=O(\fr^{-d}).$
By \eqref{MovCS}, for $t>t_0$
\Bea
\mu(\bfA_\rho^t)&\leq&\sum_{i}\left\{(q,v)\in S\cQ:\exists\,s\in[0,\varepsilon],
d\left(q, \pi(\phi^{s+t}(q,v))\right)<\frac{\rho}{\gamma(q)^{1/d}},q\in B_i\right\}\\
&\leq&\sum_{i}\left\{(q,v)\in S\cQ:\exists\,s\in[0,\varepsilon],
d\left(x_i, \pi(\phi^{s+t}(q,v))\right)<c\rho,q\in B_i\right\}\\
&\leq&{\sum_i\mu(\fA_{x_i,c\fr})^{1+\eta}}\leq\sum_{i}C\rho^{d(1+\eta)}\leq C\rho^{\eta}.
\Eea
This completes the proof of Proposition \ref{prop.targets2} and finishes the proof of Theorem  \ref{ThmGeoBall}. \end{proof}

\subsection{   Poisson regime. Proof of Corollary \ref{CrPoisHit}.}

\label{SSPoisConf}

Part (a) follows from Theorem \ref{ThPoisson}, since conditions $(M1)_r$ and $(M2)_r$ are satisfied
for all $r,$ due to the results of \S \ref{SSGeoMultiLog}.

The proof of part (b) follows the same argument as the proof of Theorem \ref{ThPoisRet}
except that now $(M1)_r$ is satisfied since the RHS of \eqref{PoiLawM1} takes form
$\rho^d \lambda$ because $\lambda$ defined by \eqref{GeoGamma} does not depend on $a.$ \hfill { \tiny $\Box$}

\subsection{  Notes.}
In \cite{Mau06}, Maucourant proved that for all $a\in \cQ$ and almost every $(q,v)\in S\cQ$
\begin{equation*}\label{GeoBallMau}
\limsup _{t \rightarrow+\infty} \frac{|\ln d\left(a,\pi\left(\phi^{t}(q,v)\right)\right)|}{\ln t}=\frac{1}{d}.
\end{equation*}
 \cite{KZ18} generalized Maucourant's result to study a shrinking target
problem for time $h$ map.
The shrinking target problems for sets with complicated geometry is discussed in
\cite{Galatolo10, GN, GK17, Kel17, KO18, KY19}.

Concerning Poisson Limits we note that visits to sets with complicated geometry naturally
appears in Extreme Value Theory, see Section \ref{ScExtreme} for details.
\cite{Yang19} provided a general conditions for the number of visits to a small neighborhood
of arbitrary submanifold to be asymptotically Poisson.

\section{  Multiple Khintchine-Groshev Theorem.}
\label{ScKhinchine}
\subsection{  Statements.} $\ $

\bigskip \noindent {\sc  Homogenous approximations.}
For $x \in \R^d$, we use the notation  $|x|=\sqrt{\sum x_i^2}$.

\begin{Def}[{$(r,s)$-approximable vectors}] Given $\alpha=(\alpha_1, \dots, \alpha_d) \in\reals^d$, $s\geq 0$, $c>0$, let
$D_N(\alpha,s,c)$ be the set of $k=(k_1, \dots, k_d)\in\integers^d$ such that
$$ |k|\leq N \text{ and }\exists m\in\integers: gcd(k_1, \dots, k_d, m)=1
\text{ and }  |k|^d \left|\langle k, \alpha \rangle+m\right|\leq \frac{c}{\ln N (\ln \ln N)^s}. $$
Call $\alpha$ {\it $(r,s)$-approximable} if for any $c>0$, $\Card(D_N(\alpha,s,c))\geq 2r$ for infinitely many $N$s.
\end{Def}

\begin{Thm} \label{theo.KG}
If $s\leq 1/r$ then the set of $(r,s)$-approximable vectors $\a \in \T^d$  has full measure.
If $s>1/r$ then the set of $(r,s)$-approximable numbers has zero measure.
\end{Thm}

\begin{Rem} Observe that an equivalent statement of Theorem \ref{theo.KG} is to replace $2r$ with $r$ in the definition of $(r,s)$ approximable vectors provided we restrict to $k \in \Z^d$ such that $k_1 > 0$. This will be the version that we will prove in the sequel.
\end{Rem}

\bigskip \noindent {\sc Inhomogeneous approximations.}

\begin{Def}[{$(r,s)$-approximable couples}]
Given $\alpha=(\alpha_1, \dots, \alpha_d) \in\reals^d$ and $z \in \R$, $s\geq 0$ and $c>0$, let
$D_N(\alpha,z,s,c)$ be the set of $k=(k_1, \dots, k_d)\in\integers^d$ such that
$$ |k|\leq N \text{ and }\exists m\in\integers:  |k|^d \left|z+\langle k, \alpha \rangle+m\right|\leq \frac{c}{\ln N (\ln \ln N)^s}. $$
Call the couple $(\alpha,z)$ {\it $(r,s)$-approximable} if for any $c>0$, $\Card(D_N(\alpha,z,s,c))\geq r$ for infinitely many $N$s.
\end{Def}

\begin{Thm} \label{theo.KG_affine}
If $s\leq 1/r$ then the set of $(r,s)$-approximable couples $(\a,z)\in  \R^d \times \R$ has full measure.
If $s>1/r$ then the set of $(r,s)$-approximable couples $(\a,z)\in  \R^d \times \R$ has zero measure.
\end{Thm}

\bigskip \noindent {\sc Extensions.}
One can extend the above results to general Kintchine Groshev $0-1$ laws for Diophantine approximations of linear forms. For example

\begin{Def}[{$(r,s)$-simultaneously approximable vectors}] Given $\alpha=(\alpha_1, \dots \alpha_d) \in\reals^d$, $s\geq 0$, $c>0$, let
$D_N(\alpha,s,c)$ be the set of $k \in\integers^*$ such that
\begin{multline*} k \leq N \text{ and }  \exists m\in \Z^d :  gcd(k,m_1, \dots,m_d)=1 \\  \text{ and  for all } i =1,\ldots, d, \quad  k^{\frac{1}{d}}  \left| k \alpha_i+m_i \right|\leq \frac{c}{(\ln N)^{\frac 1 d}  (\ln \ln N)^{\frac s d}}. \end{multline*}

Call $\alpha$ {\it $(r,s)$-simultaneously approximable} if for any $c>0$, $\Card(D_N(\alpha,s,c))\geq r$ for infinitely many $N$s.
\end{Def}

\begin{Thm} \label{theo.KG_simultaneous}
If $s\leq 1/r$ then the set of $(r,s)$-simultaneously approximable vectors $\a \in \T^d$  has full measure.
If $s>1/r$ then the set of $(r,s)$-simultaneously approximable numbers has zero measure.
\end{Thm}

We omit the proof of Theorem \ref{theo.KG_simultaneous} since it is obtained by routine
modification of the proof of Theorem \ref{theo.KG}.

 \subsection{  Reduction to a problem on the space of lattices} \  \black Let $\cM$ be the space of $d+1$ dimensional unimodular lattices. We identify $\cM$  with $SL_{d+1}(\R)/SL_{d+1}(\Z)$. Denote Haar measure on $\cM$ by $\mu$. Define

\[ \Lambda_\a=
\begin{pmatrix} \text{Id}_d & 0 \\ \a & 1\end{pmatrix}.
\]
For $t \in \R$, we consider $g_t \in SL_{d+1}(\R)$
\begin{align}
\label{DefGT}
  g_t=
 \begin{pmatrix}
  2^{-t}\\
  & \ddots\\
  && 2^{-t}\\
 &&& 2^{dt}
 \end{pmatrix}
\end{align}

For a lattice $\cL \subset \cM_{d+1}$, we say that a vector in $\cL$ is prime if it is not an integer multiple of another vector in $\cL$.

Given a function $f$ on $\R^{d+1}$ we consider its Siegel transform
$\cS(f):\cM\to\R$  defined by
\begin{equation} \label{eq.siegel} \cS(f)(\cL)  =\sum_{e\in \cL, \  e \text{ prime }} f(e). \end{equation}


For $a>0$, let $\phi_a$ be the indicator of the set\footnote{We added $x_1 > 0$ in the definition of $E_a$ since we will restrict to vectors $k \in \Z^d$ with $k_1\geq 0$.}
\begin{equation*}
  E_a:=\left\{
  (x,y)
  \in\R^{d}\times\R\mid
x_1 > 0,  |x| \in [1,2],
  |x|^{d} |y| \in [0,a]\right\}.
\end{equation*}

Fix $s\geq 0, c>0$. For $M \in \N^*$, define
\begin{equation} \label{eq.Ec} \nu:=\frac{c}{M (\ln M)^s}, \quad
\Phi_\nu:=\cS(\phi_\nu).
\end{equation}
For $t\geq 0$, we then define
$$A_t(M):= \{\a \in \T^d : \Phi_{\nu}(g_t \Lambda_\a) \geq 1\}$$

It is readily checked that  $\a\in A_t(M)$ if and only if there exists  $k=(k_1,\ldots,k_d)$ with $k_1\geq 0$, and $2^{t}<|k|\leq 2^{t+1}$ such that
\begin{equation} \label{eq.M} \exists m, \quad gcd(k_1,\dots k_d, m)=1, \quad   |k|^d |\langle k, \alpha \rangle+m|\leq \frac{c}{M (\ln M)^s}.\end{equation}

If $\a$ is such that $\Phi_{\nu}(g_t \Lambda_\a) \leq 1$ for every $t \in \N$, then we get that $\a$ is
 $(r,s)$-approximable if and only if there exists
 infinitely many $M$ for which there exists $0<t_1<t_2<\ldots<t_r \leq M$ satisfying $\a \in \bigcap_{j=1}^r A_{t_j}(M).$

But in general, for $\a$ and $t\leq M$ such that $\a \in A_t(M)$, there may be multiple solutions $k$ such that $2^{t}<|k|\leq 2^{t+1}$ for the same $t$. Since in Theorem \ref{theo.KG} we are counting all solutions we have to deal with this issue.

The following proposition proven in \S \ref{SSModifyHom}
shows that for a.e. $\a$, multiple solutions do not occur.

\begin{Prop} \label{prop_multiple} For almost every $\a$, we have that for every $M$ sufficiently large, for every $t\in [0,M]$, it holds that  $\Phi_{\nu}(g_t \Lambda_\a) \leq 1$
\end{Prop}

Hence, Theorem \ref{theo.KG} is equivalent to the following.

\begin{Thm}  \label{Theo_Lattice}
If $rs\leq 1$, then for almost every $\a \in \T^d$, there exists infinitely many $M$ for which there exists $0<t_1<t_2<\ldots<t_r \leq M$ satisfying
$$\a \in \bigcap_{j=1}^r A_{t_j}(M).$$

If $rs> 1$, then for almost every $\a \in \T^d$, there exists at most finitely  many $M$ for which there exists $0<t_1<t_2<\ldots<t_r \leq M$ satisfying
$$\a \in \bigcap_{j=1}^r A_{t_j}(M).$$
 \end{Thm}

\subsection{Modifying the initial distribution: homogeneous case.}
\label{SSModifyHom}
 We transformed our problem into a problem of multiple recurrence of the diagonal action $g_t$ when applied to a piece of horocycle in the direction of $\Lambda_\a : \a \in \T^d$. But this horocycle is exactly the full strong unstable direction of the rapidly mixing partially hyperbolic action $g_t$.
Due to the equidistribution of the strong unstable horocycles, it is thus possible and much more convenient  to work with Haar measure  on $\cM$
instead of Haar measure on $\Lambda_\a$ for $\a \in \T^d$.

Hence, we define
$$B_t(M):= \{\cL \in \cM : \Phi_{\nu}(g_t \cL) \geq 1\},$$
where we recall that
$\DS \nu:=\frac{c}{M (\ln M)^s},$
$\DS \Phi_\nu:=\cS(\phi_\nu),$
and  $\phi_\nu$ is the indicator of the set $E_\nu=\left\{
  (x,y)
  \in\R^{d}\times\R\mid
x_1 > 0,  |x| \in [1,2],
  |x|^{d} |y| \in [0,\nu]\right\}$.

Our goal becomes to prove the following.

\begin{Prop} \label{prop_multiple_Haar} For $\mu$-almost every $\cL \in \cM$, we have that for every $M$ sufficiently large, for every $t\in [0,M]$, it holds that  $\Phi_{\nu}(g_t \cL) \leq 1.$
\end{Prop}

\begin{Thm}  \label{Theo_Lattice_Haar}

If $rs\leq 1$, then for $\mu$-almost every $\cL \in \cM$, there exists infinitely many $M$ for which there exists $0<t_1<t_2<\ldots<t_r \leq M$ satisfying
$$\cL \in \bigcap_{j=1}^r B_{t_j}(M).$$

If $rs> 1$, then for $\mu$-almost every $\cL \in \cM$, there exists at most finitely  many $M$ for which there exists $0<t_1<t_2<\ldots<t_r \leq M$ satisfying
$$\cL \in \bigcap_{j=1}^r B_{t_j}(M).$$
\end{Thm}

 \noindent {\it Proof that Proposition \ref{prop_multiple_Haar} and Theorem  \ref{Theo_Lattice_Haar} imply Proposition \ref{prop_multiple} and Theorem  \ref{Theo_Lattice}.}

Recall that for $M \in \N$ we defined $\nu=\frac{c}{M (\ln M)^s}$.
Fix $\eta>0$ and define  ${\Phi}_\nu^\pm$ as in \eqref{eq.Ec} but with $(1+\eta)c$ and $(1-\eta)c$ instead of $c$. Next, define for $\beta \in \R^d$

\[ \Lambda^-_\beta=
\begin{pmatrix} \text{Id}_d & \beta \\ 0& 1\end{pmatrix},
\]
and for $B \in SL_{d}(\R)$ we define
\begin{align*}D_B= \begin{pmatrix} B & 0 \\ 0 & 1\end{pmatrix},
\end{align*}
and finally
$$\tLambda_{\a,\beta,B}=  D_B \Lambda^-_\beta \Lambda_\a.$$

Fix $0<\eps\ll \eta$. If $B$ is distributed according to a smooth density with respect to
Haar measure on $SL_{d}(\R)$ in an $\eps$ neighborhood of the Identity, $\beta$ is distributed in some
$\eps$ neighborhood of $0$ in $\R^d$ with a smooth density according to Haar measure of $\T^d$, and $\a$ is distributed according to any measure with smooth density with respect to Haar measure on $\T^d$, then
the lattice $\tLambda_{\a,\beta,B}$ is distributed according to a smooth density in $\cM$
with respect to the Haar measure $\mu$.
Moreover, because $ \Lambda^-_\beta$ forms the stable direction of $g_t$ and because $D_B$ forms the centralizer of $g_t$, we have that
 if $M$ is sufficiently large, then
\begin{equation*} {\Phi}_\nu^-(g_{t} \tLambda_{\a,\beta,B})
\geq 1 \implies \Phi_\nu(g_{t} \Lambda_\a)  \geq 1  \implies {\Phi}_\nu^+(g_{t} \tLambda_{\a,\beta,B})
\geq 1.  \end{equation*}
This shows that Proposition \ref{prop_multiple} and Theorem  \ref{Theo_Lattice} follow from Proposition \ref{prop_multiple_Haar} and Theorem  \ref{Theo_Lattice_Haar} respectively.  \hfill $\Box$

 \subsection{  Rogers identities}\label{sec.rogers}
The following identities (see \cite{Mar1, V}) play an important role in our argument.
Denote
$$ \bc_1=\zeta(d+1)^{-1}, \quad \bc_2=\zeta(d+1)^{-2}, \quad \text{where }
\zeta(d+1)=\sum_{n=1}^\infty n^{-(d+1)} $$
is the Riemann zeta function.

Let $f, f_1, f_2$ be piecewise smooth functions with compact support on $\reals^{d+1}.$

Let
$$ F(\cL)=\sum_{e\in \cL, \text{ prime}} f(e), \quad
\brF(\cL)=\sum_{e_1\neq \pm e_2\in \cL, \text{ prime } } f_1(e_1) f_2(e_2). $$
$F$ is the {\it Siegel transform of $f$} that we denoted $\cS(f).$
\begin{Lem}
\label{LmRI} We have

\begin{align*} (a)\quad \int_{\cM} F(\cL) d\mu(\cL)&=\bc_1 \int_{\reals^{d+1}} f(x) dx, \\
(b)\quad \int_{\cM} \brF(\cL) d\mu(\cL)&=\bc_2 \int_{\reals^{d+1}} f_1(x) dx \int_{\reals^{d+1}} f_2(x) dx. \end{align*}
\end{Lem}

 \subsection{Multiple solutions on the same scale.  Proof of Proposition \ref{prop_multiple_Haar}.}
Recall that $\nu=\frac{c}{M \ln M^s}$
\begin{Lem}  \label{lemma_multiple}
There exists a constant $C>0$, such that for every $M$, for every $t\in \R$, it holds that
$$\mu \left(\Phi_{\nu}(g_t \cL) > 1\right) \leq C c^2 M^{-2} (\ln M)^{-2s}.$$
\end{Lem}

For $K\geq 0$, apply the lemma for $M=2^K$ and sum over all $t\in [0,M]$, then
$$\mu\left(\exists t \leq 2^K, \Phi_{4\nu}(g_t \cL) > 1\right) \leq 16 Cc^2 2^{-K} K^{-2s} .$$
The straightforward side of Borel Cantelli lemma  gives that for almost every $\cL$, for $K$ sufficiently large, for any $t \leq 2^K, \Phi_{4\nu}(g_t \cL) \leq 1$. For the same $\cL$, it then holds that for $M$ sufficiently large, for any $t \leq M$, $ \Phi_{\nu}(g_t \cL) \leq 1$.

To finish the proof of Proposition \ref{prop_multiple_Haar} we give
\begin{proof}[Proof of Lemma \ref{lemma_multiple}]

Since $g_t$ preserves Haar measure on $\cM$ it suffices to prove the lemma for $t=0$. But the condition $k_1\geq 0$ implies that
\begin{equation*}
\label{PhiSq-Phi}
\Phi_{\nu}^2(\cL) -\Phi_{\nu}(\cL) =    \sum_{e_1\neq e_2\in L \text{ prime}} \phi_{\nu}(e_1) \phi_{\nu}(e_2)
=\sum_{e_1\neq \pm e_2\in L \text{ prime}}\phi_{\nu}(e_1) \phi_{\nu}(e_2).
\end{equation*}
It then follows from Rogers identity (b) of Lemma \ref{LmRI} that \vskip2mm

\noindent
$\DS \mu\left(\Phi_{{\nu}}(\cL) > 1\right)\leq  \E\left(\Phi_{\nu}^2(\cL) -\Phi_{\nu}(\cL)\right)\leq \bc_2 \left(\int_{\R^{d+1}} \phi_{\nu}(u) du \right)^2\leq C c^2 M^{-2}(\ln M)^{-2s}.
$
\end{proof}

{\subsection{  Proof of Theorem \ref{Theo_Lattice_Haar}} \

We want to apply Corollary \ref{cor.mixing}.
For the system $(f,X,\mu)$ we take $(g_1,\mathcal M,\mu)$, where $\mu$ is the Haar measure on $\cM$.
For the targets, we take $\fA_\fr=\{\cL : \Phi_{\fr}(\cL) \geq 1\}$ and  $\fA_\fr^t=g_{-t}\fA_\fr.$ Note that by the invariance of the Haar measure  by $g_t$ we have that $\mu(\fA^t_{\fr})=\mu(\fA_{\fr})$ for any~$t$.


For $s \in \N$, we define the sequence $\fr_M:=\frac{c}{M (\ln M)^s}$.
The conclusions of Theorem \ref{Theo_Lattice_Haar} will then follow from the conclusion of Corollary \ref{cor.mixing} applied to $N^M_{\fr_M}$, where $N^n_{\fr}$ is the  number of times $t\leq n$ such that $\fA^t_\fr$ occurs.

Indeed, if we recall the definition of
$$ \bS_r=\sum_{j=1}^\infty \left(2^j \bv_j\right)^r, \quad \bv_j=\fv(\fr_{2^j}), \quad \fv(\fr)=\mu(\fA_\fr)$$
we see that $\bS_r=\infty$ if and only if  $rs\leq1$.

This being said, to be able to apply Corollary \ref{cor.mixing} and finish, we still need to check the conditions of Definition \ref{def.mixing}  and Definition \ref{def.targets} for the system  $(g_1,\cM,\mu,\BAN)$ and for the family of targets given by $\fA_\fr$ and the sequence $\fr_M.$  $(EM)_r$ follows from Theorem 1.1 of \cite{BEG}, Remark \ref{lowreg} and Theorem \ref{MultiMixRet}. And
the approximation condition (Appr) can be checked as follows.  Indeed we have:


\medskip

\noindent{\sc Claim.} There exists $\sigma>0$ such that, for every $\fr>0$ sufficiently small, there exists  $A_\fr^-, A_\fr^+ \in {\rm Lip(\cM)}$
such that
\begin{itemize}
\item[(i)] $\|A_\fr^\pm\|_\infty\leq 2$
and $\|A_\fr^\pm\|_{\rm Lip}\leq  \fr^{-\sigma};$
\item[(ii)] $A_\fr^-\leq 1_{\fA_\fr}\leq A_\fr^+;$
\item[(iii)] $\mu(A_\fr^+)-\mu(A_\fr^-)\leq \fr^2$
\end{itemize}

\medskip

Clearly the claim implies (Appr) since $\mu(\fA_\fr)=\cO(\fr)$.

\begin{proof}[Proof of the claim]

{Recall that $\Phi_\rho=\cS(\phi_\rho)$, where
$\phi_\rho$ is the indicator of the set $E_\rho=\left\{
  (x,y)
  \in\R^{d}\times\R\mid
x_1 > 0,  |x| \in [1,2],
  |x|^{d} |y| \in [0,\rho]\right\}$. We will construct $A^+_\rho$ that satisfies  $(i)$, $(ii)$ and $(iii)$ with $\mu(A_\fr^-)$ replaced by $\mu(1_{\fA_\fr})$. The construction of $A^-_\rho$ is similar.

Pick {$f^{+}\in {\rm Lip(\R^{d+1},[0,2])}$} such that  for some $\sigma>0$
\begin{itemize}
\item $\|f^+\|_{ {\rm Lip}}\leq \rho^{-\sigma}$
\item For $z\in E_\rho$, $f^+(z)=1$
\item For $z\notin E_{\rho+\rho^{10}}$, $f^+(z)=0$.
\end{itemize}
As the consequence $\cS(f^+) \in {\rm Lip(\cM)}$ and $\Phi_\rho\leq \cS(f^+),$
and using Rogers identity
of Lemma \ref{LmRI}(a) (applied to the Siegel transform of the characteristic function of the set ${E_{\rho+\rho^{10}}-E_\rho}$) we get for $\rho$ sufficiently small an open set $\cE_\rho\subset \cM$ such that $\mu(\cE_\rho)\leq \rho^3$
 \begin{itemize}
 \item[($P1$)] For $\cL \notin \cE_\rho$, if $\cS(f^+) > 0$, then $\Phi_\rho\geq1.$
 \item[$(P2)$] If $\cM_\rho:=\{\cL: \cS(f^+) <2\}$, then $\|\cS(f^+)\|_{\rm{Lip(\cM_\rho)}}\leq  \rho^{-\sigma-1}.$
 \end{itemize}

  Let now $u:\R\to [0,1]$ be some increasing $C^\infty$ function such that $u(x)=0$ for $x\leq 0$ and $u(x)=1$ for $x\geq 1$.

Finally, introduce $A^+_\rho: \cM \to \R$ such that for $\cL \in \cM$
\begin{equation*}
A^+_\rho(\cL)= u\left(\cS(f^+)(\cL)\right)
\end{equation*}

We now check that $A^+_\rho$ satisfies the requirements of the claim.


Since $u\in C^\infty(\cM,[0,1])$ we get that $A^+_\rho \in{\rm Lip(\cM)}$ and  $\|A^+_\rho\|_\infty\leq 2$. To prove the Lipschitz bound, observe that for $\cL \notin \cM_\rho$ we have that $A^+_\rho(\cL)=1$, while for $\cL \in M_\rho$ we have $(P2)$. Hence
$\|A^+_\rho\|_{\rm Lip}\leq \rho^{-2\sigma}$. This proves $(i)$ of the claim. To see $(ii)$, just observe that 
$$\Phi_\rho(\cL)\geq 1\implies \cS(f^+)(\cL)\geq 1\implies A^+_\rho(\cL)=1.$$
We turn to $(iii)$. If $\cL \notin \cE_\rho$, then by $(P1)$
$$A^+_\rho(\cL)>0 \implies\cS(f^+)(\cL)>0 \implies \Phi_\rho(\cL) \geq 1 \implies A^+_\rho(\cL)=1.$$
Since $\mu(\cE_\rho )\leq \rho^3$ and $\|A^+_\rho\|_\infty \leq 2$, we get that $\mu(A_\fr^+)-\mu(1_{\fA_\fr}) \leq \rho^2$ and $(iii)$ is proved.

  }

  \end{proof}

Next we show now how Rogers identity
of Lemma \ref{LmRI}(b) implies
(Mov).  Define
\begin{equation*}
  E^\tau_\nu=\left\{
  (x,y)
  \in\R^{d}\times\R\mid
x_1 > 0,  2^{-\tau} |x| \in [1,2],
  |x|^{d} |y| \in [0,\nu]\right\}
\end{equation*}
and  let $\phi_\nu^\tau$ be the indicator function of $E^\tau_\nu$. Then
$$ \mu(\fA_\fr \cap g_{-\tau} \fA_\fr)\leq \E(\Phi_\rho \Phi_\rho \circ g_\tau)=  \int_{\cM} \sum_{e_2\neq \pm e_1\in \cL \text{ prime}} \phi_\fr(e_1) \phi_\fr^\tau(e_2)
 d\mu(\cL),
$$
where the contribution of $e_2=-e_1$ vanishes because the contribution of any pair $(e_1,e_2)$ where not both $e_{1,1}$ and $e_{2,1}$ are positive is zero. Applying Lemma \ref{LmRI} (b) we get that
\begin{equation*} \mu(\fA_\fr \cap g_{-\tau} \fA_\fr) \leq C\mu(\fA_\fr)^2\end{equation*}
which is stronger than the required (Mov).

Finally,
the condition \eqref{rhobound} clearly holds for the sequence $\rho_M=\frac{c}{M (\ln M)^s}$ due to Lemma \ref{LmRI} (a).  \hfill  $\Box$}

\subsection{  The argument in the inhomogeneous case.}
The proof of Theorem \ref{theo.KG_affine} is very similar to that of Theorem \ref{theo.KG}, and below we only outline the main differences.

Let $\tcM$ be the space of $d+1$ dimensional unimodular affine lattices. We identify  $\tcM$  with $SL_{d+1} (\R)\ltimes \R^{d+1}/ SL_{d+1} (\Z)\ltimes \Z^{d+1}$, where
the multiplication rule in $SL_{d+1} (\R)\ltimes \R^{d+1}$ is defined as
$(A, a)(B, b)=(AB, a+Ab)$. We denote by $\tilde{\mu}$ the Haar measure on $\tcM$.

For  $\a \in \R^d$ and $z\in \R$, we define
\begin{equation}
\label{DefLambdaAZ}
\Lambda_{\a,z}=(\Lambda_\a, (0,\ldots,0,z))
\end{equation}

Given a function $f$ on $\R^{d+1}$ we consider its Siegel transform
$\cS(f):\tcM\to\R$  defined by
\begin{equation}\label{SiegelAffine} \tcS(f)(\tcL)  =\sum_{e\in \tcL} f(e). \end{equation}

Note that, unlike our definition of the Siegel transform in the case of regular lattices, we do not require in this affine setting that the vectors $e$ in the summation be prime. This is because in this affine setting, when a vector $k \in \Z^d$  contributes to the Diophantine approximation counting problem there is no reason for the multiples of $k$ to contribute.

For $a>0$, let $\tphi_a$ be the indicator of the set\footnote{Note that we do not ask in this affine setting that $x_1 > 0$ in the definition of $\tE_a$ since the symmetric contributions of  $- k$ for every $k \in \Z^d$ that contributes to the Diophantine approximation counting problem
in the homogenous case of Theorem \ref{theo.KG} do not appear in the inhomogeneous Diophantine approximation problem of Theorem \ref{theo.KG_affine}.}

\begin{equation*}
  \tE_a:=\left\{
  (x,y)
  \in\R^{d}\times\R\mid
  |x| \in [1,2],
  |x|^{d} |y| \in [0,a]\right\}
\end{equation*}

Fix $s\geq 0, c>0$. For $M \in \N^*$, define
\begin{equation} \label{eq.Ec} \nu:=\frac{c}{M (\ln M)^s}, \quad
\tPhi_\nu:=\tcS(\phi_\nu).
\end{equation}
For $t\geq 0$, we then define
$$\tA_t(M):= \{(\a,z) \in \R^d\times \R : \tPhi_{\nu}(g_t \Lambda_{\a,z}) \geq 1\}$$

It is readily checked that  $(\a,z) \in \tA_t(M)$ if and only if there exists  $k=(k_1,\ldots,k_d)$ such that $2^{t}<|k|\leq 2^{t+1}$ and that
\begin{equation*}\label{eq.M} \exists m, \quad   |k|^d |z+\langle k, \alpha \rangle+m|\leq \frac{c}{M (\ln M)^s }.\end{equation*}

If $\a$ is such that $\tPhi_{\nu}(g_t \Lambda_{\a,z}) \leq 1$ for every $t \in \N$, then we get that $(\a,z)$ is
 $(r,s)$-approximable if and only if there exists
 infinitely many $M$ for which there exists $0<t_1<t_2<\ldots<t_r \leq M$ satisfying
 $\DS (\a,z) \in \bigcap_{j=1}^r \tA_{t_j}(M).$

But in general, for $\a$ and $t\leq M$ such that $(\a,z) \in \tA_t(M)$, there may be multiple solutions $k$ such that $2^{t}<|k|\leq 2^{t+1}$ for the same $t$. As in the case of Theorem \ref{theo.KG} we have to deal with this issue.

The following proposition shows that almost surely on $(\a,z)$, multiple solutions do not occur. Its proof is based on Rogers identity for the second moment of the Siegel transforms.

\begin{Prop} \label{prop_multiple_affine}For almost every $(\a,z)\in \R^d \times \R$, we have that for every $M$ sufficiently large, for every $t\in [0,M]$, it holds that  $\tPhi_{\nu}(g_t \Lambda_{\a,z}) \leq 1$
\end{Prop}

Hence, Theorem \ref{theo.KG_affine} is equivalent to the following.

\begin{Thm}  \label{Theo_Lattice_affine}
If $rs\leq 1$, then for almost every $(\a,z)\in \T^d \times \T$, there exists infinitely many $M$ for which there exists $0<t_1<t_2<\ldots<t_r \leq M$ satisfying
$$\a \in \bigcap_{j=1}^r \tA_{t_j}(M).$$

If $rs> 1$, then for almost every $(\a,z)\in \T^d \times \T$, there exists at most finitely  many $M$ for which there exists $0<t_1<t_2<\ldots<t_r \leq M$ satisfying
$$\a \in \bigcap_{j=1}^r \tA_{t_j}(M).$$
\end{Thm}

 \subsection{  Modifying the initial distribution: inhomogeneous case.}
 Since the horocycle directions of $\Lambda_{\a,z}$, $(\a,z) \in \T^d \times \T$ account for all the strong unstable direction of the diagonal flow $g_t$ acting on $\tcM$, we can transform the requirement of Proposition \ref{prop_multiple_affine}  and Theorem \ref{Theo_Lattice_affine} into a problem of multiple recurrence of the diagonal action $g_t$ when applied to a random lattice in $\tcM$.

We define
$$\tB_t(M):= \{\tcL \in \tcM : \tPhi_{\nu}(g_t \tcL) \geq 1\}$$

Our goal becomes to prove the following.

\begin{Prop} \label{prop_multiple_Haar_affine}
For $\tilde{\mu}$-almost every $\tcL \in \tcM$, we have that for every $M$ sufficiently large, for every $t\in [0,M]$, it holds that  $\tPhi_{\nu}(g_t \tcL) \leq 1$
\end{Prop}

\begin{Thm}  \label{Theo_Lattice_Haar_affine}

If $rs\leq 1$, then for $\tilde{\mu}$-almost every $\tcL \in \tcM$, there exists infinitely many $M$ for which there exists $0<t_1<t_2<\ldots<t_r \leq M$ satisfying
$$\tcL \in \bigcap_{j=1}^r \tB_{t_j}(M).$$

If $rs> 1$, then for $\tilde{\mu}$-almost every $\tcL \in \tcM$, there exists at most finitely  many $M$ for which there exists $0<t_1<t_2<\ldots<t_r \leq M$ satisfying
$$\tcL \in \bigcap_{j=1}^r \tB_{t_j}(M).$$
\end{Thm}

{\subsection{  Proofs of
Proposition  \ref{prop_multiple_Haar_affine}
and Theorem \ref{Theo_Lattice_Haar_affine}.} Again, the proofs of
Proposition  \ref{prop_multiple_Haar_affine}
and Theorem \ref{Theo_Lattice_Haar_affine} are very similar to the proofs of  their counterpart in the homogeneous case, Proposition  \ref{prop_multiple_Haar}
and Theorem \ref{Theo_Lattice_Haar}.

Similarly to the homogeneous case,
we want to apply Corollary \ref{cor.mixing}.
For the system $(f,X,\mu)$ we take $(g_1,\widetilde{\cM},\tilde{\mu})$, where $\tilde{\mu}$ is the Haar measure on $\widetilde{\cM}$.
For the targets, we take $\fA_\fr=\{\tcL : \tPhi_{\fr}(\tilde\cL) \geq 1\}$. Observe that from the invariance of the Haar measure  by $g_t$ we have that $\tilde{\mu}(\fA^t_{\fr})=\tilde{\mu}(\fA_{\fr})$ for any $t$.

The only difference in the proof of Proposition  \ref{prop_multiple_Haar_affine}
and Theorem \ref{Theo_Lattice_Haar_affine} compared to that of Proposition  \ref{prop_multiple_Haar}
and Theorem \ref{Theo_Lattice_Haar}, is in the application of Rogers identities to prove  Proposition  \ref{prop_multiple_Haar_affine} as well as in the proof of (Mov) that is part of the proof of Theorem \ref{Theo_Lattice_Haar_affine}.

We explain this difference now.

In fact, Rogers identities are slightly simpler in the affine case, where there is no need to pay a special attention to the multiples of a vector in the affine lattice. Recall \eqref{SiegelAffine}
Rogers identities for affine lattices (read \cite{Mar1})
\begin{align*}
\E(\tcS(f))&=\int_{\R^{d+1}} f(u)du\\
\E(\tcS(f)^2)&=\left(\int_{\R^{d+1}} f(u)du\right)^2+
\int_{\R^{d+1}} f^2(u)du.
\end{align*}
(The idea behind the proof for the second moment identity is that the linear functionals on  the space of continuous functions on $\R^{d+1} \times \R^{d+1}$ that are $SL_{d+1} (\R)\ltimes \R^{d+1}$ invariant
can be identified to invariant measures on  $\R^{d+1} \times \R^{d+1}$
by the action of $SL_{d+1} (\R)\ltimes \R^{d+1}$. But the orbits of the latter action decompose into pairs of independent vectors and pairs of equal vectors.)

Now for the proof Proposition  \ref{prop_multiple_Haar_affine},  we have that
$$
\tilde\mu \left(\tPhi_{{\nu}}(\tilde\cL) > 1\right)\leq  \E\left(\tPhi_{\nu}^2(\tilde\cL) -\tPhi_{\nu}(\tilde\cL)\right)\leq  \left(\int_{\R^{d+1}} \tphi_{\nu}(u) du \right)^2 \leq C c^2 M^{-2}(\ln M)^{-2s}  $$
and
Proposition  \ref{prop_multiple_Haar_affine} then follows by a Borel Cantelli argument exactly as in the regular lattices case.

For the proof of (Mov) in the affine case we write for $\tau\geq 1$
\begin{equation*}
  \tE^\tau_\nu=\left\{
  (x,y)
  \in\R^{d}\times\R\mid   2^{-\tau} |x| \in [1,2],
  |x|^{d} |y| \in [0,\nu]\right\}
\end{equation*}
and for
$\tphi_\nu^\tau$ the indicator function of $\tE^\tau_\nu$, observe that
$$ \tilde{\mu}(\fA_\fr \cap g_{-\tau} \fA_\fr)
\leq \E\left(\tPhi_\fr \left(\tPhi_\fr \circ g_\tau\right)\right)=  \int_{\cM} \sum_{e_2, e_1\in \tcL } \tphi_\fr(e_1) \tphi_\fr^\tau(e_2)
 d\tmu(\tcL) $$
 $$=  \int_{\tilde\cM} \sum_{e_2 \neq e_1\in \tcL } \tphi_\fr(e_1) \tphi_\fr^\tau(e_2)
 d\tmu(\tcL) =  \left(\int_{\R^{d+1}} \tphi_\fr(u) du \right)^2   \leq C \tilde{\mu}(\fA_\fr)^2
 $$
 which is stronger than the required (Mov).
\hfill $\Box$}
\bigskip

\subsection{  Multiple recurrence for toral translations.}
\begin{proof}[Proof of Theorem \ref{ThReturnsTransl}.]
\noindent {\sc Proof of part $(a)$.} We begin with several reductions. Let $z=x-y.$ Then
$d(x, y+k\alpha)=d(z, k\alpha).$ Accordingly denoting
$\hat{d}_n^{(r)}(z, \alpha)$ to be the $r$-th smallest among $\DS \{d(z, k\alpha)\}_{k=0}^{n-1}$
we need to show that for almost every $(z, \alpha)\in (\Tor^d)^2$ we have
\begin{equation}
\label{TransD1}
\lim\sup_{n\to\infty} \frac{|\ln\hat{d}^{(1)}_n(z, \alpha)|-\frac{1}{d} \ln n}{\ln \ln n}=\frac{1}{d},
\end{equation}
\begin{equation}
\label{TransD2}
\lim\sup_{n\to\infty} \frac{|\ln\hat{d}^{(r)}_n(z, \alpha)|-\frac{1}{d} \ln n}{\ln \ln n}=\frac{1}{2d}, \text{ for }r\geq 2.
\end{equation}
Next we claim that it suffices to prove \eqref{TransD2} only for $r=2.$ Indeed, since $\hat{d}_n^{(r)}$ is non decreasing in $r,$
\eqref{TransD2} with $r=2$ implies that for $r>2,$
$$ \lim\sup_{n\to\infty} \frac{|\ln\hat{d}^{(r)}_n(z, \alpha)|-\frac{1}{d} \ln n}{\ln \ln n}\leq \frac{1}{2d}. $$
To get the upper bound, suppose that $\hat{d}_n^{(2)}(z, \alpha)\leq \eps.$ Then there are $0\leq k_1<k_2<n$ such that
$k_j \alpha\in B(z, \eps).$ Let $k=k_2-k_1.$ Then
$$k_2+s\alpha\in B(z, (1+2s)\eps)$$
for $s=1,\cdots,r-2.$
Thus
$\hat{d}_{(r-1) n}^{(r)}(z, \alpha)\leq (2r-1) \hat{d}_n^{(2)}(z, \alpha).$ Taking limit superior, we obtain that if \eqref{TransD2} holds for $r=2$
then it holds for arbitrary $r.$ In summary, we only need to show
\eqref{TransD1} and
\begin{equation}
\label{TranD2}
\lim\sup_{n\to\infty} \frac{|\ln\hat{d}^{(2)}_n(z, \alpha)|-\frac{1}{d} \ln n}{\ln \ln n}=\frac{1}{2d}.
\end{equation}

The proofs of \eqref{TransD1} and \eqref{TranD2} are similar to but easier than the proof of Theorem \ref{theo.KG_affine}
so we only explain the changes.
First, it is suffices to take limit superior, for $n$ of the form $2^M$ since for $2^{M-1}\leq n\leq 2^M$ we have
$$ \hat{d}_{2^M}^{(r)}(z,\alpha)\leq
\hat{d}_n^{(r)}(z,\alpha)\leq \hat{d}_{2^{M-1}}^{(r)}(z,\alpha). $$
Let $\nu_M=M^{-s}$ for a suitable $s$ and
\begin{equation} \hat{E}_\nu=\{e=(e', e'')\in \R^d\times \R: ||e'||\leq \nu, e''\in (0,1]\} . \label{eq.Enu} \end{equation}
Then a direct inspection shows that
$$ \hat{d}_{2^M}^{(r)}(z,\alpha)\leq \nu_M\Leftrightarrow
\tS(\one_{\hat{E}_{\nu_M}})(\hat{g}_M \hat{\Lambda}_{\alpha, z} )\geq r,$$
where $\tS$ is defined by \eqref{SiegelAffine},  $\hat{g}_M=g_{-M/d}$ for $g$ given by \eqref{DefGT}, and $\hat{\Lambda}_{\alpha, z}$ is defined by $\hat{\Lambda}_{\a,z}=(\hat{\Lambda}_\a, (z,0))$ for
\[ \hat{\Lambda}_\a=
\begin{pmatrix} \text{Id}_d & \a \\ 0 & 1\end{pmatrix}.
\]

Recall $\tcM$  denoted by the space of $d+1$ dimensional unimodular affine lattices and $\tilde{\mu}$ the Haar measure on $\tcM.$ As in the proof of Theorem \ref{theo.KG_affine} one can show that
$\DS  \tS(\one_{\hat{E}_{\nu_M}})(\hat{g}_M \hat{\Lambda}_{\alpha, z} )\geq r$ infinitely often
for almost every $(z, \alpha)$ if and only if
$\DS  \tS(\one_{\hat{E}_{\nu_M}})(\hat{g}_M \tcL)\geq r$ infinitely often
for almost every $\tcL\in\tcM.$
Thus we need to show that for almost every $\tcL\in\tcM$
\begin{equation}
\label{ST1IO}
\tS(\one_{\hat{E}_{\nu_M}})(\hat{g}_M \tcL)\geq 1\text{ infinitely often if } s<\frac{1}{d},
\end{equation}
\begin{equation}
\label{ST1FO}
\tS(\one_{\hat{E}_{\nu_M}})(\hat{g}_M \tcL)\geq 1\text{ finitely often if } s>\frac{1}{d},
\end{equation}
\begin{equation}
\label{ST2IO}
\tS(\one_{\hat{E}_{\nu_M}})(\hat{g}_M \tcL)\geq 2\text{ infinitely often if } s<\frac{1}{2d},
\end{equation}
\begin{equation}
\label{ST2FO}
\tS(\one_{\hat{E}_{\nu_M}})(\hat{g}_M \tcL)\geq 2\text{ finitely often if } s>\frac{1}{2d}.
\end{equation}

To prove \eqref{ST1IO}--\eqref{ST2FO}, we need the following fact.

\begin{Lem}
\label{LmSTNZ}
  (a) $\tmu\left(\tS(1_{\hat{E}_\nu})=1\right)=c_d \nu^d(1+\cO(\nu^{2d})), $

  (b) $c' \nu^{2d} \leq \tmu\left(\tS(1_{\hat{E}_\nu})\geq 2\right)\leq c'' \nu^{2d}.$
  \end{Lem}

Before we give the proof of the lemma, we see how
it allows to obtain \eqref{ST1IO}--\eqref{ST2FO} and  finish the proof of part $(a)$ of Theorem \ref{ThReturnsTransl}.

Indeed, Lemma \ref{LmSTNZ} shows that
$$ \sum_M \tmu\left(\tS(1_{\hat{E}_{\nu_M}})=1\right)=\infty \iff s\leq\frac{1}{d}, \quad
  \sum_M \tmu\left(\tS(1_{\hat{E}_{\nu_M}})\geq 2\right)=\infty \iff s\leq\frac{1}{2 d}. $$

  From there,  \eqref{ST1IO}--\eqref{ST2FO} follow from the the classical Borel Cantelli Lemma, that is,
from the case $r=1$ in our Theorem \ref{ThMultiBC}.\footnote{We note that in case $r=1$ Theorem \ref{ThMultiBC}
is a minor variation of standard dynamical Borel Cantelli Lemmas such as e.g.,
the Borel Cantelli Lemma of \cite{KM99}.}
For this, denote $\hat{\Phi}_\nu=\tS(1_{\hat{E}_\nu}),$ and observe that the verification of the conditions of Definitions \ref{def.mixing}, and Definition \ref{def.targets} for the targets $\fA_\fr=\{\tilde\cL : \hat{\Phi}_{\fr}(\tilde\cL) \geq 1\}$ is very similar to the proof of Theorem \ref{theo.KG_affine} so we omit it.

\begin{proof}[Proof of Lemma \ref{LmSTNZ}]  we get by Rogers
$$ \EXP(\hat{\Phi}_\nu)=c_d \nu^d, \quad \EXP(\hat{\Phi}^2_\nu-\hat{\Phi}_\nu)=\left(c_d \nu^d\right)^2. $$
  It follows that
  $$\tmu(\hat{\Phi}_\nu\geq 2)\leq \EXP(\hat{\Phi}_\nu^2-\hat{\Phi}_\nu)/2\leq C \nu^{2d}$$
proving the upper bound of part (b).

In addition
  $$\EXP\left(\hat{\Phi}_\nu \one_{\hat{\Phi}_\nu\geq 2}\right)\leq \left( c_d \nu^d\right)^2$$
so that
\begin{equation}
\label{1pointAL}
\tmu(\hat{\Phi}_\nu=1)=\EXP(\hat{\Phi}_\nu)-\EXP(\hat{\Phi}_\nu\one_{\hat{\Phi}_\nu\geq 2})=c_d \nu^d+\cO\left(\nu^{2d}\right).
\end{equation}
This proves part (a).

To prove the lower bound in part (b) we need the following estimate.
{Denote $\cL_{prime}$  the set of prime vectors in $\cL$ for $\cL\in \cM=SL_{d+1}(\R)/SL_{d+1}(\Z).$}
Let
\begin{align*} \bar{E}_1&=\left\{(e', e'')\in \R^d\times \R:
{|e'|\in\left[\frac{\nu}{10},\frac{\nu}{5}\right],\, |e''|\leq \frac {1} {10}}\right\}, \\ \bar{E}_2&=\left\{(e', e'')\in \R^d\times \R: { |e'|\leq\frac{\nu}{5},\,|e''|\leq \frac {1} {10} }\right\}, \\
 \cA&=\left\{\cL\in \cM: \Card\left(\cL_{prime}\cap \bar{E}_1\right)=
\Card\left(\cL_{prime} \cap \bar{E}_2\right)=1\right\}. \end{align*}

\smallskip
\noindent{\sc Claim.}{\it  We have \begin{equation}
\label{ResLat}
\mu(\cA)=c \nu^{d}.
\end{equation}
}

\smallskip

Assume the claim holds.  {Denote $\tilde z=(z,0).$} For $\cL\in \cA$, the fundamental domain of $\R^{d+1}/\cL$ can be chosen to contain
$$\bar{E}_3=\left\{(e', e'')\in \R^d\times \R: |e'|\leq \frac{\nu}{100} , |e''|\leq  \frac{1}{100}\right\}. $$
We thus have
  $$\mu\left((\cL+\tilde z): \Card\left((\cL+\tilde z)\cap \hat{E}_\nu\right)\geq 2\right)\geq
  \mu\left(\cA\right) \mu\left(\Card\left((\cL+\tilde z)\cap \hat{E}_\nu\right)\geq 2|\cA\right) $$
  $$ \geq
  \mu(\cA)\mu(z\in \bar{E}_3)\geq c' \nu^{2d}. $$
This gives the lower bound in part (b) of Lemma \ref{LmSTNZ}. To complete the proof,
we now give the
\begin{proof}[Proof of the claim]
We consider the cases $d>1$ and $d=1$ separately.

In case $d>1,$ denote $\Psi_j=\tS(\one_{\bar{E}_j})$ for $j=1,2.$
By Rogers identities,
$$ \EXP\left(\Psi_1 \right)=\frac {1} {10} c_d \nu^d,
\quad
\EXP\left(\Psi_1^2-\Psi_1 \right)=\left(\frac {1} {10} c_d \nu^d\right)^2.$$
Thus arguing as in the proof of \eqref{1pointAL} we conclude that
\begin{equation}
\label{2points1}
\mu(\Psi_1=1)=\frac {1} {10} c_d \nu^d+\cO\left(\nu^{2d}\right).
\end{equation}
Rogers identities also give
$$ \EXP(\Psi_1 (\Psi_2-\Psi_1))=\cO\left(\nu^{2d}\right) .$$
Hence
\begin{equation}
\label{2points2}
\mu\left(\Card (\cL_{prime} \cap \bar{E}_1)\geq 1\text{ and } \Card \left(\cL_{prime} \cap
\left(\bar{E}_2\setminus \bar{E}_1\right)\right)\geq 1\right)=\cO\left(\nu^{2d}\right).
\end{equation}
Combining \eqref{2points1} and \eqref{2points2} we obtain \eqref{ResLat} for $d>1.$

In case $d=1$ we still have $\EXP(\Psi_1)=c \nu+\cO(\nu^2).$ On the other hand, for $d=1$ we have
$\Card\left(\cL_{prime}\cap \bar{E}_2\right)\leq 1$ since $\cL$ is unimodular. Thus

$\DS \EXP(\Psi_1)=\mu(\Psi_1=1)=\mu(\Psi_1=1\text{ and } \Psi_2-\Psi_1=0)=c \nu. $
\end{proof}

This completes the proof of Lemma \ref{LmSTNZ} and thus of part $(a)$ of Theorem \ref{ThReturnsTransl}.
\end{proof}

\medskip

 \noindent {\sc Proof of part $(b)$.}  It is clear that for any $r$, if $\bar \cE_{r}$ is not empty then it is equal to $M$.  The fact that
$\bar \cE_{1}=M$ implies that $\bar \cE_{r}=M$ for all $r$ is exactly similar to the implication of \eqref{TransD2} from \eqref{TranD2}, so we just focus on showing that $\bar \cE_{1}=M$.  Adapting the beginning of the proof of part $(a)$ to the current homogeneous setting, we see that what we want to prove boils down to showing that for almost every $\cL\in\cM$
\begin{align}
\label{ST1IOh}
\mathcal S(\one_{E_{\nu_n}})(\hat{g}_n \cL)&\geq 1\text{ infinitely often if } s<\frac{1}{d},\\
\label{ST1FOh}
\mathcal S(\one_{E_{\nu_n}})(\hat{g}_n \cL)&\geq 1\text{ finitely often if } s>\frac{1}{d},
\end{align}
where $E_{\nu_n}$ is as in \eqref{eq.Enu}, and $\mathcal S$ designates the Siegel transform as in \eqref{eq.siegel}. By Rogers identity, Lemma \ref{LmRI}(a), we have that
$\EXP\left(\mathcal S(\one_{E_{\nu_n}})\right)=cn^{-sd}$, hence \eqref{ST1IOh} and \eqref{ST1FOh} follow by classical Borel Cantelli Lemma
(see for example the Borel Cantelli Lemma of \cite{KM99})
or by the case $r=1$ of our Theorem \ref{ThMultiBC} .

This completes the proof of Theorem \ref{ThReturnsTransl}.
\end{proof}

\subsection{  Notes.}
A classical Khintchine--Groshev Theorem
is given by \eqref{EqKG1}--\eqref{EqKG2}.
A lot of interest is devoted to extending this
result to $\alpha$ lying in a submanifold of $\R^d$ (see e.g. \cite{BBV13, BD99}).
The applications of dynamics to Diophantine approximation are based on Dani correspondence
\cite{Da86}. In particular, \cite{KM98} discusses Khintchine--Groshev type results on manifolds
using dynamical tools.
The use of Siegel transform as a convenient analytic tool for applying Dani correspondence
can be found in \cite{Mar1}.
Surveys on applications of dynamics to metric Diophantine approximations
include \cite{BM00, BG17,  DF15, EEEKLMMY, Es98, GN15,
Kl99, KSS02, Mar07}. Limit Theorems for Siegel transforms are discussed in
\cite{AGT15, APT16, BG-CLTDA, db_poisson, DFV_linear}.

\section{  Extreme values.}
\label{ScExtreme}
\subsection{  From hitting times to extreme values.}
Here we describe applications of our results to extreme value theory.

Let $(f,M,\mu)$ be as in Definition \ref{def.mixing}. Recall that  the sets $\cG_r$ and $\cH$
are introduced in Defenitions \ref{DefG} and \ref{DefH} respectively.
Recall also that under the conditions of
Theorems \ref{thm3} and \ref{ThTop}  $\mu(\cG_r)=1$ and
$\cH$ contains a residual set.

Given a function $\phi$ and a point $y\in M,$
let $\phi^{(r)}_{n}(y)$ be the $r$-th minimum among the values
$\{\phi(f^j y)\}_{j=1}^n.$

\begin{Thm}
\label{ThMultiLogExt}
(a) Suppose $f$ is $(2r+1)$-fold exponentially mixing preserving a smooth measure $\mu.$ Then

(i) There is a set $\cG$ of full measure in $M$
such that  if
$\phi$ is a function with a unique non degenerate
minimum at $x\in \cG,$
then for almost every $y\in M,$
	$$
	\limsup_{n\rightarrow \infty} \frac{\left|\ln \left(\phi_n^{(r)}(y)-\phi(x)\right)\right|-\frac{2}{d}\ln n}{\ln \ln n}
        = \frac{2}{rd}.
	$$
(ii) If $\cG_1=M$ and the periodic orbits of $f$ are dense, then there is
a dense $G_\delta$ set $\cH \subset M$,
such that  if
$\phi$ is a function with a unique non degenerate
minimum at $x\in \cH,$
then for almost every $y\in M,$
$$
	\limsup_{n\rightarrow \infty} \frac{\left|\ln \left(\phi_n^{(r)}(y)-\phi(x)\right)\right|-\frac{2}{d}\ln n}{\ln \ln n}
        ={\frac{2}{d}}.
$$

(b) If $f$ is an expanding map of $\T$ and $\mu$ is a non-conformal Gibbs measure of dimension $\dd,$ $\lambda$ is the Lyapunov exponent of $\mu,$ then there is a set $\cG_\mu$ with $\mu(\cG_\mu)=1,$ such that if
$\phi$ is a function with a unique non degenerate
minimum at $x\in \cG_\mu,$
then for $\mu$--almost every $y\in M,$
$$ \lim\sup_{n\to\infty} \frac{\left|\ln \left(
\phi_n^{(r)}(y)-\phi(x)\right)\right|-\frac{2}{\dd} \ln n}{\sqrt{2(\ln n) (\ln \ln \ln n)}}=\frac{2\sigma}{\dd\sqrt{\dd\lambda}}, $$
where $\sigma$ given by \eqref{GibVar}.

(c) Part (a) remains valid for the geodesics flow on a compact  $(d+1)-$dimensional
manifold $\cQ$ and functions $\phi: \cQ\to\reals$ which have unique non-degenerate minimum
at some point on $\cQ.$
(In this case $\phi_r(y)$ is the $r$-th local minimum of the map
$t\mapsto \phi(q(t))$ where $(q(t), v(t))$ is the geodesic starting at $q$ with velocity $v.$)

(d) For toral translations we have that for almost all $\alpha$ and almost all $y$ we have
$$	\limsup_{n\rightarrow \infty} \frac{\left|\ln\left( \phi_n^{(r)}(y)-\phi(x)\right)\right|-\frac{2}{d}\ln n}{\ln \ln n}
        =\begin{cases} \frac{2}{d}  \quad \text{ if } r=1, \\    \frac{1}{d}   \quad \text{ if }  r\geq 2. \end{cases} $$
\end{Thm}

\begin{proof}
At a non-degenerate minimum $x$ we have that for $y$ close to $x$
\begin{equation}
\label{D-Phi}
K^{-1}d^2(x,y)\leq \phi(y)-\phi(x)\leq K d^2(x,y)
\end{equation}
so  part $(i)$ of (a)  holds for $x\in \cG_r$ and part $(ii)$ of (a) holds for $x \in \cH$ as defined in Theorems \ref{thm3} and \ref{ThTop}. Part (b) follows from Theorem \ref{ThRecGibbs}. Part (c) follows from Theorem \ref{ThmGeoBall}, and part (d) follows from Theorem \ref{ThReturnsTransl}.
\end{proof}

\begin{Thm}
\label{ThPoisExt}
Under the assumptions of Theorem \ref{ThMultiLogExt}(a)
or Theorem \ref{ThMultiLogExt}(d) there is a set of points $x$ of full measure such that
if $\phi$ has a non-degenerate minimum at $x$ then the
process
$$ \frac{\phi^{(1)}_n(y)-\phi(x)}{\fr^2}, \frac{\phi^{(2)}_n(y)-\phi(x)}{\fr^2}, \dots ,
\frac{\phi^{(r)}_n(y)-\phi(x)}{\fr^2}, \dots
$$
with $n=[\tau \fr^{-d}]$
converges as $\fr\to 0$ to the Poisson process on $\reals^+$ with measure $\gamma(\phi) \tau \frac{d}{2} t^{\frac{d}{2}-1}dt,$ where $\gamma(\phi)>0$ depends on $x$ and $\phi$.
\end{Thm}

\begin{proof}
Note that \eqref{D-Phi} does not provide enough information to deduce the result from
\eqref{ScaledHits} of Theorem \ref{PrEasy}. However, for any choice of
$r_1^-<r_1^+<r_2^-<r_2^+<\dots <r_s^-<r_s^+$, consider the targets
\begin{equation}
\label{PhiTargets}
 \fA^{n, j}=\left\{y: \phi(y)-\phi(x)\in \left[{r_j^-}{\fr^2}, {r_j^+}{\fr^2}\right]\right\},
\end{equation}
that satisfy
$$\lim_{\rho \to 0} \tau \fr^{-d}\mu(\fA^{n, j})=  \tau \gamma(\phi)({(r_j^+)}^{\frac d 2}-{(r_j^-)}^{\frac d 2})=\tau \gamma(\phi) \int_{r_j^-}^{r_j^+} \frac{d}{2} t^{\frac{d}{2}-1}dt.
$$
Conditions $\widetilde{(M1)}_r$ and $(M2)_r$ from \S \ref{SSDetails}
can easily be checked for the targets $\Omega^{n,j}$
using the results of  Section \ref{SSMultExpMix}.
Since (Mov) for targets
\eqref{PhiTargets} follows from (Mov) for balls,
only (Appr) needs to be checked
but the latter follows immediately from Lemma \ref{RemLip}. We can thus apply Theorem \ref{ThPoissonProcess} and conclude the Poisson limit.
\end{proof}

Next, we consider functions of the form
\begin{equation}
\label{Mero}
 \psi(y)=\frac{c}{d^s(x,y)}+\widetilde{\psi}(y), \quad\text{where}\quad c<0\quad\text{and}\quad
  \widetilde{\psi}\in Lip(M).
\end{equation}

\begin{Thm}
\label{ThEVPower}
Let $f$ be $(2r+1)$-fold exponentially mixing. Then

\noindent (a)
There is a set $\cG$ or full measure such that
if $\psi$ satisfies \eqref{Mero} with $x\in \cG$ then for almost all $y$
	$$
	\limsup_{n\rightarrow \infty} \frac{\ln|\psi_n^{(r)}(y)|-\frac{s}{d}\ln n}{\ln \ln n}
        = \frac{s}{rd}.
	$$

\noindent (b) There is a $G_\delta$  set $\cH$ such that
if $\psi$ satisfies \eqref{Mero} with $x\in \cH$ then for almost all $y$
	$$
	\limsup_{n\rightarrow \infty} \frac{\ln|\psi_n^{(r)}(y)|-\frac{s}{d}\ln n}{\ln \ln n}
        = \frac{s}{d}.
	$$

\noindent (c) If $x\in \cG$ then
$$ \frac{\fr^s \psi^{(1)}_n(y)}{c} ,
\frac{\fr^s \psi^{(2)}_n(y)}{c}, \dots,  \frac{\fr^s \psi^{(r)}_n(y)}{c} , \dots  \quad \text{where}\quad
n=\tau \fr^{-d} $$
converges as $\fr\to 0$ to the Poisson process on $\reals^+$ with measure
$ \frac{d\tau\gamma(x)}{s}  t^{-(d/s)-1}dt.$
\end{Thm}
The proofs of the above results is similar to the proofs of Theorem \ref{ThMultiLogExt} and
\ref{ThPoisExt} so we will leave them to the readers.

The next result is an immediate consequence of Theorems \ref{ThPoisExt} and \ref{ThEVPower}(c).
\begin{Cor}
\label{CrFrechetWeibull}
(a) {\sc (Fr\'echet Law for smooth functions)} If $f$ is $(2r+1)$-fold exponentially mixing, $\phi$ is a smooth
function with non-degenerate minimum at some $x\in \cG$ then there is $\sigma=\sigma(x)$ such that
for each $t>0$
$$ \lim_{n\to\infty} \mu(y: \phi^{(1)}_n(y)>n^{-2/d} t)=e^{-\sigma t^{d/2}}.  $$

(a) {\sc (Weibull Law for unbounded functions)} If $f$ is $(2r+1)$-fold exponentially mixing, $\phi$ is given by
\eqref{Mero} with
 $x\in \cG$ then there is $\sigma=\sigma(x)$ such that
for each $t>0$
$$ \lim_{n\to\infty} \mu(y: \phi^{(1)}_n(y)> -n^{-s/d} t)=e^{-\sigma t^{-d/s}}.  $$
\end{Cor}

\subsection{  Notes.}
A classical Fisher--Tippett--Gnedenko theorem
says that for independent  identically distributed
random variables the only possible limit distributions of normalized extremes
are the Gumbel distribution
the Fr\'echet distribution, or the Weibull distribution.
Corollaries \ref{CrGeoGumbel} and \ref{CrFrechetWeibull}(a) and (b) provide typical examples where
one can encounter each of these three  types.
We refer to \cite{LLR83} for the proof of Fisher--Tippett--Gnedenko theorem
as well as for extensions of this theorem
to weakly dependent random variables. The weak dependence
conditions used in the book have a similar sprit to our conditions (M1) and (M2).
More discussions about relations of extreme value theory to
Poisson limit theorems in the context of dynamical systems can be found in \cite{FFM18}.
The book \cite{LFdFdFHKMTV} discusses extreme value theory for dynamical systems and lists
various applications.
One application of extreme value theory, is that for non-integrable functions,
such as described in Theorem \ref{ThEVPower} above, the growth of ergodic
sums are dominated by extreme values, see \cite{AN, CN17, DV86, KS19, KM92, Mo76}
and references wherein.

\appendix

\section{Multiple exponential mixing.}
\label{AppExpMix}
\subsection{  Basic properties}

Let $f$ be a smooth map of a compact manifold $M$ preserving a smooth probability measure $\mu.$
In the dynamical system literature, for $r\geq 1,$ $f$  is called $(r+1)$-fold exponentially mixing if
there are constant $s, \brC$ and $\brtheta<1$ such that for any $C^s$ functions $A_0, A_1,\dots, A_r$
for any $r$ tuple $k_1<k_2<\dots<k_r$
 \begin{equation}\label{MultMix}
 \left| \int \prod_{j=0}^{r} \left(A_j \circ f^{k_j} \right) d \mu  -
 \prod_{j=0}^{r}\int A_j d\mu  \right| \leq \brC \brtheta^m \prod_{j=0}^{r} \|A_j\|_{C^s},
 \end{equation}
 where $\DS m=\min_j (k_j-k_{j-1})$ with $k_0=0.$

In this paper we need to consider a larger class of functions, namely we need that
there are constants $s, C$ and $\theta<1$ such that for any
$B\in C^s(M^{r+1})$ we have
 \begin{equation}\label{MultiMixRetPH}
\left|\int B(x_0,f^{k_1}x_0,\cdots,f^{k_r}x_0) d\mu(x_0)-\int B(x_0,\cdots,x_r)d\mu(x_0)\cdots d\mu(x_r)\right|\leq C_s\brtheta^m\left\|B\right\|_{C^s}.
 \end{equation}

In this section we show equivalence of \eqref{MultMix} and \eqref{MultiMixRetPH}.
We use the following fact.

\begin{Rem}\label{lowreg}
If \eqref{MultMix} holds for some $s$ then it holds for all $s$ (with different $\brtheta$).
The same applies for \eqref{MultiMixRetPH}.

\end{Rem}
Indeed suppose that \eqref{MultiMixRetPH} for some $C^s$ functions.
Pick some $\alpha<s.$ We claim that it also holds for $C^\alpha$ functions.
Indeed pick a small $\eps$ and approximate a $C^\alpha$ function $B$ with
$\|B\|_{C^\alpha}=1$ by a
$C^s$ function $\brB$, so that (assuming that $m$ is large)
$$ \|B-\brB\|_{C^0}\leq e^{-\eps \alpha m}, \quad \|\brB\|_{C^s}\leq  e^{\eps s m} .$$
Then
\begin{align*}
& \int B(x_0,f^{k_1}x_0,\cdots,f^{k_r}x_0) d\mu(x_0)=
\int \brB(x_0,f^{k_1}x_0,\cdots,f^{k_r}x_0) d\mu(x_0)+O\left(e^{-\eps \alpha m}\right)\\
= &
\int \brB(x_0, x_1, \dots x_r) d \mu(x_0) d\mu(x_1)\dots d\mu(x_r)+O\left( e^{-\eps \alpha m}\right)
+O\left(\theta^m e^{\eps s m}\right)\\
=&
\int B(x_0, x_1, \dots x_r) d \mu(x_0) d\mu(x_1)\dots d\mu(x_r)+O\left( e^{-\eps \alpha m}\right)
+O\left(\theta^m e^{\eps s m}\right).
\end{align*}
and the second error term is exponentially small if $\eps$ is small enough.
The argument for \eqref{MultMix} 
is identical.

We now ready to show that \eqref{MultMix} implies \eqref{MultiMixRetPH}.

\begin{Thm}\label{MultiMixRet}
Suppose that \eqref{MultMix} holds and $s$ is sufficiently large. Then \eqref{MultiMixRetPH} holds.
\end{Thm}

\begin{proof}[Proof of Theorem \ref{MultiMixRet}]
Since $B\in C^s(M^{r+1})$ it also belongs to Sobolev space $H^s(M^{r+1}).$ Hence we can decompose
$$ B=\sum_\lambda b_\lambda \phi_\lambda $$
where $\phi_\lambda$ are eigenfunctions of Laplacian on $M^{r+1}$ with eigenvalues
$\lambda^2$ and  $\|\phi_\lambda\|_{L^2}=1.$
The eigenfunctions $\phi_\lambda$ are of the form
$$ \phi_\lambda(x_0, x_1, \dots, x_r)=\prod_{j=0}^r \psi_j(x_j) $$
where $\Delta_M \psi_j=\zeta_j^2 \psi_j$ and $\lambda^2=\sum_j \zeta_j^2.$
Recall that by Sobolev Embedding Theorem for compact manifolds,
$H^s(M)\subset C^{s-\frac{d}{2}-1-\eps}(M)$ for any $\eps>0.$
Since $\|\psi_j\|_{H^s}=\zeta_j^s$ we have
$$\|\psi_j\|_{C^1}\leq C_u \zeta_j^{u} \leq C_u \lambda^u \quad\text{if}\quad u>1+\frac{d}{2}.$$
It follows from \eqref{MultMix} that if $\phi\not\equiv 1$ then
$$ \left|\int \phi_\lambda(x, f^{k_1} x, \dots , f^{k_r} x) d\mu(x)-\prod_{j=0}^r \int\psi_j d\mu\right|
\leq C \lambda^{u(r+1)} \theta^m.$$
Therefore
$$
 \left|\int B(x, f^{k_1} x, \dots , f^{k_r} x) d\mu(x)-\int B(x_0,\cdots,x_r)d\mu(x_0)\cdots d\mu(x_r)\right|$$
$$ \leq C \theta^m \sum_\lambda b_\lambda \lambda^{u(r+1)} \leq C \theta^m ||B||_{H^{u(r+1)}(M^{r+1})}
.$$
This proves the result if $s>\left(1+\frac{d}{2}\right)(r+1).$
\end{proof}

\subsection{  Mixing for Gibbs measures.}
\label{AppMixGibbs}
\begin{proof}[Proof of Proposition \ref{PrGibbsMix}]
The proof consists of three steps.

{\it Step 1.}
By  the same argument as in \cite[Proposition 3.8]{Viana97}, we have that for $\hpsi_1\in \rm{Lip(\mathbb{T})},$ $\hpsi_2\in L^1(\mu),$
\begin{equation}
\label{PWM2Gibbs}
\left| \int \hpsi_1(\hpsi_2 \circ f^n)  d\mu - \int \hpsi_1 d\mu \int \hpsi_2 d\mu  \right| \leq
C \|\hpsi_1\|_{Lip}\|\hpsi_2\|_{L^1}  \brtheta^n, \,n\geq 0.
\end{equation}

{\it Step 2.}
We proceed to show inductively that for each $r>0$ and $\psi_i\in\rm{Lip(\mathbb{T})}$ for $i=1,\dots, r,$
\begin{equation}\label{PWMrGibbs}
\left|\int \left(\prod_{i=1}^r \psi_i \circ f^{k_i}  \right)d\mu-\prod_{i=1}^r\int\psi_i d\mu\right|\leq C\brtheta^m\prod_{i=1}^r||\psi_i||_{Lip},
\end{equation}
where $\DS m=\min_{1\leq i\leq r-1}(k_{i+1}-k_i),$ $k_0=0.$

By invariance of $\mu$ we may assume that $k_1=0.$ Applying \eqref{PWM2Gibbs} with
$\hpsi_1=\psi_1,$ $\DS \hpsi_2=\prod_{j=2}^r \psi_j \circ f^{k_j-k_2} $ we get
$$ \left|\int \left(\prod_{i=1}^r \psi_i \circ f^{k_i}  \right)d\mu-\left(\int \psi_1 d\mu \right)
\left[\int \left(\prod_{i=2}^r \psi_i\circ f^{k_i}  \right) d\mu \right] \right|$$
$$ \leq C \brtheta^m \|\psi_1\|_{Lip} \left\| \left(\prod_{i=2}^r \psi_i\circ f^{k_i}  \right)\right\|_{L_1}
\leq C\brtheta^m\prod_{i=1}^r||\psi_i||_{Lip}.
$$
Applying inductive estimate to
$$ \int \left(\prod_{i=2}^r \psi_i\circ f^{k_i} \right) d\mu $$
we obtain \eqref{PWMrGibbs}.

{\it Step 3.} Applying the same argument as in proof of  Theorem \ref{MultiMixRet}
we get $(EM)_r.$
\end{proof}

\subsection{  Examples of exponentially mixing systems}
There are many results about double (=$2$-fold) exponential mixing. Many examples of those systems
are partially hyperbolic. In particular, they expand an invariant foliation $W^s$ by unstable manifolds.
The next result allows to promote double mixing to $r$ fold mixing.

\begin{Thm}
\label{ThMixShortest}
\rm(\cite[Theorem 2]{Dolgopyat}\rm)\label{thm5}
Suppose that for each subset $D$ in a single unstable leave of bounded geometry\footnote{We refer
the reader to \cite{Dolgopyat} for precise requirements on $D$ since those
requirements are not essential for the
present discussion.} and any H\"older probability density $\rho$ on $D$ we have
$$ \left| \int_D A(f^n x) \rho(x) dx-\int A d\mu\right|\leq C \theta^n \|A\|_{C^s} \|\rho\|_{C^\alpha} $$
for $A\in C^s.$ Then $f$ is $r$-fold exponentially mixing for all $r\geq 2.$
\end{Thm}

Examples of maps satisfying the conditions of Theorem \ref{ThMixShortest}
include expanding maps, volume preserving
Anosov diffeomorphsims \cite{Bow, PP90}, time one maps of contact Anosov flows \cite{L},
mostly contracting systems \cite{C04, D00},
partially hyperbolic translations on
homogeneous spaces \cite{KM96},
and partially hyperbolic automorphisms of nilmanifolds \cite{GS}.

We also note the following fact.

\begin{Thm}
A product of exponentially mixing maps is exponentially mixing.
\end{Thm}

The proof of this theorem is very similar to the proof of Theorem \ref{MultiMixRet} so we leave it to the reader. We also note that instead of direct products one can also consider certain skew products
(so called generalized $T, T^{-1}$ transformations) provided that the skewing function has positive drift.
We refer the reader to \cite{DDKN} for more details.

Another source of exponential mixing is spectral gap for transfer operators (cf. \S \ref{AppMixGibbs}
as well as
\cite{PP90, Viana97}). This allows to handle
non-uniformly hyperbolic systems admitting Young tower with exponential tails \cite{Young98}
as well as piecewise expanding maps \cite{Viana97}.

We note that the maps described in the last paragraph do not fit in the framework of the present
paper due to either lack of smoothness or lack of smooth invariant measure. It is interesting to extend
the result of the paper to cover those systems as well as some slower mixing system and this is a
promising direction for a future work.

\section{Gibbs measures for expanding maps on the circle}
\label{AppLILocDim}
\subsection{Some notation.}
Recall that we assume $P(g)=0,$ so we have
\begin{equation}
\label{ErgSumMu}
\ln \mu \left(B_n(x,\varepsilon)\right) = \sum_{j=0}^{n-1} g(f^j x) + O(1).
\end{equation}
Denote
$$r_n = \sup_{r >0} \{ r \mid B(x,r) \subset B_n(x,\varepsilon) \},
\quad \bar r_n = \inf_{r>0} \{ r \mid B(x,r) \supset B_n(x,\varepsilon) \}.
$$
By bounded distortion property, there exist constants $C_0 > 0$ and $\alpha>0$
such that if $d(f^n y, f^n x) < \varepsilon$ then
$$ \left(C_0 \exp \varepsilon^\alpha\right)^{-1} \leq \frac{|D f^n (y)|}{|D f^n (x)|}
\leq C_0 \exp \varepsilon^\alpha.$$

Recalling \eqref{SRBPot}
$$   \exp\left[\left(\sum_{j=0}^{n-1} f_u (f^j x)\right)-\varepsilon^\alpha\right] \frac{d(x,y)}{C_0} \leq d(f^n x, f^n y) \leq
C_0 \exp\left[\left(\sum_{j=0}^{n-1} f_u (f^j x)\right) +
\varepsilon^\alpha\right]  d(x,y) .$$
Hence
$$ \varepsilon C_0^{-1}  \exp\left[\left(-\sum_{j=0}^{n-1} f_u(f^j x) \right)-\varepsilon^\alpha \right] \leq r_n \leq \bar{r}_n
\leq  \varepsilon C_0  \exp\left[\left(-\sum_{j=0}^{n-1} f_u (f^j x) \right)+\varepsilon^\alpha \right] . $$
It follows that
\begin{equation}
\label{ErgSumR}
\ln r_n = \sum_{j=0}^{n-1} -f_u(f^j x) + O(1), \quad
\ln \bar{r}_n = \sum_{j=0}^{n-1} -f_u (f^j x) + O(1).
\end{equation}

Next define
$$N(r)=\max\left(n: B(x,r)\subset B_n(x, \eps)\right), \quad
\brN(r)=\min\left(n: B(x,r)\supset B_n(x, \eps)\right). $$
Then, similarly to \eqref{ErgSumR} we obtain
\begin{equation}
\label{ErgSumN}
\ln r= \sum_{j=0}^{N(r)-1} -f_u(f^j x) + O(1)= \sum_{j=0}^{\brN(r)-1} -f_u (f^j x) + O(1).
\end{equation}

\subsection{Proof of \eqref{GibbsSqueeze}
and \eqref{Alhfors}.}
\label{SSGibbsReg}
Note that
\begin{equation}
\label{MuB-MuBB}
 \mu(B_{\brN(r)}(x,\eps))\leq\mu(B(x,r))\leq  \mu(B_{N(r)}(x,\eps)).
\end{equation}
Since $f$ is uniformly expansing there is a positive constant $C$ such that for each $x$
$1/C\leq f_u(x)\leq C.$ Accordingly
\begin{equation}
\label{R-LnN}
\frac{N(r)}{C} \leq |\ln r|\leq C N(r),
\quad
\frac{\brN(r)}{C} \leq |\ln r|\leq C \brN(r).
\end{equation}
On the other hand, since $P(g)=0$, \cite[Chapter 3]{PP90} shows that there is a function a H\"older function $\hg(x)$
such that $\hg=g+h-h\circ f$ for a H\"older function $h$ and moreover
$$\sum_{f(y)=x} e^{\hg(y)}=1. $$
In particular, $\hg(y)$ is negative and, since it is continuous,
there are  constants  $\hC_1>\heps>0$ such that for any $x\in \Tor$  we have $\hg(x)\in(-\hC_1,-\heps).$
Using the estimate
$$ \sum_{n=0}^{N-1} g(f^n x)=\sum_{n=0}^{N-1} \hg(f^n x)+O(1) $$
we conclude that for some constant $\hC_2>0$ we have for every $x \in \T,$
\begin{equation}
\label{SumHG}
 -\hC_1 N -\hC_2 \leq\sum_{n=0}^{N-1} g(f^n x)\leq -\heps N+\hC_2
\end{equation}
Combining \eqref{ErgSumMu}, \eqref{MuB-MuBB}, \eqref{R-LnN} and \eqref{SumHG}
we obtain \eqref{GibbsSqueeze}.

Next \eqref{ErgSumN} shows that $N(4r)-\brN(r)=O(1).$ Now \eqref{Alhfors} follows from
\eqref{ErgSumMu} and \eqref{MuB-MuBB}.

\subsection{Proof of Lemma \ref{LmDimFl}(b).} Observe that \eqref{ErgSumMu}
\eqref{ErgSumR} give
$$
\ln \mu \left(B_n(x,\varepsilon)\right) - \dd \ln r_n = \sum_{j=0}^{n-1} \psi(f^j x)+O(1), \quad
\ln \mu \left(B_n(x,\varepsilon)\right)- \dd \ln \bar r_n= \sum_{j=0}^{n-1} \psi(f^j x)+O(1)$$
where $\psi$ is defined by \eqref{PsiPot}.

By Law of Iterated Logarithm \cite{HofbauerKeller82},
 $$
 \limsup_{n\rightarrow \infty} \frac{\sum_{j=0}^{n-1} \psi (f^j x)}{\sqrt{2  n \ln \ln n}} =  \sigma, \quad
 \liminf_{n\rightarrow \infty} \frac{\sum_{j=0}^{n-1}\psi (f^j x)}{\sqrt{ 2 n \ln \ln n}} =  -\sigma.
$$
Since $B(x,r_n) \subset B_{n} (x,\varepsilon)\subset B(x, \bar{r}_n)$
$$  \limsup_{n\rightarrow \infty} \frac{|\ln \mu\left( B(x,\bar r_n)\right)| - \dd |\ln \bar r_n|}{\sqrt{2  n \ln \ln n}} \leq\sigma \leq
\limsup_{n\rightarrow \infty} \frac{|\ln \mu \left(B(x,r_n)\right)|-\dd |\ln r_n|}{\sqrt{ 2 n \ln \ln n}}.
$$
Using \eqref{ErgSumR} again, we conclude that for every sufficiently small $\delta,$
there exists $n(\delta)$  and $k$ independent of $\delta$ and $n(\delta)$ such that
$\bar r_{n+k}\leq \delta \leq r_n. $
Then
$$ \sigma\leq \limsup _{\delta\to 0} \frac{|\ln \mu(B(x,r_{n(\delta)}))|-\dd |\ln r_{n(\delta)}|}{\sqrt{ 2 n(\delta) \ln \ln n(\delta)}}
\leq \limsup_{\delta\to 0} \frac{\left|\ln \mu \left(B(x,\delta)\right)\right|-\dd |\ln \delta|}{\sqrt{ 2 n(\delta) \ln \ln n(\delta)}}$$
$$\leq \limsup _{\delta\to 0} \frac{|\ln \mu (B(x,\bar r_{n(\delta)}))|-\dd |\ln \bar r_{n(\delta)}|}{\sqrt{ 2 n(\delta) \ln \ln n(\delta)}}\leq \sigma.
$$
It follows that all inequalities above are in fact equalities. In particular,
$$ \limsup_{\delta\to 0} \frac{|\ln \mu (B(x,\delta))|-\dd |\ln \delta|}{\sqrt{ 2 n(\delta) \ln \ln n(\delta)}}=\sigma.$$
On the other hand by \eqref{ErgSumR} and the ergodic theorem we see that {\bgreen for $\mu$-a.e. $x\in \T$}, it holds that
$\DS \lim_{n\to\infty} \frac{|\ln r_n|}{n}=\lambda.$ For such $x$ we have
$ \DS \lim_{n\to\infty} \frac{|\ln r_n| (\ln \ln |\ln r_n|)}{n \ln \ln n}=\lambda.$
Since $r_n/C\leq \delta\leq r_n$ we have
$$ \lim_{\delta\to 0} \sqrt{\frac{n(\delta) \ln \ln n(\delta)}{|\ln \delta| (\ln \ln |\ln \delta|)}}=\frac{1}{\sqrt\lambda}.$$
Multiplying the last two displays we obtain for $\mu$-a.e. $x\in \T$
$$ \limsup_{\delta\to 0} \frac{|\ln \mu (B(x,\delta))|-\dd |\ln \delta|}{\sqrt{ 2 |\ln \delta| (|\ln \ln |\ln \delta|)}}=
\frac{\sigma}{\sqrt{\lambda}},$$
and likewise
$$ \liminf_{\delta\to 0} \frac{|\ln \mu (B(x,\delta))|-\dd |\ln \delta|}{\sqrt{ 2 |\ln \delta| (|\ln \ln |\ln \delta|)}}=-
\frac{\sigma}{\sqrt{\lambda}}.$$
This proves part (b) of Lemma \ref{LmDimFl}. \hfill $\Box$

\subsection{Proof of Lemma \ref{LmDimFl}(a).}
Suppose that $\sigma^2=0.$ Since we also have that $\int \psi d\mu=0$
\cite[Proposition 4.12]{PP90} shows that $\psi$ is a coboundary, that is, there exists
a H\"older function $\eta$ such that $\psi(x)=\eta(x)-\eta(fx).$ Thus
$\DS \sum_{k=0}^{n-1} \psi(f^k x)=\eta(x)-\eta(f^n x) $
is uniformly bounded with respect to both $n$ and $x.$ Recalling the definition of $\psi$
we see that in this case
$$ \sum_{k=0}^{n-1} g(f^k x)=-\left[\dd\sum_{k=0}^{n-1} f_u(f^k x)\right]+O(1). $$
Now \eqref{ErgSumMu} and \eqref{ErgSumR} show that $\mu$ is conformal. \hfill $\Box$

\section{Geodesic Flows: Geometry of targets  in the configuration space. Proof of Lemma \ref{lem.mes} and Lemma \ref{lem.sigma}.}
\label{app.spheres}

\subsection{ Geometry of spheres. Proof of Lemma \ref{lem.sigma}.} $ \ $

Denote $\gamma(t)=\phi^t(q,v).$ The Jacobi field of $\gamma$ are defined by the  solution of the linear equation
$$ J''(t) + R(J(t), \gamma'(t)) \gamma'(t)=0,$$
where $ J' = \frac{d}{ d t} J$ and $R(X,Y) Z$ denotes the curvature tensor,
which is equivalent to
$$ (J^i)''(t) + \sum_{j=1}^{n} A_j^i(t) J^j(t) =1, i=1,\ldots, n, $$
where the matrix $A(t)=(A_j^i(t))_{i,j=1,\ldots, n}$ is symmetric. {Since $\cQ$ has negative curvature,
the spectrum of $A(t)$} lies between $ -K_1^2 $ and $-K_2^2$ for some $K_1$ and $K_2.$

{Recall the following fact (see  [Lemma 1.1]\cite{Kni02}).}
\begin{Prop}
\label{PrGeo2Der}
The differential $$D\phi^t(v): T_{\pi v}\cQ\times T_{\pi v}\cQ\rightarrow T_{\pi\phi^t(v)}\cQ\times T_{\pi \phi^t(v)}\cQ$$
is given by $D\phi^t(v)(x,y)=(J(t),J'(t)),$ where $J(0)=x,$ $J'(0)=y.$
\end{Prop}

{We are interested in the case
\begin{equation}
\label{GeoINI}
J(0)=0, \quad \|J'(0)\|=1.
\end{equation}

Now Lemma \ref{lem.sigma} follows
combining Proposition \ref{PrGeo2Der} with
Lemma \ref{Prop1} below.}

\begin{Lem}\label{Prop1}
{If \eqref{GeoINI} holds then for each $t_0$ there is a constant $C>0$}  such that
\begin{equation}
\label{J-JPrime}
\|J'(t)\| \leq C \| J(t) \|  \quad \text{ for }t>t_0.
\end{equation}
\end{Lem}

\begin{proof}
{Denote
	$S(t) =  \langle J(t), J'(t) \rangle$, $N(t)=\|J'(t)\|^2$ and
	${|||J|||^2=\|J\|^2+\|J'\|^2}.$
	Then
	\begin{equation}
 \frac{d}{dt} S(t) =  \|J'(t)\|^2 + \langle J(t), J''(t) \rangle
	=   \|J'(t)\|^2 + \langle J(t), -K(t)J(t) \rangle
	\label{SDer}
		 \geq  C_1 |||J|||^2
		  \end{equation}
for some $C_1>0$. It follows that $S(t)>0$ for $t>0.$
	Once we know that $S(t)$ is positive we can also conclude
	from \eqref{SDer} that $\DS \frac{d}{dt}S(t)>\frac{C_1 S(t)}{2} $, whence
\begin{equation}
\label{SLower1}
S(t) > S(u) e^{C_1(t-u)/2}\quad\text{ for }t>u.
\end{equation}
Next
	$\DS N(t)\geq N(0) e^{-K_2^2 t}=e^{-K_2^2 t}$ which together with \eqref{SDer}
	gives
\begin{equation}
\label{SLower2}	
	S(t)\geq e^{-K_2^2 t} t\quad\text{ for }t\in [0,1].
\end{equation}	
Combining this with 	\eqref{SLower1} we get
\begin{equation}
\label{SLower3}
S(t) > e^{-K_2^2} e^{C_1(t-1)/2}
\quad\text{ for }t>1.
\end{equation}
Combining \eqref{SLower2} and \eqref{SLower3} with a trivial bound
\begin{equation}
\label{JTriv}
N(t)\leq |||J(t)|||\leq N(0) e^{K_2^2 t}=e^{K_2^2 t}
\end{equation}
proves \eqref{J-JPrime} for small $t.$
To prove this estimate for large $t$ we shall use the fact, proven in \cite[Lecture 6]{AS67}
that $J$ can be decomposed as
$J=c_+ J_++c_- J_-,$ where
$$ \max(|c_+|, |c_-|)\leq C_3,\quad |||J_-|||\leq C_4 e^{-K_1 t} $$
and
\begin{equation}
\label{JRic}
J_+=R(t) J_+'(t)
\end{equation}
 where $R$ is a symmetric matrix with spectrum between
$K_1$ and $K_2.$
It follows that
\begin{equation}
\label{J-JPlus}
 |||J(t)|||\leq c_+|||J_+(t) |||+ C_3 C_4 e^{-K_1 t}
 \leq \sqrt{1+K_2^2} \; \|c_+J_+(t) \|+ C_3 C_4 e^{-K_1 t}
 \end{equation}
On the other hand \eqref{SLower3} gives a uniform lower bound
\begin{equation}
\label{J-S}
|||J|||\geq 	2 e^{-K_2^2/2} e^{C_1(t-1)/4}
\end{equation}
Combining
 \eqref{J-JPlus} and \eqref{J-S} we obtain
 $$ \|J(t)\|\geq \|c_+ J^+(t)\|-c_- \|J^-(t)\| \geq \frac{2}{1+K_2^2}  e^{-K_2^2/2} e^{C_1(t-1)/4}
 -2 C_3 C_4 e^{-K_1 t}
 $$
which proves \eqref{J-JPrime} for large $t.$}
\end{proof}

\subsection{{ Volume of the targets} in the configuration space}

\begin{proof}[Proof of Lemma \ref{lem.mes}.]
If $(q,v)\in \hat{B}_\rho(a),$ denote
$$ L(q,v)=L^+(q,v)+L^-(q,v) \text{ where }
L^\pm (q,v)=\sup\{t:\phi^{\pm s}(q,v)\in\hat{B}_\rho(a)\,\text{for}\, 0\leq s\leq t\}.
$$
Then we have the following estimate
\begin{equation*}
\label{Sweep}
\mu\left(\Omega_{a,\rho}\right)=\varepsilon \left(
\int_{\hat{B}_\rho(a)} \frac{1}{L(q,v)}d\mu
\right) \left(1+O(\rho)\right)
\end{equation*}
(see e.g. \cite{Chernov97}). Note that $\mu$ is of the form
$d\mu(q,v)=\frac{d\lambda(q) d\sigma(v)}{\lambda(\cQ)}$ where $\lambda$ is the Riemann volume on
$\cQ$ and $\sigma$ is normalized volume on the $d$ dimensional sphere.
If $\rho$ is small then the integral in parenthesis equals to $\rho^d \gamma (1+O(\rho))$
where
\begin{equation}
\label{GammaIntegral}
 \gamma=\frac{1}{\lambda(\cQ)} \int_{\cB\times \mathbb{S}^d} \frac{1}{\cL(x,v)}dx d\sigma(v)
\end{equation}
where $\cB$ is the unit ball in $\reals^{d+1}$ and $\cL(\cdot)$ is defined similarly $L(\cdot)$
with geodesics in $\cQ$ replaced by geodesics in $\reals^{d+1}.$
Specifically, an elementary plane geometry gives $\cL(x,v)=\sqrt{1-r_{min}^2}$ where
$r_{min}$ is the minimal distance between the  line $x+tv$ and the origin. Thus $r_{min}=r\sin \theta$
where $r$ is the distance from $x$ to $0$,
$\theta$ is the angle between $v$ and the segment from $x$ to $0.$ This proves \eqref{GeoGamma}
with $\gamma$ given by \eqref{GammaIntegral}.
\end{proof}


\begin{thebibliography}{999}
\bibitem{Fr73}  Freedman, D.:
Another note on the Borel-Cantelli Lemma and the strong law with the Poisson approximation as a by-product. {\it Ann. Probability} {\bf 1} (1973), 910--925.
\end{thebibliography}
\end{document}